\documentclass[11pt, reqno]{amsart}
\usepackage{a4wide}
\usepackage{wrapfig}
\usepackage{amsmath}
\usepackage{amsthm}
\usepackage{mathrsfs}
\usepackage{color}
\usepackage{xcolor}
\usepackage{graphicx}
\usepackage{amssymb}
\usepackage{ dsfont }
\usepackage{url}
\usepackage{multicol}
\usepackage{tikz}
\usepackage{amsmath}
\usepackage{fancyhdr}
\usetikzlibrary{matrix}
\usetikzlibrary{arrows}
\usepackage{subfig}
\usepackage{hyperref}
\usepackage{listings}

\usepackage{esint}

\usepackage{enumitem}
\usepackage{varwidth}

\allowdisplaybreaks

\title[LWP for gKdV on the background of a bounded function]{Local well-posedness for the gKdV equation on the background of a bounded function}

\author[J.M. Palacios]{Jos\'e Manuel Palacios}
\address{Institut Denis Poisson, Universit\'e de Tours, Universit\'e d'Orleans, CNRS, Parc Grandmont 37200, Tours, France}
\email{jose.palacios@lmpt.univ-tours.fr}

\newcommand{\be}{\begin{equation}}
\newcommand{\ee}{\end{equation}}
\newcommand{\bp}{\begin{proof}}
\newcommand{\ep}{\end{proof}}
\newcommand{\bel}{\begin{equation}\label}
\newcommand{\eeq}{\end{equation}}
\newcommand{\bea}{\begin{eqnarray}}
\newcommand{\eea}{\end{eqnarray}}
\newcommand{\bee}{\begin{eqnarray*}}
\newcommand{\eee}{\end{eqnarray*}}
\newcommand{\ben}{\begin{enumerate}}
\newcommand{\een}{\end{enumerate}}

\newcommand{\R}{\mathbb{R}}

\newcommand{\N}{\mathbb{N}}
\newcommand{\Z}{\mathbb{Z}}
\newcommand{\T}{\mathbb{T}}

\newcommand{\supp}{\operatorname{supp}}

\newtheorem{thm}{Theorem}[section]

\newtheorem{lem}[thm]{Lemma}
\newtheorem{prop}[thm]{Proposition}

\newtheorem{defn}[thm]{Definition}

\theoremstyle{remark}
\newtheorem{rem}{Remark}[section]

\definecolor{codegreen}{rgb}{0,0.6,0}
\definecolor{codegray}{rgb}{0.5,0.5,0.5}
\definecolor{codepurple}{rgb}{0.58,0,0.82}
\definecolor{backcolour}{rgb}{0.95,0.95,0.92}

\lstdefinestyle{mystyle}{
	backgroundcolor=\color{backcolour},   
	commentstyle=\color{codegreen},
	keywordstyle=\color{magenta},
	numberstyle=\tiny\color{codegray},
	stringstyle=\color{codepurple},
	basicstyle=\footnotesize,
	breakatwhitespace=false,         
	breaklines=true,                 
	captionpos=b,                    
	keepspaces=true,                 
	numbers=left,                    
	numbersep=5pt,                  
	showspaces=false,                
	showstringspaces=false,
	showtabs=false,                  
	tabsize=2
}
\numberwithin{equation}{section}

\usepackage{pgfplots}
\pgfplotsset{compat=newest}



\theoremstyle{definition}

\numberwithin{ej}{section}

\begin{document}





\renewcommand{\sectionmark}[1]{\markright{\thesection.\ #1}}
\renewcommand{\headrulewidth}{0.5pt}
\renewcommand{\footrulewidth}{0.5pt}
\begin{abstract}
We prove the local well-posedness for the generalized Korteweg-de Vries equation in $H^s(\R)$, $s>1/2$, under general assumptions on the nonlinearity $f(x)$, on the background of an $L^\infty_{t,x}$-function $\Psi(t,x)$, with  $\Psi(t,x)$ satisfying some suitable conditions. As a consequence of our estimates, we also obtain the unconditional uniqueness of the solution in $H^s(\R)$. This result not only gives us a framework to solve the gKdV equation around a Kink, for example, but also around a periodic solution, that is, to consider localized non-periodic perturbations of a periodic solution. As a direct corollary, we obtain the unconditional uniqueness of the gKdV equation in $H^s(\R)$ for $s>1/2$. We also prove global existence in the energy space $H^1(\R)$, in the case where the nonlinearity satisfies that $\vert f''(x)\vert\lesssim 1$.
\end{abstract}

\maketitle 

\section{Introduction}

\subsection{The model} The initial value problem for the $k$-Korteweg-de Vries equation ($k$-KdV) \begin{align}\label{kdv}
\begin{cases}\partial_tu+\partial_x^3u\pm u^k\partial_xu=0, \quad t\in\R,\,x\in\R,\,k\in\Z^+,
\\ u(0,x)=u_0(x),
\end{cases}
\end{align}
has been extensively studied in the last five decades and is one of the most famous equations in the context of dispersive PDEs. This family of equations includes the celebrated Korteweg-de Vries (KdV) equation (case $k=1$), which was derived as a model for the unidirectional propagation of nonlinear dispersive long waves \cite{KdV}, and subsequently found in the study of collision-free hydro-magnetic waves \cite{GaMo}. Nowadays, the KdV equation has shown applications to several physical situations, such as for example, in plasma physics, for the study of ion-acoustic waves in cold plasma \cite{BeKa,WaTa}, as well as  some relationships with the Fermi-Pasta-Ulam problem \cite{Ku,Za1,Za2,Za3}. Moreover, some connections with algebraic geometry were given in \cite{DuMaNo} (see also \cite{Mi1} and the references therein). On the other hand, in \cite{Wi} it has been shown that this equation also describes pressure waves in a liquid-gas bubble mixture, as well as waves in elastic rods (see \cite{Na}). We refer to \cite{Mi1} for a more extensive description of all of these (and more) physical applications.

\medskip

In the case where $k=2$ we find another fairly celebrated equation, the so-called modified KdV equation, which also models the propagation of weak nonlinear dispersive waves. In this regard, a large class of hyperbolic models has been reduced to the latter two equations. It is worth to notice that there is a deep relationship
between these two models given by the Miura transformation \cite{Mi2}.

\medskip

These two cases ($k=1,2$) correspond to completely integrable systems, in terms of the existence of a Lax-pair, and both of them have been solved via inverse scattering. An interesting  property of \eqref{kdv} is that these are the only two cases on which this equation corresponds to a completely integrable system (see \cite{GaGrKrMi1,GaGrKrMi2}).

\medskip

One of the most important features of equation \eqref{kdv} is the existence of solitary wave solutions of both types, localized solitary waves and kink solutions. In the completely integrable cases, these solutions correspond to soliton solutions, that is, they preserve their shape and speed after collision with objects of the same type. 

\medskip

In this work we seek to study the initial value problem associated with the following generalization of the $k$-KdV equation \eqref{kdv},  \begin{align}\label{ggkdv_v}
\begin{cases} 
\partial_tv+\partial_x(\partial_x^2v+ f(v))=0,
\\ v(0,x)=\Phi(x),
\end{cases}
\end{align}
where $v=v(t,x)$ stands for a real-valued function, the nonlinearity $f$ is also real-valued, and $t,x\in\R$. Motivated by the study of Kink solutions, here we do not intend to assume any decay of the initial data $\Phi(x)$ but, for the moment, only that $\Phi\in L^\infty(\R)$. Instead, we decompose the solution $v(t,x)$ in the following fashion \begin{align}\label{deco_v_u}
v(t,x)=u(t,x)+\Psi(t,x),
\end{align}
where we assume that $\Psi\in L^\infty(\R^2;\R)$ is a given function (see \eqref{hyp_psi} below for the specific hypotheses on $\Psi$) and we seek for $u(t)\in H^s(\R)$. Then, it is natural to rewrite the above IVP in terms of the Cauchy problem associated with the following generalized Korteweg-de Vries (gKdV) equation
\begin{align}\label{gkdv_v}
\begin{cases}
\partial_tu+\partial_t\Psi+\partial_x\big(\partial_x^2u+\partial_x^2\Psi+f(u+\Psi)\big)=0,
\\ u(0,x)=u_0(x)\in H^s(\R).
\end{cases}
\end{align}
We stress that equation \eqref{gkdv_v} is nothing else that equation \eqref{ggkdv_v} once replacing the decomposition given in \eqref{deco_v_u}. In the case where $f(x)=x^2$ and $\Psi=\Psi(x)$ is a time-independent function belonging to the so-called Zhidkov class \[
\mathcal{Z}(\R):=\left\{\Psi\in\mathcal{D}'(\R): \ \Psi\in L^\infty(\R),\ \Psi'\in H^{\infty}(\R)\right\},
\]
this equations has been used to model the evolution of bores on the surface of a channel, incorporating nonlinear and dispersive effects intrinsic to such propagation \cite{BoRaSc}.

\medskip

\textbf{Important:} In this work we only assume that $f:\R\to\R$ is a real-analytic function satisfying that its Taylor expansion around zero has infinite radius of convergence, that is, there exists a family $\{a_k\}_{k\in\N}\subset\R$ such that, for all $x\in\R$, the nonlinearity $f(x)$ can be represented as \begin{align}\label{hyp_f_nonl}
f(x)=\sum_{k=0}^{\infty} a_kx^k, \quad \hbox{ with } \quad \limsup_{k\to+\infty}\sqrt[k]{\vert a_k\vert}=0.
\end{align}
Notice that any polynomial $p(x)$ satisfies the previous hypothesis, as well as $\exp(x)$, $\sinh(x)$, $\cosh(x)$, $\sin(x)$, $\cos(x)$, $p(\sin(x))$, etc.

\medskip

It is worth to notice that, since equation \eqref{gkdv_v} can be regarded as a perturbation of the gKdV equation \eqref{ggkdv_v} with initial data $v(0,\cdot)\in H^s(\R)$, one might think that, in order to prove local well-posedness for equation \eqref{gkdv_v}, it is reasonable to proceed by using the contraction principle as in \cite{KePoVe2}. However, it seems that this does not even hold in the case where $f(x)=x^2$, due to the occurrence of the term $\Psi\partial_xu$ since $\Psi$ is not integrable, what makes this problem more involved even for the KdV case.

\medskip

As mentioned before, one of our main motivations comes from studying Kink solutions. For instance, we can consider the defocusing modified Korteweg-de Vried (mKdV) equation, that is $f(u)=-u^3$, as well as the Gardner equation, that is $f(u)=u^2-\beta u^3$. Both equations are well-known to have Kink solutions given by  (respectively): \begin{align*}
\Psi_{\mathrm{mKdV},c}(t,x)&=\pm\sqrt{c}\tanh\left(\sqrt{\tfrac{c}{2}}(x+ct)\right), \quad c>0, \quad \hbox{and}
\\ \Psi_{\mathrm{Gardner},c}(t,x)&=\dfrac{1}{3\beta}\pm \dfrac{1}{\sqrt{\beta}}\Phi_{\mathrm{mKdV},c}(t,x-\tfrac{t}{3\beta}), \quad \beta>0.
\end{align*}
Moreover, at the same time we also seek to give a framework to study localized non-periodic perturbations of periodic solutions for the generalized model \eqref{gkdv_v}, such as for example, the famous cnoidal and dnoidal wave solutions of the KdV and the mKdV equations (respectively)
\begin{align}\label{sncn}
\Psi_{\mathrm{cn},c}(t,x)&:=\alpha+\beta\mathrm{cn}^2\big(\gamma(x-ct),\kappa\big), \quad \Psi_{\mathrm{dn},c}(t,x):=\beta\mathrm{dn}\big(\gamma(x-ct),\kappa\big),
\end{align}
with $c>0$ and $(\alpha,\beta,\gamma,\kappa)\in \R^4$ satisfying some suitable conditions, where $\mathrm{cn}(\cdot,\cdot)$ and $\mathrm{dn}(\cdot,\cdot)$ stand for the Jacobi elliptic cnoidal and dnoidal functions respectively.

\medskip

It important to mention that, the gKdV equation \eqref{ggkdv_v} enjoys (at least formally) several conservation laws, such as the conservations of the mean, the mass and the energy, which are given by (respectively) \begin{align*}
I_1(v(t))&:=\int_\R v(t,x)dx=I_1(v_0),
\\ I_2(v(t))&:=\int_\R v^2(t,x)dx=I_2(v_0),
\\ I_3(v(t))&:=\int_\R \Big(v_x^2(t,x)-F\big(v(t,x)\big)\Big)dx=I_3(v_0),
\end{align*}
where $F(\cdot)$ stands for a primitive of $f(\cdot)$. However, due to the presence of $\Psi(t,x)$, non of these quantities are conserved for solutions of equation \eqref{gkdv_v}. Although, a suitable modification of the energy functional $I_3$ shall play a key role in proving global well-posedness in $H^1(\R)$ when $f(x)$ grows at most as $x^2$ (see Theorem \ref{MT_gwp} below).

\medskip

\textbf{Important:} In the sequel we shall always assume that the given function $\Psi(t,x)$ satisfies the following hypotheses:
\begin{align}\label{hyp_psi}
\begin{cases}
\partial_t\Psi\in L^\infty(\R^2), 
\\ \Psi\in L^\infty(\R,W^{s+1^+,\infty}(\R)), 
\\ (\partial_t\Psi+\partial_x^3\Psi+\partial_xf(\Psi))\in L^\infty(\R,H^{s^+}(\R)).
\end{cases}
\end{align}

\begin{rem} Note that any function $\Psi=\Psi(x)\in L^\infty(\R)$ such that $\Psi'\in H^\infty(\R)$, for example $\Psi$ being a Kink, satisfies all the conditions in \eqref{hyp_psi}. Hence, equation \eqref{gkdv_v} together with conditions \eqref{hyp_psi} contain as particular cases all the frameworks considered in  \cite{BoRaSc,Ga,IoLiSc,Zhi}. However, we do not require that $\Psi(t,x)$ has well-defined limits at $\pm\infty$ as in those previous works. For instance, if $\Psi=\Psi(t,x)$ solves the gKdV equation \eqref{ggkdv_v}, then the latter expression in \eqref{hyp_psi} is identically zero, and hence the third hypothesis is immediately satisfied. In particular, we can consider $\Psi(t,x)$ being a periodic solution of the gKdV equation. Nevertheless, $\Psi(t,x)$ does not need to be a solution, neither to have a small $H^{s+}$-norm once replaced in the equation. For example, we can solve equation \eqref{gkdv_v} with $\Psi$ being \[
\Psi(t,x)=1+4\tanh(x+t)+\cos\left(\log(1+x^2+t^2)\right).
\]
Notice that this function does not have symmetries (neither odd nor even), nor well-defined limits at $\pm\infty$, also, none of its derivatives has exponential decay. Clearly, it does not solve the equation either, whereas it satisfies all the conditions in \eqref{hyp_psi}.
\end{rem}

\smallskip

\subsection{Unconditional uniqueness}

The generalized Korteweg-de Vries equation \eqref{ggkdv_v} has been proven to be locally well-posed (LWP) for regular localized initial data in \cite{AbBoFeSa,BoSm,Ka,Sa}. Since then, considerable effort has been devoted to the understand the Cauchy problem \eqref{kdv} with rough data. In the seminal work of Kenig, Ponce and Vega, LWP for equation \eqref{kdv} has been established in $H^s$-spaces, for all $k\in \Z_+$,  with $s$ moving in a range that depends on $k$. In the case where $k\geq 4$ these results are sharp, in the sense that they reach the critical index given by the scaling invariance \cite{KePoVe2}. This proof relies in Strichartz estimates, along with local smoothing effect and maximal estimates. Then, a normed functional space is constructed based on these estimates, which allows them to solve \eqref{kdv} via a fixed point argument. The solutions obtained in this way are obviously unique in such a resolution space. However, as explained in \cite{MoVe}, the question of whether this uniqueness holds for solutions that do not belong to these resolution spaces turns out to be far from trivial at this level of regularity. This type of question was first raised by Kato in \cite{Ka2} in the context of Schr\"odinger equations. We refer to such uniqueness in $L^\infty((0, T), H^s(\R))$, without intersecting with any other auxiliary functional space
as \emph{unconditional uniqueness}. This type of uniqueness has been shown to be useful, for example, to pass to the limit in perturbative analyzes, when one of the coefficients of the equation tends to zero (see  \cite{Mo} for instance). 

\smallskip

\subsection{Main results}
In the remainder of this work we focus on studying the Cauchy problem associated with \eqref{gkdv_v}. Before going further, let us give a precise definition of what we mean by a solution. 

\begin{defn}
Let $T>0$ and $s>1/2$, both being fixed. Consider $u\in L^\infty((0,T),H^s(\R))$. We say that $u(t,x)$ is a solution to \eqref{gkdv_v} emanating from the initial data $u_0\in H^s(\R)$ if $u(t,x)$ satisfies \eqref{gkdv_v} in the distributional sense, that is, for any test function $\phi\in C^\infty_0((-T,T)\times \R)$ we have \begin{align}\label{distributional_sol}
&\int_0^{\infty}\int_\R\Big[\big(\phi_t+\phi_{xxx}\big)u+\phi_x\big(f(u+\Psi)-f(\Psi)\big)\Big] dxdt+\int_\R\phi(0,x)u_0(x)dx
\\ & \qquad =-\int_0^{\infty}\int_\R\Big[\big(\phi_t+\phi_{xxx}\big)\Psi+\phi_xf(\Psi)\Big] dxdt-\int_\R \phi(0,x)\Psi(0,x)dx.\nonumber
\end{align}
\end{defn}

\begin{rem}\label{rem_spaces_ic_duhamel}
Notice that, for $u\in L^\infty((0,T),H^s(\R))$ with $s>1/2$, we have that $u^p$ is well defined for all $p\in\N$, which along with \eqref{hyp_f_nonl} and  \eqref{hyp_psi} implies that \[
f(u+\Psi)-f(\Psi)\in L^\infty\big((0,T),H^{s}(\R)\big).
\]
Therefore, relation \eqref{distributional_sol} and the hypotheses in \eqref{hyp_psi} forces that $u_t\in L^\infty((0,T),H^{s-3}(\R))$, and hence \eqref{gkdv_v} is satisfied in $L^\infty((0,T),H^{s-3}(\R))$. In particular, we infer that \begin{align}\label{continuity_H_s_3}
u\in C([0,T],H^{s-3}(\R)).
\end{align}
Moreover, from hypotheses \eqref{hyp_psi} we infer that $\Psi\in C([0,T],L^\infty(\R))$, which in turn, together with \eqref{distributional_sol} and \eqref{continuity_H_s_3}, forces the initial condition $u(0)=u_0$. On the other hand, notice that we also have that $u\in C_w([0,T],H^s(\R))$, and that $u\in C([0,T],H^\theta(\R))$ for all $\theta<s$. Finally, we stress that the above relations also ensure that $u$ satisfies the Duhamel formula associated with \eqref{gkdv_v} in $C([0,T],H^{s-3}(\R))$.
\end{rem}

Our first main result give us the (unconditional) local well-posedness for \eqref{gkdv_v}.

\begin{thm}[Local well-posedness]\label{MT1}
Let $s>1/2$ fixed. Consider $f:\R\to\R$ to be any real-analytic function such that its Taylor expansion around zero has infinite radius of convergence. Consider also $\Psi(t,x)$ satisfying the conditions in \eqref{hyp_psi}. Then, for any $u_0\in H^s(\R)$ there exists $T=T(\Vert u_0\Vert_{H^s})>0$ and a solution to the IVP \eqref{gkdv_v} such that \[
u\in C([0,T],H^s(\R))\cap L^2_TW^{s^-,\infty}_x\cap X^{s-1,1}_T. 
\]
Furthermore, the solution is unique in the class \[
u\in L^\infty((0,T),H^s(\R)).
\]
Moreover, the data-to-solution map $\Phi: u_0\mapsto u$ is continuous from $H^s(\R)$ into $C([0,T],H^s(\R))$.
\end{thm}

\begin{rem}
We refer to the next section for a definition of Bourgain spaces $X^{s,b}$ and their corresponding time-restricted versions $X^{s,b}_T$.
\end{rem}

\begin{rem}
Notice that the previous Theorem allows us both to take $f(x)$ being any polynomial but also to consider $f(x)=e^x$. In particular, if $f(x)=x^2$ or $f(x)=x^3$, the previous theorem allows us to take $\Psi(t,x)$ being, for instance, a periodic solution such as the cnoidal or dnoidal wave solutions (or any other traveling wave solution) given in \eqref{sncn}, respectively.
\end{rem}

As a direct corollary of the previous theorem, by considering $\Psi\equiv0$, we infer the unconditional uniqueness for the gKdV equation \eqref{ggkdv_v}, for initial data $v_0\in H^s(\R)$ with $s>1/2$. \begin{thm}
The Cauchy problem associated with \eqref{ggkdv_v} is unconditionally locally well-posed in $H^s(\R)$ for $s>1/2$.
\end{thm}

Finally, under some extra conditions on the growth of $f(x)$, we prove global well-posedness for equation \eqref{gkdv_v}. 

\begin{thm}[GWP in $H^1(\R)$]\label{MT_gwp}
Assume further that $f:\R\to\R$ satisfies that \[
\vert f''(x)\vert\lesssim 1,\qquad \forall x\in\R.
\]
If the initial data $u_0\in H^s(\R)$, with $s\geq 1$, then the local solution $u(t)$ provided by Theorem \ref{MT1} can be extended for any $T>0$.
\end{thm}

\begin{rem} Notice that the previous theorem give us the GWP, in particular, for $f(x)=x^2$ but also for $f(x)=\sin(x)$ or $f(x)=\cos(x)$ as nonlinearities.
\end{rem}

\begin{rem}
We stress that Theorems \ref{MT1} and \ref{MT_gwp} give us the local and global well-posedness for $H^s(\R)$-perturbations, $s> 1/2$ and $s\geq 1$ respectively, of regular periodic solutions of the KdV equation, in particular, for $H^s(\R)$-perturbations of periodic traveling waves solutions. 
\end{rem}

From the above results we are able to deduce local well-posedness for  equation \eqref{ggkdv_v} on Zhidkov spaces. To this end, we introduce $\mathcal{Z}^s(\R)$ as the function space given by \[
\mathcal{Z}^s(\R):=\{\Psi\in\mathcal{D}'(\R):\, \Psi\in L^\infty(\R),\,\Psi'\in H^{s-1}(\R)\},
\]
endowed with the natural norm $\Vert\Psi\Vert_{\mathcal{Z}^s}:=\Vert \Psi\Vert_{L^\infty}+\Vert \Psi'\Vert_{H^{s-1}}$.
\begin{thm}\label{MT_Zhidkov}
Let $s>1/2$. Consider $f:\R\to\R$ to be any real-analytic function such that its Taylor expansion around zero has infinite radius of convergence. Then, for any $v_0\in \mathcal{Z}^s(\R)$ there exists $T=T(\Vert v_0\Vert_{\mathcal{Z}^s})>0$ and a solution to the IVP  \eqref{ggkdv_v}, such that \[
v\in C([0,T],\mathcal{Z}^s(\R))\cap L^2_TW^{s^-,\infty}_x \quad \hbox{and} \quad v-v_0\in C([0,T],H^s(\R)). 
\]
Furthermore, the solution is unique in the class \[
v(t)-v_0\in L^\infty((0,T),H^s(\R)).
\]
Also, the data-to-solution map $\Phi: v_0\mapsto v(t)$ is continuous from $\mathcal{Z}^s(\R)$ into $C([0,T],\mathcal{Z}^s(\R))$. Moreover, if $f(x)$ satisfies that $\vert f''(x)\vert\lesssim1$ for all $x\in\R$, then the solution $v(t)$ can be extended for all times $T>0$.
\end{thm}

Our method of proof relies in the improvements of the energy method, recently developed in \cite{MoVe,MoPiVe,MoPiVe3}, along with symmetrization arguments previously used in \cite{CoKeStTaTa}, for example. However, due to the presence of $\Psi(t,x)$ (which breaks the symmetry) and the general nonlinearity, the present analysis shall be more involved that the previous cases.

\medskip

At this point it is important to remark that local well-posedness in such a general framework has never been established for equation \eqref{gkdv_v} before. However, the smooth case is by no mean a new result, but rather a suitable rewrite of the previous proofs (see \cite{AbBoFeSa,IoLiSc} for example). For the sake of completeness, we prefer to state this theorem and give a brief proof of its most important parts though (see Section \ref{section_MT_smooth}). In fact, the key point to prove LWP for equation \eqref{ggkdv_v} in the smooth case are the commutator estimates which, in the case of equation  \eqref{gkdv_v}, can be performed with almost no changes (with respect to \cite{AbBoFeSa,IoLiSc}) thanks to our hypothesis on $\Psi$.

\medskip

\subsection{Previous literature}

Concerning the local well-posedness of equation \eqref{kdv} there exists a  bast literature for each case of $k\in\Z_+$. In the case where $k=1$, Kenig, Ponce and Vega \cite{KePoVe3} showed the LWP in $H^s(\R)$ for $s\geq -3/4$ via the contraction principle. This result is sharp \cite{ChCoTa}, in the sense that the flow map fails to be uniformly continuous in $H^s(\R)$ for $s<-3/4$. Then, global well-posedness (GWP) was proved for $s>-3/4$ by using the $I$-method  \cite{CoKeStTaTa}. Without asking for the uniform continuity but just continuity of the data-to-solution map, by using the complete integrability of the equation, Killip and Visan \cite{KiVi} showed GWP in $H^{-1}(\R)$, which is the lowest index one can obtain due to the result of Molinet \cite{Mo2} which ensures that this map cannot be continuous below $H^{-1}(\R)$. On the other hand, in  \cite{Zho} Zhou demonstrated the unconditional uniqueness in $L^2(\R)$. In the periodic case, LWP was proved in $H^s(\T)$ for $s\geq -1/2$ by Kenig, Ponce and Vega \cite{KePoVe3}. These local-in-time solutions are
also shown to exist on an arbitrary time interval. Moreover, the unconditional uniqueness in $L^2(\T)$ was established in \cite{BaIlTi}. 

\medskip

Concerning the mKdV case, that is when $k=2$, the result of Kenig-Ponce-Vega ensures the LWP for $s\geq 1/4$ on the line \cite{KePoVe2}. It has been proven that this result is sharp in the sense that the flow map fails to be uniformly continuous in $H^s(\R)$ as soon as $s<1/4$, for both, the focusing mKdV \cite{KePoVe4} and the defocusing one \cite{ChCoTa}. Then, GWP was shown for $s>1/4$ in \cite{CoKeStTaTa} by using the $I$-method (see also \cite{FoLiPo}). Moreover, unconditional uniqueness in $H^s(\R)$ for $s>1/3$ was established by Molinet et al. in \cite{MoPiVe}, and recently improved for $s>1/4$ by Kwon et al. in \cite{KwOhYo}. On the other hand, in the periodic case, unconditional LWP for $s\geq 1/3$ was proved by Molinet et al. in \cite{MoPiVe2}, by using the improved energy method developed in \cite{MoVe} together with the construction of modified energies (see also \cite{NaTaTs}). Furthermore, global well-posedness has been shown in $H^s(\T)$ for $s\geq 1/2$ in \cite{CoKeStTaTa}. 

\medskip 

In the case where $k=3$, we refer to \cite{Gr} for the LWP in $H^s(\R)$ for $s\geq -1/6$, and to \cite{KePoVe2} for the local well-posedness  in the case where $k\geq 4$, up to the critical index given by the scaling invariance (inclusive). These results do not refer about the unconditional uniqueness.

\medskip

Regarding equation \eqref{gkdv_v}, as far as we know this equation has never been treated in such a general framework, and hence there is no abundant specific literature for it. However, in the case of the KdV equation, i.e. $f(x)=x^2$, with $\Psi=\Psi(x)$ being a time-independent function belonging to the Zhidkov class, we find the result of Iorio et al. \cite{IoLiSc} for regular data (see also \cite{BoRaSc}). To the best of our knowledge, the best result to date (in the previously mentioned framework) is given by Gallo \cite{Ga} where LWP was established for the KdV case for $s>1$ under the same hypothesis on $\Psi(x)$ as in the work of Iorio et al \cite{IoLiSc}. Note that Theorem \ref{MT_Zhidkov} extend both results \cite{Ga,IoLiSc} to the whole range $s\in (1/2,1]$ and also provides the GWP in the case $s\geq 1$. On the other hand, in the case of general nonlinearity $f(x)$, under some extra conditions concerning the values of $\Psi(x)$ at $\pm \infty$ and the value of the integral of $f(x)$ on the region $[\Psi(-\infty),\Psi(+\infty)]$, Zhidkov \cite{Zhi} established local well-posedness for data in $H^2(\R)$. In the same work, he also proved the $H^1(\R)$ orbital stability of Kinks of equation \eqref{gkdv_v} for $H^2$ solutions. Then, by using this stability property, he showed the global existence of $H^2$ solutions for small $H^1$-perturbations of such Kinks. In order to prove these results, Zhidkov assumed, among other hypotheses, that $\Psi'(x)>0$ for all $x\in\R$, and that $\Psi(x)$ converges exponentially fast to its limits at $\pm\infty$. As we already mentioned, Theorem \ref{MT1} contains (and improves) the results in \cite{BoRaSc,Ga,IoLiSc,Zhi}. In particular, notice that Theorem \ref{MT_Zhidkov} allows us to extend the existence and the stability result of Zhidkov by only considering $H^1$-solutions which are $H^1$-close to those Kinks. Finally, we point out that, in the case $ f(x)=x^2$,  Theorem \ref{MT1} is related to the existence problem for the KdV equation with variable coefficients, which has recently been treated by a similar approach in \cite{MoTaZa}.

\medskip


\subsection{Organization of this paper}

This paper is organized as follows. In Section \ref{preliminaries} we introduce all the notations that we shall use in the sequel, and then we state a series of preliminary results needed in our analysis. In Section \ref{energy_section} we prove the main apriori energy estimates for solutions and for the difference of solutions. In Section \ref{section_MT1} we establish Theorem \ref{MT1}. Then, in Section \ref{section_MT_smooth} we sketch the proof of the LWP in the smooth case (see Theorem \ref{MT2} below). Finally, in Section \ref{section_MT_gwp} we prove the global well-posedness result, i.e. Theorem \ref{MT_gwp}.

\medskip

\section{Preliminaries}\label{preliminaries}

\subsection{Basic notations}

For any pair of positive numbers $a$ and $b$, the notation $a\lesssim b$ means that there exists a positive constant $c$ such that $a\leq cb$. We also denote by $a\sim b$ when $a\lesssim b$ and $b\lesssim a$. Moreover, for $\alpha\in\R$, we denote by $\alpha^+$, respectively $\alpha^-$,  a number slightly greater,
respectively lesser, than $\alpha$. Furthermore, we shall occasionally use the notation $F(x)$ to denote a primitive of the nonlinearity $f(x)$, that is,  $F(s)=\int_0^s f(s')ds'$.

\medskip

Now, for $u(t,x)\in \mathcal{S}'(\R^2)$, $\mathcal{F}u=\hat{u}$ shall denote its space Fourier transform,
whereas $\mathcal{F}_tu$, respectively $\mathcal{F}_{t,x}u$, shall denote its time Fourier transform,
respectively space-time Fourier transform. Additionally, for $s\in\R$, we introduce the Bessel and Riesz potentials of order $-s$, namely $J^s_x$ and $D^s_x$, by (respectively) \[
J^s_xu := \mathcal{F}^{-1}\big((1+\vert\xi\vert^2)^{s/2}\mathcal{F}u\big) \quad \hbox{ and }\quad D^s_xu:= \mathcal{F}^{-1}\big(\vert\xi\vert^s\mathcal{F}u\big).
\]
We also denote by $U(t)$ the unitary group associated with the linear part of \eqref{kdv}, that is, the Airy group, \[
U(t)g:=e^{-\partial_x^3}g=\mathcal{F}^{-1}\big(e^{it\xi^3}\mathcal{F}g\big).
\]
On the other hand, throughout this work we consider a fixed smooth cutoff function $\eta$ satisfying \begin{align}\label{def_eta}
\eta\in C^\infty_0(\R),\quad 0\leq \eta\leq 1, \quad \eta\big\vert_{[-1,1]}=1 \quad \hbox{and} \quad \supp\eta\subset[-2,2].
\end{align}
We define $\phi(\xi):=\eta(\xi)-\eta(2\xi)$ and, for $\ell\in\Z$, we denote by $\phi_{2^\ell}$ the function given by \[
\phi_{2^\ell}(\xi):=\phi\big(2^{-\ell}\xi\big).
\]
Additionally, we shall denote by $\psi_{2^\ell}$ the function given by \[
\psi_{2^\ell}(\tau,\xi):=\phi_{2^\ell}\big(\tau-\xi^3\big) \quad \hbox{for }\, \ell\in\N\setminus\{0\} \quad \,\hbox{and} \,\quad \psi_1(\tau,\xi):=\eta(\tau-\xi^3).
\]
Any summations over capitalized variables such as $N$, $L$, $K$ or $M$ are presumed to be dyadic. Unless stated otherwise, we work with homogeneous dyadic decomposition for the space-frequency and time-frequency variables, and nonhomogeneous decompositions for modulation variables, i.e., these variables range over numbers of the form $\{2^\ell:\,\ell\in \Z\}$ and $\{2^\ell:\, \ell\in\N\}$, respectively. We denote these sets by $\mathbb{D}$ and $\mathbb{D}_\mathrm{nh}$ respectively. Then, with the previous notations and definitions, we have that \[
\sum_{N\in \mathbb{D}}\phi_{N}(\xi)=1 \, \hbox{ for all }\, \xi\in\R\setminus\{0\} \quad \hbox{ and } \quad \supp\phi_{N}\subset\{\tfrac{1}{2}N\leq \vert\xi\vert\leq 2N\}.
\]
In the same fashion, it also follows that \[
\sum_{L\in\mathbb{D}_{\mathrm{nh}}} \psi_{L}(\tau,\xi)=1 \, \hbox{ for all } \, (\tau,\xi)\in\R^2.
\]
We define the Littlewood-Paley multipliers by the following identities \begin{align}\label{def_pr_P_Q}
P_Nu:=\mathcal{F}_x^{-1}\big(\phi_N\mathcal{F}u\big), \quad R_Ku:=\mathcal{F}_t^{-1}\big(\phi_K\mathcal{F}_tu\big)  \quad Q_Lu:=\mathcal{F}_{t,x}^{-1}\big(\psi_L\mathcal{F}_{t,x}u\big).
\end{align}
With these definitions at hand, we introduce the operators \[
P_{\geq N}:=\sum_{M\geq N}P_M, \quad P_{\leq N}:=\sum_{M\leq N}P_M, \quad Q_{\geq L}:=\sum_{\tilde{L}\geq L}Q_{\tilde{L}}, \quad Q_{\leq \tilde{L}}:=\sum_{\tilde{L}\leq L}Q_{\tilde{L}}.
\]
In addition, we borrow some notations from \cite{Ta}. For $k\geq 2$ natural number and $\xi\in\R$, we denote by $\Gamma^k(\xi)$ the $(k-1)$-dimensional affine-hyperplane of $\R^k$ given by \[
\Gamma^k(\xi):=\{(\xi_1,...,\xi_k)\in\R^k:\, \xi_1+...+\xi_k=\xi\},
\]
endowed with the natural measure \[
\int_{\Gamma^k(\xi)}F(\xi_1,...,\xi_k)d\Gamma^k(\xi):=\int_{\R^k}F(\xi_1,...,\xi_{k-1},\xi-\xi_1-...-\xi_{k-1})d\xi_1...d\xi_{k-1},
\]
for any function $F:\Gamma^k(\xi)\to \mathbb{C}$. Moreover, when $\xi=0$ we shall simply denote by $\Gamma^k=\Gamma^k(0)$ with the obvious modifications.

\medskip 

To finish this first subsection, we introduce the notation for the pseudoproduct operator that we shall repeatedly use in the sequel. Let $\chi$ be a (possibly complex-valued) measurable bounded function on $\R^2$. We define the operator $\Pi=\Pi_\chi$ on $\mathcal{S}(\R)^2$ by the expression  \begin{align}\label{pseudo_def}
\mathcal{F}\big(\Pi(f,g)\big)(\xi):=\int_\R \hat{f}(\xi_1)\hat{g}(\xi-\xi_1)\chi(\xi,\xi_1)d\xi_1.
\end{align}
This bilinear operator behaves as a product operator in the sense that it satisfies the following properties
\begin{align*}
\Pi(f,g)=fg \, \hbox{ when } \, \chi\equiv 1, \quad \hbox{ and }\quad \int \Pi_\chi(f,g)h=\int f\Pi_{\chi_1}(g,h)=\int \Pi_{\chi_2}(f,h)g,
\end{align*}
where $\chi_1(\xi,\xi_1)=\overline{\chi}(\xi_1,\xi)$ and $\chi_2(\xi,\xi_1)=\overline{\chi}(\xi-\xi_1,\xi)$ for all trio of functions $f,g,h\in\mathcal{S}(\R)$. Throughout this paper we shall also use pseudoproduct operators with $k$-entries, which are defined as in \eqref{pseudo_def} with the obvious modifications.

\smallskip

\subsection{Function spaces}\label{function_spaces_section}

For $s,b\in\R$ we define the Bourgain space $X^{s,b}$ associated with the linear part of \eqref{kdv} as the completion of the Schwartz space $\mathcal{S}(\R^2)$ under the norm \[
\Vert u\Vert_{X^{s,b}}^2:=\int_{\R^2}(1+\vert \tau-\xi^3\vert )^{2b}(1+\vert\xi\vert)^{2s}\vert \mathcal{F}_{t,x}[u](\tau,\xi)\vert^2d\xi d\tau.
\]
We recall that these spaces satisfy \[
\Vert u\Vert_{X^{s,b}}\sim \Vert U(-t)u\Vert_{H^{s,b}_{t,x}}, \quad \hbox{ where } \quad \Vert u\Vert_{H^{s,b}_{t,x}}:=\big\Vert J^s_xJ^b_t u\big\Vert_{L^2_{t,x}} .
\]
Additionally, we define the frequency-enveloped spaces associated with $H^s(\R)$ as follows: Let $s\in\R$ and $\kappa>1$ fixed. Consider a sequence $\{\omega_N\}_{N\in\mathbb{D}}$ of positive real numbers satisfying that $\omega_N\leq \omega_{2N}\leq 2^\varepsilon\omega_{N}$, for some $\varepsilon>0$ such that $\varepsilon<\min\{\delta_1,\delta_2\}$, where $\delta_1$ and $\delta_2$ are the small numbers  associated with the choices we make in the hypotheses \[
\Psi\in L^\infty_tW^{s+1^+,\infty}_x \quad \hbox{and} \quad (\partial_t\Psi+\partial_x^3\Psi+\partial_xf(\Psi))\in L^\infty_tH^{s^+}_x.
\]
In other words, if we suppose $\Psi\in L^\infty_tW^{s+1+\delta_1,\infty}_x$ and $(\partial_t\Psi+\partial_x^3\Psi+\partial_xf(\Psi))\in L^\infty_tH^{s+\delta_2}_x$, for some $\delta_1,\delta_2>0$ small, then we assume, in particular, that $\omega_N$ satisfies \[
\dfrac{\omega_N}{N^{\delta_*}}\xrightarrow{N\to+\infty}0, \quad \delta_*:=\min\{\delta_1,\delta_2\}.
\]
Furthermore, we also assume that $\omega_N\to1$ as $N\to0$. Then, we define the space $H^s_\omega(\R)$ associated with $\{\omega_N\}$ as the completion of the Schwarz space $\mathcal{S}(\R)$ under the norm \[
\Vert f\Vert_{H^s_\omega}^2:=\sum_{N}\omega_N^2\Vert P_Nf\Vert_{H^s}^2 \sim \sum_{N}\omega_N^2N^{2s}\Vert P_Nf\Vert_{L^2}^2.
\]
Of course, by definition we have $H^s_\omega(\R) \subseteq H^s(\R)$. Moreover, if $\omega_N=1$ for all $N\in \mathbb{D}$, then $H^s_\omega=H^s$. The goal we seek by introducing these  frequency-enveloped spaces is to be able to prove the continuity part in Theorem \ref{MT1}.

\medskip

Finally, we define the restriction-in-time version of all of the above spaces. Let $T>0$ fixed, and consider $F$ to be any normed space of space-time functions. We define its restriction in time version $F_T$ as the space of functions $u:[0,T]\times \R\to\R$ satisfying \[
\Vert u\Vert_{F_T}:=\inf\big\{\Vert \tilde{u}\Vert_{F}: \, \tilde{u}:\R\times\R\to\R, \hbox{ with } \tilde{u}=u \hbox{ on } [0,T]\times \R\big\}<+\infty.
\]

\smallskip

\subsection{Extension operator}

In this subsection we introduce the extension operator that we shall use in the sequel. We borrow this definition from \cite{MoPiVe}. The key property of this operator is that it defines a bounded operator from $L^\infty_TH^s_\omega \cap X^{s-1,1}_T\cap L^2_TW^{r,\infty}_x$ into $L^\infty_tH^s_\omega \cap X^{s-1,1}\cap L^2_tW^{r,\infty}_x$ with $r<s$.

\begin{defn}
Let $T\in(0,2)$ and $u:[0,T]\times \R\to\R$ given. We define the extension operator $\rho_T$ by the following identity \begin{align}\label{rho_def}
\rho_T[u](t):=U(t)\eta(t)U(-\mu_T(t))u(\mu_T(t)),
\end{align}
where $\eta$ corresponds to the function given in \eqref{def_eta} and $\mu_T$ is the following continuous function
\[
\mu_T(t):=\begin{cases}
0 & \hbox{if }\ t\leq0,
\\ t & \hbox{if }\ t\in(0,T),
\\ T & \hbox{if }\ t\geq T.
\end{cases}
\]
\end{defn}

\begin{rem}
Notice that, directly from the definition, we have $\rho_T[u](t,x)=u(t,x)$ for all $(t,x)\in [0,T]\times \R$.
\end{rem}

The next lemma give us the main properties of this operator (see \cite{MoPiVe}).
\begin{lem}
Let $T\in(0,2)$ fixed. Consider $(s,r,\theta,b)\in\R^4$ satisfying $r<s$ and $b\in (1/2,1]$. Then, the following holds: \[
\rho_T: L^\infty_T H^s_\omega\cap X^{\theta,b}_T\cap L^2_TW^{r,\infty}_x\mapsto L^\infty_tH^s_\omega \cap X^{\theta,b}\cap L^2_tW^{r,\infty}_x.  
\]
In other words, we have the following inequality  \[
\Vert \rho_T[u]\Vert_{L^\infty_tH^s_\omega}+\Vert \rho_T[u]\Vert_{X^{\theta,b}}+\Vert \rho_T[u]\Vert_{L^2_tW^{r,\infty}_x}\lesssim \Vert u\Vert_{L^\infty_TH^s_\omega}+\Vert u\Vert_{X^{\theta,b}_T}+\Vert u\Vert_{L^2_TW^{r,\infty}_x}.
\]
Moreover, the implicit constant involved in the latter inequality can be chosen independent of $(T,s,r,\theta,b)$.
\end{lem}

\smallskip

\subsection{Resolution space}

From now on, for any $s\in\R$ and any sequence $\{\omega_N\}_{N\in\mathbb{D}}$ satisfying the hypotheses in Section \ref{function_spaces_section}, we define the resolution space $\mathcal{M}^s_\omega$ being \[
\mathcal{M}^{s}_\omega:=L^\infty_tH^s_\omega\cap X^{s-1,1},
\]
endowed with the natural norm
\begin{align}\label{mst_norm_def}
\Vert u\Vert_{\mathcal{M}^s_\omega}:=\Vert u\Vert _{L^\infty_t H^s_\omega}+\Vert u\Vert_{X^{s-1,1}}.
\end{align}
When $\omega_N\equiv 1$ we simply write $\mathcal{M}^s=\mathcal{M}^s_\omega$. Before going further we recall the following basic lemma concerning Sobolev spaces. 
\begin{lem}[See \cite{Ad}, for example]\label{sob_algebra}
Let $a,b,c\in\R$ a triplet of real numbers satisfying \[
a\geq c,\quad b\geq c \quad a+b\geq 0\quad \hbox{and}\quad a+b-c>\tfrac{n}{2}.
\]
Then, the map $(f,g)\mapsto f\cdot g$ is a continuous bilinear form from $H^a(\R^n)\times H^b(\R^n)$ into $H^c(\R^n)$.
\end{lem}

The following lemma ensures us that $L^\infty_TH^s_{\omega}$-solutions also belong to $\mathcal{M}^s_{\omega,T}$, whereas the difference of two solutions in $L^\infty_TH^{s-1}_x$ take place in $\mathcal{M}^{s-1}_{T}$.
\begin{lem}\label{mst_basic_lemma}
Let $s>1/2$ and $T\in(0,2)$  given. Let $u\in L^\infty((0,T),H^s_\omega(\R))$ be a solution to equation \eqref{gkdv_v} with initial data $u_0\in H^s_\omega(\R)$. Then, $u\in\mathcal{M}^s_{\omega,T}$ and the following inequality holds \begin{align}\label{umt}
\Vert u\Vert_{\mathcal{M}^s_{\omega,T}}&\lesssim \Big(1+T^{1/2}\mathcal{F}_1\big(\Vert u\Vert_{L^\infty_TH^s_x},\Vert \Psi\Vert_{L^\infty_TW^{s^+,\infty}_x}\big)\Big)\Vert u\Vert_{L^\infty_TH^s_\omega}\nonumber
\\ & \qquad +T^{1/2}\Vert \partial_t\Psi+\partial_x^3\Psi+\partial_xf(\Psi)\Vert_{L^\infty_TH^{s-1}_x},
\end{align}
where $\mathcal{F}_1:\R^2\to\R_+$ is a smooth function. Moreover, for any pair $u,v\in L^\infty((0,T),H^s(\R))$ solutions to equation \eqref{gkdv_v} associated with initial data $u_0,v_0\in H^s(\R)$, the following holds: \begin{align}\label{umt_2}
\Vert u-v\Vert_{\mathcal{M}^{s-1}_{T}}\lesssim \Big(1+T^{1/2}\mathcal{F}_2\big(\Vert u\Vert_{L^\infty_TH^s_x},\Vert v\Vert_{L^\infty_TH^s_x},\Vert \Psi\Vert_{L^\infty_TW^{s^+,\infty}_x}\big)\Big)\Vert u-v\Vert_{L^\infty_TH^{s-1}_x},
\end{align}
for some smooth function $\mathcal{F}_2:\R^3\to\R_+$.
\end{lem}

\begin{proof}
First of all, we have to extend the functions $u(t)$ and $v(t)$ from $(0,T)$ to the whole line $\R$. Hence, we benefit from the extension operator $\rho_T$ defined in \eqref{rho_def}, which we use to take extensions $\tilde{u}:=\rho_T[u]$, $\tilde{v}:=\rho_T[v]$ defined on $\R^2$, both supported in $(-2,2)$. For the sake of notation, we drop the tilde in the sequel. 

\medskip

Once we have extended (in time) both solutions, comparing inequality \eqref{umt} with the definition of the $\mathcal{M}^s_T$ norm in \eqref{mst_norm_def}, it is clear that it is enough to estimate the $X^{s-1,1}$-norm. In fact, let us start out by writing the solution in its Duhamel form \[
u(t)=U(t)u_0+c\int_0^tU(t-t')\big(\partial_t\Psi+\partial_x^3\Psi+\partial_xf(u+\Psi)\big)dt'.
\]
Then by using standard linear estimates in Bourgain spaces we obtain
\begin{align}
\Vert u\Vert_{X^{s-1,1}_T}&\lesssim \Vert u\Vert_{L^\infty_TH^{s-1}_x}+\Vert \partial_t\Psi+\partial_x^3\Psi+\partial_xf(\Psi)\Vert_{X^{s-1,0}_T}\nonumber
\\ & \qquad +\big\Vert \partial_x(f(u+\Psi)-f(\Psi))\big\Vert_{X^{s-1,0}_T}\nonumber
\\ & \lesssim \Vert u\Vert_{L^\infty_TH^{s-1}_x}+T^{1/2}\Vert \partial_t\Psi+\partial_x^3\Psi+\partial_xf(\Psi)\Vert_{L^\infty_TH^{s-1}_x}\label{xst_ineq}
\\ & \qquad  +T^{1/2}\Vert f(u+\Psi)-f(\Psi)\Vert_{L^\infty_TH^s_x}.\nonumber
\end{align}
Now, from classical Sobolev estimates for products it is not difficult to see that, there exists a constant $c>0$ such that, for $k,m\in\N$, with $k\geq m$, \begin{align}\label{sobolev_embedding}
\Vert g^{k-m}h^m\Vert_{H^s}\leq c^k\Vert g\Vert_{H^s}^{k-m}\Vert h\Vert_{W^{s^+,\infty}}^m, \quad s>1/2.
\end{align}
Then, to control the contribution of the latter term in \eqref{xst_ineq}, it is enough to use Taylor expansion, from where we get\begin{align*}
\Vert f(u+\Psi)-f(\Psi)\Vert_{L^\infty_TH^s_x}&\lesssim \sum_{k=1}^{\infty}\sum_{m=0}^{k-1}c^k\vert a_k\vert \binom{k}{m}\Vert u\Vert_{L^\infty_TH^s_x}^{k-m}\Vert \Psi\Vert_{L^\infty_TW^{s^+,\infty}_x}^m
\\ & \lesssim \Vert u\Vert_{L^\infty_TH^s_x}\sum_{k=1}^\infty kc^k\vert a_k\vert \big(\Vert u\Vert_{L^\infty_TH^s_x}+\Vert\Psi\Vert_{L^\infty_TW^{s^+,\infty}_x}\big)^{k-1}.
\end{align*}
This concludes inequality \eqref{umt} thanks to the hypothesis on the coefficients $a_k$ in \eqref{hyp_f_nonl}. Now we turn to show \eqref{umt_2}. In order to do so, let us define $w:=u-v$. Notice that $w(t,x)$ solves the following equation
\begin{align*}
\partial_tw+\partial_x\big(\partial_x^2w+f(u+\Psi)-f(v+\Psi)\big)=0.
\end{align*}
Thus, writing $w(t,x)$ in its Duhamel form, and then using standard linear estimates for Bourgain spaces, we get
\begin{align*}
\Vert w\Vert_{X^{s-2,1}_T}&\lesssim \Vert w\Vert_{L^\infty_TH^{s-2}_x}+T^{1/2}\Vert f(u+\Psi)-f(v+\Psi)\Vert_{L^\infty_TH^{s-1}_x}.
\end{align*}
Then, by using Taylor expansion,  proceeding in a similar fashion as above, it is not difficult to see that \begin{align*}
&\Vert f(u+\Psi)-f(v+\Psi)\Vert_{L^\infty_TH^{s-1}_x}
\\ & \qquad \lesssim \Vert w\Vert_{L^\infty_TH^{s-1}_x}\sum_{k=1}^\infty kc^k\vert a_k\vert\big(\Vert u\Vert_{L^\infty_TH^s_x}+\Vert v\Vert_{L^\infty_TH^s_x}+\Vert\Psi\Vert_{L^\infty_TW^{s^+,\infty}_x}\big)^{k-1},
\end{align*}
where we have used again \eqref{sobolev_embedding}, as well as Lemma \ref{sob_algebra}. The proof is complete.
\end{proof}

\smallskip

\subsection{Preliminary lemmas} For $T>0$ fixed, we consider $\mathds{1}_T(t)$, the characteristic function on $[0,T]$. Then, for $\eta$ given in \eqref{def_eta}, and $R>0$, we decompose $\mathds{1}_T(t) $ as \begin{align}\label{decomposition}
\mathds{1}_T(t)=\mathds{1}_{T,R}^{\mathrm{low}}(t)+\mathds{1}_{T,R}^{\mathrm{high}}(t), \quad \hbox{ where }\quad  \mathcal{F}_t\big(\mathds{1}_{T,R}^{\mathrm{low}}\big)(\tau)=\eta\left(\dfrac{\tau}{R}\right) \mathcal{F}_t(\mathds{1}_T)(\tau).
\end{align}
The following lemma gives us some basic estimates that shall be particularly convenient to take advantage of the above decomposition along with Bourgain spaces.
\begin{lem}[See \cite{MoVe}]\label{lem_high_low_1}
For any $R>0$, $T>0$ and $q\geq1$ the following bounds hold: \begin{align}\label{indicatrix_inequality}
\Vert \mathds{1}_{T,R}^{\mathrm{high}}\Vert_{L^q}\lesssim \min\{T,R^{-1}\}^{1/q} \quad \hbox{ and }\quad \Vert \mathds{1}_{T,R}^{\mathrm{low}}\Vert_{L ^\infty}\lesssim 1.
\end{align}
Moreover, if $L\gg R$, then for all $u\in L^2(\R^2)$, the following inequality holds: \[
\big\Vert Q_L\big(\mathds{1}_{T,R}^{\mathrm{high}}u\big)\big\Vert_{L^2}\lesssim \Vert Q_{\sim L}u\Vert_{L^2}.
\]
Furthermore, for any $s\in\R$ and any $p\in[1,\infty]$, the operator $Q_{\leq L}$ is bounded in $L^p_tH^s_x$ uniformly in $L$.
\end{lem}

\begin{proof}
We point out that the only part which is not strictly contained in \cite{MoVe} is given by the first inequality in \eqref{indicatrix_inequality}. However, this proof follows very similar lines, except for one straightforward step, and hence we shall be brief. In fact, a direct computation yields \begin{align*}
\Vert \mathds{1}_{T,R}^{\mathrm{high}}\Vert_{L^q}&=\left(\int_\R\left\vert \int_\R\left(\mathds{1}_T(t)-\mathds{1}_T\big(t-\tfrac{s}{R}\big)\right)\mathcal{F}^{-1}\eta(s)ds\right\vert^q dt\right)^{1/q}
\\ & \leq \int_\R\left(\int_{\R}\left\vert\mathds{1}_T(t)-\mathds{1}_T\big(t-\tfrac{s}{R}\big)\right\vert^q\big\vert\mathcal{F}^{-1}\eta(s)\big\vert^qdt\right)^{1/q} ds
\\ & \lesssim \int_\R \min\big\{T,\tfrac{\vert s\vert}{R}\big\}^{1/q} \big\vert\mathcal{F}^{-1}\eta(s)\big\vert ds\lesssim \min\big\{T,R^{-1}\big\}^{1/q},
\end{align*}
where to obtain the first inequality we have used Minkowski integral inequality. The proof is complete.
\end{proof}

\begin{lem}[\cite{MoVe}]\label{resonant_bourgain}
Let $k\geq 3$ be a fixed parameter. Consider $L_1,...,L_k\geq 1$ and $N_1,...,N_k>0$ to be a list of $2k$ dyadic numbers. Consider $\{u_i\}_{i=1}^k\subset\mathcal{S}'(\R^2)$ all of them being real-valued, and let $\chi\in L^\infty(\R^2;\mathbb{C})$. Additionally, assume that\footnote{If $k=3$ we omit the last inequality.} $N_1\geq N_2\geq N_3\geq 2^9k\max\{N_4,...,N_k\}$. Then, the following identity holds 
\begin{align*}
\int_{\R^2}\Pi_\chi\big(Q_{L_1}P_{N_1}u_1,Q_{L_2}P_{N_2}u\big)\prod_{i=3}^k Q_{L_i}P_{N_i}u=0,
\end{align*}
unless $L_{\max}\geq (2^{9}k)^{-1}N_1N_2N_3$, where $L_{\max}:=\max\{L_1,...,L_k\}$. 
\end{lem}

Finally, we prove some basic lemmas concerning the application of H\"older inequality with some particular instances of pseudo-product operators that shall appear in the next section.

\begin{lem}\label{pseudo_holder_obvio_l2}
Let $k\geq 2$ be a fixed natural number. Consider $k$ functions $u_1,...,u_k\in L^2(\R)$, such that each of them is supported in the annulus $\{\vert\xi_i\vert\sim N_i\}$. Additionally consider a continuous function $a\in C(\R^k)$. Then, the following inequality holds
\[
\left\vert \int_{\Gamma^k} a(\xi_1,...,\xi_k)u_1(\xi_1)...u_{k}(\xi_k)d\Gamma^k\right\vert\lesssim (N_3...N_k)^{1/2}\Vert a\Vert_{L^\infty(\Omega_k)}\prod_{i=1}^{k}\Vert u_i\Vert_{L^2},
\]
where $\Omega_{k}$ stands for the set \[
\Omega_{k}:=\big\{(\xi_1,...\xi_k)\in \Gamma^k:\, \forall i=1,...,k, \, \xi_i\in\supp u_i\big\}.
\]
\end{lem}

\begin{proof}
Let us start by assuming that $k\geq 3$, since the case $k=2$ is direct. In fact, first of all we get rid of $a(\xi_1,...,\xi_k)$ simply by bounding as follows 
\begin{align*}
\left\vert \int_{\Gamma^k} a(\xi_1,...,\xi_k)u_1(\xi_1)...u_{k}(\xi_k)d\Gamma^k\right\vert\leq \Vert a\Vert_{L^\infty(\Omega_k)}\int_{\Gamma^k} \vert u_1(\xi_1)...u_k(\xi_k)\vert d\Gamma^k.
\end{align*}
Then, it is enough notice that we can bound the latter integral in the above inequality by \begin{align*}
\int_{\Gamma^k} \vert u_1(\xi_1)...u_k(\xi_k)\vert d\Gamma^k\leq \sup_{\xi_3,...,\xi_k}\int_\R \vert u_1(-\xi_2-...-\xi_k)u_2(\xi_2)\vert d\xi_2\times \prod_{i=3}^k\int_\R \vert u_i(\xi)\vert d\xi
\end{align*} 
Finally, noticing that, by Cauchy-Schwarz inequality and the supports hypotheses, we have 
\begin{align*}
\int_\R \vert u_i(\xi)\vert d\xi\lesssim N_i^{1/2}\Vert u_i\Vert_{L^2} \quad \,\hbox{and}\,\quad \left\vert \int_\R u_1(-\xi_2-....-\xi_k)u_2(\xi_2)d\xi_2\right\vert\lesssim\Vert u_1\Vert_{L^2}\Vert u_2\Vert_{L^2}.
\end{align*}
The proof is complete.
\end{proof}

\begin{lem}\label{pseudo_holder_l2}
Let $k\geq 2$ and $m\in[2, k+1]$ two fixed natural numbers. Additionally, consider a family of  dyadic numbers $\{N_i\}_{i=1}^{k+2}$, all of them  being fixed. Let $u_1,...,u_{k+1}\in L^2(\R)$ such that each of them has Fourier transform supported on the ball $\{\vert\xi_i\vert\leq N_i\}$, respectively. Furthermore, assume that $M=\max\{N_3,...,N_{k+1}\}$. Then, the following holds
\[
\left\vert \int_{\Gamma^{k+1}} a(\xi_1,...,\xi_k)\hat{u}_1(\xi_1)...\hat{u}_{k+1}(\xi_{k+1})d\Gamma^{k+1}\right\vert\lesssim M\Vert u_1\Vert_{L^2}\Vert u_2\Vert_{L^2}\prod_{i=3}^{k+1}\Vert u_i\Vert_{L^\infty},
\]
where the implicit constant depends polynomially on $k$ and $a(\xi_1,...,\xi_k)$ stands for the following quantity \begin{align}\label{definition_symbol_a}
a(\xi_1,...,\xi_k):=\sum_{i=1}^{m}\phi_{N_{k+2}}^2(\xi_i)\xi_i,\quad \hbox{where}\quad \xi_{k+1}=-\xi_1-...-\xi_k.
\end{align}
\end{lem}

\begin{proof}
In fact, first of all notice that, except for the terms associated with $i=1,2$ in the definition of $a(\xi_1,...,\xi_k)$, the proof follows directly from Plancherel Theorem and H\"older inequality. Indeed, going back to physical-variables and then using H\"older inequality as well as Bernstein inequalities, we obtain
\begin{align*}
&\left\vert \int_{\Gamma^{k+1}} \Big(a(\xi_1,...,\xi_k)-\phi_N^2(\xi_1)\xi_1-\phi_N^2(\xi_2)\xi_2\Big)\hat{u}_1(\xi_1)...\hat{u}_{k+1}(\xi_{k+1})d\Gamma^{k+1}\right\vert
\\ & \qquad \lesssim \Vert u_1\Vert_{L^2}\Vert u_2\Vert_{L^2}\sum_{j=3}^{m}\Vert \partial_xP_N^2u_j\Vert_{L^\infty}\prod_{i=3,i\neq j}^{m}\Vert u_i\Vert_{L^\infty}
\\ & \qquad \lesssim M \Vert u_1\Vert_{L^2}\Vert u_2\Vert_{L^2}\prod_{i=3}^{k+1}\Vert u_i\Vert_{L^\infty}.
\end{align*}
Therefore, we can restrict ourselves to study the above integral when replacing $a(\xi_1,...,\xi_k)$ by the following symbol \[
\widetilde{a}(\xi_1,...,\xi_k):=\phi_N^2(\xi_1)\xi_1+\phi_N^2(\xi_2)\xi_2.
\]
Next, we split this symbol into two parts as follows
\[
\widetilde{a}(\xi_1,...,\xi_k)=\phi_N^2(\xi_1)\big(\xi_1+\xi_2\big)-\big(\phi_N^2(\xi_1)-\phi_N^2(\xi_2)\big)\xi_2=:\widetilde{a}_1(\xi_1,...,\xi_k)+\widetilde{a}_2(\xi_1,...,\xi_k).
\]
Notice now that due to the additional restriction imposed by $\Gamma^{k+1}$, in this domain we can rewrite $\widetilde{a}_1(\xi_1,...,\xi_k)$ as \[
\widetilde{a}_1(\xi_1,...,\xi_k)=-\phi_N^2(\xi_1)(\xi_3+...+\xi_{k+1}).
\]
Hence, this case also follows from the above analysis. Therefore, it only remains to consider the case of $\widetilde{a}_2$. For the sake of clarity we shall assume now that $k=2$, the proof for the general case shall be clear from this one. In fact, by using Plancherel Theorem, integration by parts and then H\"older inequality we immediately obtain that
\begin{align*}
\left\vert\int_{\Gamma^3}\widetilde{a}_2(\xi_1,\xi_2)\hat{u}_1(\xi_1)\hat{u}_2(\xi_2)\hat{u}_3(\xi_3)d\Gamma^3 \right\vert&=\left\vert\int_\R \big(u_{2,x}P_N^2u_1-u_1P_N^2u_{2,x}\big) u_3dx\right\vert
\\ & = \left\vert\int_\R u_2 \partial_x\big(u_3P_N^2u_1-P_N^2(u_1u_3)\big)dx\right\vert
\\ & \lesssim \Vert u_1\Vert_{L^2}\Vert u_2\Vert_{L^2}\Vert \partial_x u_3\Vert_{L^\infty}+\Vert u_2\Vert_{L^2}\big\Vert [P_N^2\partial_x,u_3]u_1\big\Vert_{L^2}.
\end{align*}
Then, since $\Vert \partial_xu_3\Vert_{L^\infty}\lesssim M\Vert u_3\Vert_{L^\infty}$, it only remains to control the latter factor of the above inequality. In order to do that, first notice that by direct computations we can write \begin{align*}
[P_N^2\partial_x,u_3]u_1(x)=\int_\R K(x,y)u_1(y)dy,
\end{align*} 
where the kernel $K(x,y)$ can be written as \begin{align*}
K(x,y)&=icN^2\int_\R e^{iN(x-y)\eta}\eta\phi^2(\eta)\big(u_3(y)-u_3(x)\big) d\eta,
\end{align*}
for some constant $c\in\R$. Thus, as an application of the Mean Value Theorem it is not difficult to see that there exists a function $g\in L^1(\R)$ such that 
\[
\big\vert K(x,y)\big\vert \lesssim N\Vert \partial_xu_3\Vert_{L^\infty}g\big(N(x-y)\big).
\]
Notice that the latter inequality implies, in particular, the following uniform bound \[
\sup_{y\in\R}\int_\R\big\vert K(x,y)\big\vert dx +\sup_{x\in\R}\int_\R\big\vert K(x,y)\big\vert dy\lesssim \Vert \partial_xu_3\Vert_{L^\infty},
\]
where the implicit constant does not depends on $N$. Therefore, applying Schur lemma, and then Bernstein inequality in the resulting right-hand side, we obtain that 
\begin{align*}
\big\Vert [P_N^2\partial_x,u_3]u_1\big\Vert_{L^2}\lesssim \Vert u_1\Vert_{L^2}\Vert \partial_xu_3\Vert_{L^\infty}\lesssim M \Vert u_1\Vert_{L^2}\Vert u_3\Vert_{L^\infty},
\end{align*}
what concludes the proof of the lemma.
\end{proof}

\smallskip

\subsection{Strichartz estimates}
In this sub-section we seek to prove a refined Strichartz estimate for solutions to the linear Airy equation with a general source term. The proof we present here is just a slight modification of the arguments already shown in \cite{KoTz,MoPiVe,MoPiVe3}. Before getting into the details, let us recall the classical smoothing effect derived in \cite{KePoVe}  that shall be useful in the sequel
\begin{align}\label{strichartz_classical}
\big\Vert e^{-t\partial_x^3}D_x^{1/4}u_0\big\Vert_{L^4_tL^\infty_x}\lesssim \Vert u_0\Vert_{L^2_x}.
\end{align}
Now we are ready to state our refined Strichartz estimate.
\begin{lem}\label{lem_smoothing}
Let $T>0$ and consider $\delta\geq 0$ to be a fixed parameter. Let $u(t,x)$ to be any solution defined on $[0,T]$ to the following linear equation \begin{align}\label{airy}
\partial_tu+\partial_x^3u=F.
\end{align}
Then, there exists $\kappa_1,\kappa_2\in(\tfrac{1}{4},\tfrac{1}{2})$ such that, for any $\theta>0$, the following inequality holds
\begin{align}\label{smoothing_eff}
\big\Vert u\Vert_{L^2_TL^\infty_x}\lesssim T^{\kappa_1}\big\Vert J_x^{-\frac{1}{4}(1-\delta)+\theta}u\big\Vert_{L^\infty_TL^2_x}+T^{\kappa_2}\big\Vert J_x^{-\frac{1}{4}(1+3\delta)+\theta}F\big\Vert_{L^2_TL^2_x}.
\end{align}
\end{lem}

\begin{proof}
Let $u(t,x)$ be a solution to equation \eqref{airy} defined on $[0,T]$. We use a nonhomogeneous Littlewood-Paley decomposition for the solution, that is, we write $u=\sum_Nu_N$, where $u_N=P_Nu$, and $N$ is a nonhomogeneous dyadic number. In the sequel we shall also use the notation $F_N$ for $P_NF$. At this point it is important to notice that, on the one-hand, from Minkowski inequality we know that \[
\Vert u\Vert_{L^2_TL^\infty_x}\leq \sum_{N}\Vert u_N\Vert_{L^2_TL^\infty_x}\lesssim \sup_{N}N^\theta\Vert u_N\Vert_{L^2_TL^\infty_x},
\]
for any $\theta>0$. While on the other hand, by using the low-frequency projector $P_{\leq1}$, from H\"older and Bernstein inequalities we see that \[
\Vert P_{\leq1}u\Vert_{L^2_TL^\infty_x}\lesssim T^{1/2}\Vert P_{\leq1}u\Vert_{L^\infty_TL^2_x}.
\]
Then, from the inequalities above we infer that it is enough to show that, for any $\delta>0$ and any $N>1$ dyadic number, the following holds \begin{align}\label{smoothing_eq_2}
\Vert u_N\Vert_{L^2_TL^\infty_x}\lesssim T^{\kappa_1}\big\Vert D_x^{-\frac{1}{4}(1-\delta)}u_N\big\Vert_{L^\infty_TL^2_x}+T^{\kappa_2}\big\Vert D_x^{-\frac{1}{4}(1+3\delta)}F_N\big\Vert_{L^2_TL^2_x}.
\end{align}
Now, in order to prove \eqref{smoothing_eq_2}, we chop the time-interval $[0,T]$ into several pieces of length $T^{\kappa}N^{-\delta}$, where $\kappa>0$ stands for a small number that shall be fixed later. In other words, we have \[
[0,T]=\cup_{j\in J}I_j,\quad\hbox{where}\quad I_j:=[a_j,b_j], \quad \vert I_j\vert\sim T^\kappa N^{-\delta}, \quad \hbox{and}\quad \# J\sim T^{1-\kappa}N^\delta.
\]
On the other hand, notice that $u_N(t)$ solves the integral equation \[
u_N(t)=e^{-(t-a_j)\partial_x^3}u_N(a_j)+\int_{a_j}^te^{-(t-t')\partial_x^3}F_N(t')dt',
\]
for all $t\in I_j$. Therefore, by using the classical Strichartz estimate \eqref{strichartz_classical}, as well as H\"older and Bernstein inequalites, we obtain \begin{align*}
\Vert u_N\Vert_{L^2_TL^\infty_x}&=\left(\sum_{j}\Vert u_N\Vert_{L^2_{I_j}L^\infty_x}^2\right)^{1/2}\leq \big(T^\kappa N^{-\delta}\big)^{1/4}\left(\sum_{j}\Vert u_N\Vert_{L^4_{I_j}L^\infty_x}^2\right)^{1/2}
\\ & \lesssim \big(T^\kappa N^{-\delta}\big)^{1/4}\left(\sum_j \Vert D_x^{-1/4}u_N(a_j)\Vert_{L^2_x}^2\right)^{1/2}
\\ & \qquad  +\big(T^\kappa N^{-\delta}\big)^{1/4}\left(\sum_{j}\bigg\Vert \int_{a_j}^te^{-(t-t')\partial_x^3}F_N(t')dt' \bigg\Vert_{L^4_{I_j} L^\infty_x}^2 \right)^{1/2}
\\ & \lesssim \big(T^\kappa N^{-\delta})^{1/4}(T^{1-\kappa}N^\delta)^{1/2}\Vert D_x^{-1/4}u_N\Vert_{L^\infty_TL^2_x}
\\ & \qquad  +\big(T^\kappa N^{-\delta}\big)^{1/4}\left(\sum_{j}T^{\kappa}N^{-\delta}\int_{I_j}\Vert D_x^{-1/4}F_N\Vert_{L^2_x}^2dt\right)^{1/2}
\\ & \lesssim T^{1/2-\kappa/4}\Vert D_x^{-1/4+\delta/4}u_N\Vert_{L^\infty_TL^2_x}+T^{3\kappa/4}\Vert D_x^{-1/4-3\delta/4}F_N\Vert_{L^2_TL^2_x},
\end{align*}
what concludes the proof of \eqref{smoothing_eff} by choosing, for example, $\kappa=\tfrac{1}{2}$.
\end{proof}

\medskip

\section{Energy estimates}\label{energy_section}

\subsection{A priori estimates for solutions}

The goal of this section is to prove the following proposition that give us the key improved energy estimate for smooth solutions of \eqref{gkdv_v}.

\begin{prop}\label{prop_energy_est_impro}
Let $s>1/2$ and $T\in(0,2)$ both fixed. Consider $u\in L^\infty((0,T),H^s_\omega(\R))$ to be a solution to equation \eqref{gkdv_v} associated with an initial data $u_0\in H^s_\omega(\R)$. Then, the following inequality holds: \begin{align}\label{energy_est_improved_ineq}
\Vert u\Vert_{L^\infty_TH^s_\omega}^2&\lesssim \Vert u_0\Vert_{H^s_\omega}^2+T\Vert u\Vert_{L^\infty_TH^s_\omega}\Vert \partial_t \Psi+\partial_x^3\Psi+\partial_xf(\Psi)\Vert_{L^\infty_TH^{s^+}_x}+ T^{1/4}\Vert u\Vert_{L^\infty_TH^s_\omega}^2\times
\\ & \quad \times \mathcal{Q}_*\big(\Vert u\Vert_{L^\infty_TH^{1/2^+}_x},\Vert \Psi\Vert_{L^\infty_TW^{s+1^+,\infty}_x},\Vert \partial_t\Psi\Vert_{L^\infty_{t,x}},\Vert \partial_t\Psi+\partial_x^3\Psi+\partial_xf(\Psi)\Vert_{L^\infty_TH^{-1/2^+}_x}\big),\nonumber
\end{align}
where $\mathcal{Q}_*:\R^4\to\R_+$ is a smooth function.
\end{prop}

\begin{proof}
First of all, in order to take advantage of Bourgain spaces, we have to extend the function $u(t)$ from $(0,T)$ to the whole line $\R$. Hence, we benefit from the extension operator $\rho_T$ defined in \eqref{rho_def}, which we use to take an extension $\tilde{u}:=\rho_T[u]$ defined on $\R^2$, such that $\Vert \tilde{u}\Vert_{\mathcal{M}}\leq 2\Vert u\Vert_{\mathcal{M}_T}$. For the sake of notation, we drop the tilde in the sequel.


\medskip

Now we seek to prove \eqref{energy_est_improved_ineq}. We begin by applying the frequency projector $P_N$ to equation \eqref{gkdv_v}, with $N>0$ dyadic but arbitrary. Notice that, on account of Remark \ref{rem_spaces_ic_duhamel}, we have \[
P_Nu\in C([0,T],H^\infty) \quad \hbox{and} \quad \partial_tP_Nu\in L^\infty((0,T),H^\infty).
\]
Therefore, taking the $L^2_x$-scalar product of the
resulting equation against $P_Nu$, multiplying the result by $\omega_N^2\langle N\rangle^{2s}$ and then integrating on $(0, t)$ with $0<t<T$, we obtain
\begin{align*}
\omega_N^2\langle N\rangle^{2s}\Vert P_Nu(t)\Vert_{L^2}^2&=\omega_N^2\langle N\rangle^{2s}\Vert P_Nu_0\Vert_{L^2}^2
\\ & \quad -\omega_N^2\langle N\rangle^{2s}\int_0^t\int_\R P_N\big(\partial_t\Psi+\partial_x^3\Psi+\partial_xf(u+\Psi)\big)P_Nu. 
\end{align*}
Thus, by applying Bernstein inequality we are lead to 
\begin{align*}
\Vert P_N u(t)\Vert_{H^s_\omega}^2\lesssim \Vert P_Nu_0\Vert_{H^s_\omega}^2+\omega_N^2\langle N\rangle^{2s}\sup_{t\in(0,T)} \left\vert\int_0^t\int_\R P_N\big(\partial_t\Psi+\partial_x^2\Psi+f(u+\Psi)\big)\partial_xP_Nu\right\vert.
\end{align*}
Thus, from the previous computation we infer that, in order to conclude the proof of the proposition, we need to control the sum over all $N>0$ dyadic of the second term in the latter inequality. We divide the analysis into several steps, each of which dedicated to bound one of the following integrals:
\begin{align}
\mathrm{I}&:=\sum_{N>0}\omega_N^2\langle N\rangle^{2s}\sup_{t\in(0,T)}\left\vert\int_0^t\int_\R\partial_xP_N(f(u+\Psi)-f(\Psi))P_Nu\right\vert,\label{sum_I_proof}
\\ \mathrm{II}&:= \sum_{N>0}\omega_N^2\langle N\rangle^{2s}\sup_{t\in(0,T)}\left\vert\int_0^t\int_\R P_N(\partial_t\Psi+\partial_x^3\Psi+\partial_xf(\Psi))P_Nu\right\vert.\nonumber
\end{align}
Before going further we recall that due to the analyticity hypothesis \eqref{hyp_f_nonl} we can write
\[
f\big(u(t,x)+\Psi(t,x)\big)=\sum_{k=0}^{\infty}a_k\big(u(t,x)+\Psi(t,x)\big)^k \quad \hbox{and}\quad f\big(\Psi(t,x)\big)=\sum_{k=0}^{\infty}a_k\Psi^k(t,x).
\]
With this in mind, from now on, for any $n,m\in\N$, we shall denote by $\mathrm{I}_{u^n}$ and $\mathrm{I}_{u^n\Psi^m}$ the quantity $\mathrm{I}$ above once $f(u+\Psi)$ is replaced by $u^k$ and $u^k\Psi^m$ respectively, that is, \begin{align}
\mathrm{I}_{u^k}&:= \sum_{N>0}\omega_N^2\langle N\rangle^{2s}\sup_{t\in(0,T)}\left\vert\int_0^t\int_\R\partial_xP_N(u^k)P_Nu\right\vert, \label{I_un}
\\ \mathrm{I}_{u^k\Psi^m}&:=\sum_{N>0}\omega_N^2\langle N\rangle^{2s}\sup_{t\in(0,T)}\left\vert\int_0^t\int_\R\partial_xP_N(u^k\Psi^m)P_Nu\right\vert.\label{I_un_psim}
\end{align}
We point out that, in the sequel, we shall systematically omit most of the factors depending on $k$ by hiding them in a $\lesssim_k$-sign. This sign is defined exactly as ``$\lesssim$'' in Section \ref{preliminaries} but allowing the constant $c$ to depend on $k$. Notice that, in order to make sense of the sum (in $k$) of all the following bounds, we only need to be careful that the final implicit constant depends at most as $c^k$ for some constant $c>0$. 

\medskip

\textbf{Step 1:} We begin by controlling $\mathrm{II}$ right away. In fact, by using hypothesis \eqref{hyp_psi}, we infer that it is enough to use Cauchy-Schwarz and Bernstein inequalities, from where we obtain \begin{align*}
\mathrm{II}&\lesssim \int_0^T\sum_{N>0}\omega_N^2\langle N\rangle^{2s}\big\Vert P_N\big(\partial_t\Psi+\partial_x^3\Psi+\partial_xf(\Psi)\big)\big\Vert_{L^2_x}\Vert P_Nu(s,\cdot)\Vert_{L^2_x}ds
\\ & \lesssim T\Vert u\Vert_{L^\infty_TH^s_\omega}\Vert \partial_t \Psi+\partial_x^3\Psi+\partial_xf(\Psi)\Vert_{L^\infty_TH^{s^+}_x},
\end{align*}
where we have used the hypotheses made on $\omega_N$ in Section \ref{function_spaces_section}. This concludes the proof of the first case.

\medskip

\textbf{Step 2:} Now we aim to control the general case for $\mathrm{I}_{u^k}$ in \eqref{I_un} for all $k\geq 1$, that is, we aim to control the following quantity
\begin{align}\label{def_iuk}
\mathrm{I}_{u^k}&=\sum_{N>0}\omega_N^2\langle N\rangle^{2s}\sup_{t\in(0,T)}\left\vert\int_0^t\int_{\Gamma^{k+1}}a_k(\xi_1,...,\xi_{k+1})\hat{u}(s,\xi_1)...\hat{u}(s,\xi_{k+1})d\Gamma^{k+1}ds\right\vert,
\end{align}
where the symbol $a_k(\xi_1,...,\xi_{k+1})$ is explicitly given by \begin{align*}
a_k\big(\xi_1,...,\xi_{k+1}\big):=i\phi_N^2(\xi_{k+1})\xi_{k+1}.
\end{align*}
We point out that in the previous identity \eqref{def_iuk} we have used both, the fact that $u(t,\cdot)$ is real-valued as well as the fact that $\phi_{N}$ is even. Then, in order to deal with this case, we symmetrize the multiplier $a(\xi_1,...,\xi_{k+1})$, that is, from now on we consider \begin{align*}
\widetilde{a}_k(\xi_1,...,\xi_{k+1}):=\big[a_k(\xi_1,...,\xi_{k+1})\big]_{\mathrm{sym}}=\dfrac{i}{k+1}\sum_{i=1}^{k+1}\phi_N^2(\xi_i)\xi_i.
\end{align*}
Notice that, since $\phi_N^2(\xi_1)\xi_1+\phi_N^2(\xi_2)\xi_2\equiv0$ on $\Gamma^2$, the case $k=1$ immediately vanishes, and hence, from now on we can assume that $k\geq 2$. Thus, by using frequency decomposition and the above symmetrization, the problem of bounding \eqref{def_iuk} is reduced to control the following quantity \begin{align}\label{sums_step31}
\sum_{N>0}\omega_N^2\langle N\rangle^{2s}\sup_{t\in(0,T)}\left\vert\int_0^t\sum_{N_1,...,N_{k+1}}\int_{\Gamma^{k+1}}\widetilde{a}_k(\xi_1,...,\xi_{k+1})\prod_{i=1}^{k+1}\phi_{N_i}(\xi_i)\hat{u}(t',\xi_i)d\Gamma^{k+1}dt'\right\vert.
\end{align}
Moreover, by symmetry, without loss of generality we can always assume that\footnote{If $k=2$, we only assume $N_1\geq N_2\geq N_3$. Notice also that this assumption shall introduce a factor $k^4$ into the following estimates } \[
N_1\geq N_2\geq N_3\geq N_4=\max\{N_4,...,N_{k+1}\}.
\]
Before going further, note that the case $N\lesssim1$ can be treated right away. In fact, from Lemma \ref{pseudo_holder_l2} and Bernstein inequality we see that\footnote{Notice that here, and in all of the bounds below, we obtain a constant $c^{k-1}$ comming from using Sobolev's embedding $k-1$ times.} \begin{align*}
&\sum_{N\lesssim 1}\omega_N^2\langle N\rangle^{2s}\sup_{t\in(0,T)}\left\vert\int_0^t\sum_{N_1,...,N_{k+1}}\int_{\Gamma^{k+1}}\widetilde{a}_k(\xi_1,...,\xi_{k+1})\prod_{i=1}^{k+1}\phi_{N_i}(\xi_i)\hat{u}(t',\xi_i)d\Gamma^{k+1}dt'\right\vert
\\ & \qquad \lesssim_k \int_0^T\sum_{N\lesssim 1}\omega_N^2\langle N\rangle^{2s}\sum_{N_1,...,N_{k+1}}\left\vert\int_{\Gamma^{k+1}}\widetilde{a}_k(\xi_1,...,\xi_{k+1})\prod_{i=1}^{k+1}\phi_{N_i}(\xi_i)\hat{u}(t',\xi_i)d\Gamma^{k+1}\right\vert dt'
\\ & \qquad  \lesssim_k \int_0^T\sum_{N_1,...,N_{k+1}} \omega_N^2N_3\Vert P_{N_1}u(t',\cdot)\Vert_{L^2_x}\Vert P_{N_2}u(t',\cdot)\Vert_{L^2_x}\prod_{i=3}^{k+1}N_i^{1/2}\Vert P_{N_i}u(t',\cdot)\Vert_{L^2_x}dt'
\\ & \qquad \lesssim_k \int_0^T \Vert u(t',\cdot)\Vert_{H^s_\omega}^{2}\Vert u(t',\cdot)\Vert_{H^{1/2^+}_x}^{k-1}dt'
\\ & \qquad \lesssim_k T\Vert u\Vert_{L^\infty_TH^s_\omega}^{2}\Vert u\Vert_{L^\infty_TH^{1/2^+}_x}^{k-1}.
\end{align*}
Therefore, in the sequel we just need to consider the sum over frequencies $N\gg 1$. More precisely, from now on we assume that $N\geq 8^8k$. On the other hand, from the explicit form of $\widetilde{a}_k$, it is not difficult to see that  $\widetilde{a}_k\equiv 0$, unless $ N_1\geq \tfrac{1}{2}N$. Furthermore, due to the additional constraint\footnote{By this we mean the condition $\xi_1+....+\xi_{k+1}=0$. In the sequel, each time we mention ``the constraint imposed by $\Gamma^k$'' we refer to the previous condition with $k$ frequencies.} imposed by $\Gamma^{k+1}$, we must also have that $N_2\geq \tfrac{1}{2k}N_1$. Therefore, roughly (up to a constant involving $k$), we have that $N_1\sim N_2$ with\footnote{Notice that, in the sequel, we shall repeatedly use these relations to absorb factors like $\langle N\rangle^{s}$ with $\Vert P_{N_2}u\Vert_{L^2}$, in the sense that we shall write $\langle N\rangle^{s}\Vert P_{N_2}u\Vert_{L^2}\lesssim_k \Vert P_{N_2}u\Vert_{H^s}$. Due to the above relations, in the worst case this type of bounds shall involve a factor $k^s$ due to the use of $N_2\sim N$.} $N_1\geq \tfrac{1}{2}N$. Then, we split the analysis into three possible cases. First, we divide the space into two regions, namely \begin{align}\label{cases_step31}
\hbox{either }\quad N_3\geq 2^9kN_4 \quad \hbox{ or }\quad N_3<2^9kN_4.
\end{align}
Then, only for the second case, we split the space again into two regions, namely  \begin{align}\label{split_N1}
N_1< 8kN \quad \hbox{and}\quad N_1\geq 8kN. 
\end{align}
The only reason why we separate both cases in \eqref{split_N1} is to be able to justify how we sum over the set $N\gg1$, they can certainly be treated simultaneously though. We choose to separate them for the sake of clarity.

\medskip

Before getting into the details, let us introduce some notation for each of the regions under study. From now on we denote by\footnote{Recall that we are also assuming that $N_1\geq N_2\geq N_3\geq N_4=\max\{N_4,...,N_{k+1}\}$.} \begin{align*}
\mathbf{N}^1&:=\big\{(N_1,...,N_{k+1})\in \mathbb{D}^{k+1}: \, N_3\geq 2^9kN_4\big\},
\\ \mathbf{N}^2&:=\big\{(N_1,...,N_{k+1})\in \mathbb{D}^{k+1}: \, N_1< 8kN \,\hbox{ and } \,N_3< 2^9kN_4\big\},
\\ \mathbf{N}^3&:=\big\{(N_1,...,N_{k+1})\in \mathbb{D}^{k+1}: \, N_1\geq 8kN \,\hbox{ and } \,N_3< 2^9kN_4\big\},
\end{align*} 
and by $\mathcal{G}_1$, $\mathcal{G}_2$ and $\mathcal{G}_3$, the corresponding contribution of \eqref{sums_step31} associated with each of these regions\footnote{That is, the quantity obtained once restricting the inner sum in \eqref{sums_step31} to $\mathbf{N}^1$ and $\mathbf{N}^2$, respectively.}, respectively.

\medskip

Notice that all of the above regions require that $k\geq3$ to be well-defined. However, we point out that the case $k=2$ shall follow directly from the analysis that we shall carry out to deal with the first region above, that is, the region\footnote{In other words, roughly speaking, when $k=2$ we could think of $N_4$ as being equal to $0$, and hence the relation that defines $\mathbf{N}^1$ is always satisfied. Hence, if $k=2$, we only have one case, which corresponds to $\mathbf{N}^1$.} $\mathbf{N}^1$.

\medskip

\textbf{Step 2.1:} In this first sub-step we seek to deal with the first case in \eqref{cases_step31}, that is, to control the contribution of the region $N_3\geq2^9kN_4$. We aim to take advantage of classical Bourgain estimates. In order to do so, we begin using the decomposition given in \eqref{decomposition}, from where we infer that it is enough to control the following quantities \begin{align*}
\mathcal{G}_{1,R}^{\mathrm{high}}& :=\sum_{N\gg 1}\sum_{\textbf{N}^1}\omega_N^2\langle N\rangle^{2s}\sup_{t\in(0,T)}\left\vert \int_{\R^2}\Pi_{\widetilde{a}_k}\big(\mathds{1}_{t,R}^{\mathrm{high}}P_{N_1}u,\mathds{1}_tP_{N_2}u,...,P_{N_k}u\big)P_{N_{k+1}}u\right\vert,
\\ \mathcal{G}_{1,R}^{\mathrm{low,high}}&:=\sum_{N\gg 1}\sum_{\textbf{N}^1}\omega_N^2\langle N\rangle^{2s}\sup_{t\in(0,T)}\left\vert\int_{\R^2}\Pi_{\widetilde{a}_k}\big(\mathds{1}_{t,R}^{\mathrm{low}}P_{N_1}u,\mathds{1}_{t,R}^{\mathrm{high}}P_{N_2}u,P_{N_3}u,...,P_{N_k}u\big)P_{N_{k+1}}u\right\vert,
\\ \mathcal{G}_{1,R}^{\mathrm{low,low}}&:=\sum_{N\gg 1}\sum_{\textbf{N}^1}\omega_N^2\langle N\rangle^{2s}\sup_{t\in(0,T)}\left\vert \int_{\R^2}\Pi_{\widetilde{a}_k}\big(\mathds{1}_{t,R}^{\mathrm{low}}P_{N_1}u,\mathds{1}_{t,R}^{\mathrm{low}}P_{N_2}u,P_{N_3}u,...,P_{N_k}u\big)P_{N_{k+1}}u\right\vert,
\end{align*}
where $R$ stands for a large real number that shall be fixed later. For the sake of clarity we split the analysis into two steps. Before getting into it, let us recall the definition of the resonant relation for $(k+1)$-terms, which is given by
\[
\Omega_k(\xi_1,...,\xi_{k+1})=\xi_1^3+...+\xi_{k+1}^3.
\]
We emphasize that, as an abuse of notation, sometimes we also write $\Omega_k$ with only $k$ entries. However, in that case, $\Omega_k$ is given by $\xi_1^3+...+\xi_k^3-(\xi_1+...+\xi_k)^3$. Notice that both definition are equivalent due to the constraint imposed by $\Gamma^{k}$.

\medskip
 
\textbf{Step 2.1.1:} We begin by considering the case of $\mathcal{G}_{1,R}^{\mathrm{high}}$. The idea is to take advantage of the operator $\mathds{1}_{t,R}^{\mathrm{high}}$ by using Lemma \ref{lem_high_low_1}. In fact, by choosing \begin{align}\label{def_R}
R(N,N_1,...,N_{k+1}):=N_1N_3,
\end{align}
we can bound $\mathcal{G}_{1,R}^{\mathrm{high}}$ by using the first inequality in Lemma \ref{lem_high_low_1}, Lemma \ref{pseudo_holder_l2}, as well as Sobolev's embedding, in the following fashion
\begin{align*}
\mathcal{G}_{1,R}^{\mathrm{high}}&\lesssim_k \sum_{N \gg 1}\sum_{\textbf{N}^1}T^{1/4}\omega_N^2\langle N\rangle^{2s} \Vert \mathds{1}_{T,R}^{\mathrm{high}}\Vert_{L^{4/3}}\left\Vert\int_\R \Pi_{\widetilde{a}_k}\big(P_{N_1}u,...,P_{N_k}u\big)P_{N_{k+1}}u\right\Vert_{L^\infty}
\\ & \lesssim_k \sum_{N \gg 1}\sum_{\textbf{N}^1}T^{1/4}\omega_N^2\langle N\rangle^{2s}N_3R^{-3/4}\big\Vert P_{N_1}u\big\Vert_{L^\infty_tL^2_x}\Vert P_{N_2}u\Vert_{L^\infty_tL^2_x}\prod_{i=3}^{k+1} \Vert P_{N_i}u\Vert_{L^\infty_{t,x}}
\\ & \lesssim_k \sum_{N \gg 1}\sum_{\textbf{N}^1}T^{1/4}N_1^{-1/2}\Vert P_{N_1}u\Vert_{L^\infty_tH^s_\omega}\Vert P_{N_2}u\Vert_{L^\infty_tH^s_\omega}\prod_{i=3}^{k+1} \min\big\{N_i^{\frac{1}{2}},N_i^{-(0^+)}\big\}\Vert P_{N_i}u\Vert_{L^\infty_{t}H^{1/2^+}_x}
\\ & \lesssim_k T^{1/4}\Vert u\Vert_{L^\infty_tH^s_\omega}^{2}\Vert u\Vert_{L^\infty_tH^{1/2^+}_x}^{k-1},
\end{align*}
where we have used the fact that, thanks to the hypothesis on $\omega_N$ in Section \ref{function_spaces_section}, we have the inequality $\omega_N/\omega_{N_i}\lesssim 1$, $i=1,2$. To finish this first case, we point out that, thanks to the operator $\mathds{1}_{t,R}^\mathrm{high}$ acting on the factor $P_{N_2}u$, the same estimates also hold for $\mathcal{G}_{1,R}^{\mathrm{low,high}}$.

\medskip

\textbf{Step 2.1.2:} Now we consider the last term in the decomposition, that is, $\mathcal{G}_{1,R}^{\mathrm{low,low}}$. In fact, first of all, for the sake of notation let us define the following functional
\begin{align*}
&\mathcal{I}\big(u_1,...,u_{k+1}):=\sum_{N\gg 1}\sum_{\textbf{N}^1}\omega_N^2\langle N\rangle^{2s}\sup_{t\in(0,T)}\left\vert \int_{\R^2} \Pi_{\widetilde{a}_k}\big(u_1,...,u_k\big)u_{k+1}\right\vert.
\end{align*}
Then, we claim that, due to the relationship between the frequencies belonging to $\mathbf{N}^1$, the resonant relation satisfies \[
\big\vert\Omega_k(\xi_1,...,\xi_k)\big\vert\sim N_1N_2N_3.
\]
In fact, let us start by recalling that, due to the additional constraint imposed by $\Gamma^{k+1}$, we have the relation $\xi_1+...+\xi_{k+1}=0$. Then, by using the bound $N_3>2^9k^3N_4$, we deduce  \begin{align}
\big\vert\Omega_k(\xi_1,...,\xi_k)\big\vert&=\big\vert\xi_1^3+\xi_2^3+\xi_3^3+...+\xi_{k+1}^3\big\vert\nonumber
\\ & =\big\vert\xi_2^3+\xi_3^3-(\xi_2+\xi_3+...+\xi_{k+1})^3\big\vert+O(N_4^3)\nonumber
\\ & = 3\big\vert(\xi_2+\xi_3)\xi_2\xi_3\big\vert+O(N_1^2N_4)\nonumber
\\ & = 3\big\vert\xi_1\xi_2\xi_3\big\vert+O(N_1^2N_4)\sim N_1N_2N_3.\label{resonant_o_n1n3}
\end{align}
Therefore, taking advantage of the above relation, we can now decompose $\mathcal{G}_{1,R}^{\mathrm{low,low}}$ with respect to modulation variables in the following fashion
\begin{align*}
\big\vert\mathcal{G}_{1,R}^{\mathrm{low,low}}\big\vert &\leq \mathcal{I}\big(Q_{\gtrsim N^*}\mathds{1}_{t,R}^{\mathrm{low}}P_{N_1}u,\mathds{1}_{t,R}^{\mathrm{low}}P_{N_2}u,P_{N_3}u,...,P_{N_{k+1}}u\big)
\\ & \quad +\mathcal{I}\big(Q_{\ll N^*}\mathds{1}_{t,R}^{\mathrm{low}}P_{N_1}u,Q_{\gtrsim N^*}\mathds{1}_{t,R}^{\mathrm{low}}P_{N_2}u,P_{N_3}u,...,P_{N_{k+1}}u\big)
\\ & \quad +\mathcal{I}\big(Q_{\ll N^*}\mathds{1}_{t,R}^{\mathrm{low}}P_{N_1}u,Q_{\ll N^*}\mathds{1}_{t,R}^{\mathrm{low}}P_{N_2}u,Q_{\gtrsim N^*}P_{N_3}u,...,P_{N_{k+1}}u\big) 
\\ & \quad +...+\mathcal{I}\big(Q_{\ll N^*}\mathds{1}_{t,R}^{\mathrm{low}}P_{N_1}u,Q_{\ll N^*}\mathds{1}_{t,R}^{\mathrm{low}}P_{N_2}u,Q_{\ll N^*}P_{N_3}u,...,Q_{\gtrsim N^*}P_{N_{k+1}}u\big)
\\ & =: \mathcal{I}_1+...+\mathcal{I}_{k+1},
\end{align*}
where $N^*$ stands for $N^*:=N_1N_2N_3$. At this point it is important to notice that, since in this case we have $N_2\geq \tfrac{1}{8}N_1$, then we must also have that $N^*\gg N_1N_3=R$, what allows us to use the last inequality in Lemma \ref{lem_high_low_1}. Thus, bounding in a similar fashion as before, by using H\"older and Bernstein inequalities, as well as Lemma \ref{lem_high_low_1}, Lemma \ref{pseudo_holder_l2} and classical Bourgain estimates, we obtain
\begin{align*}
\mathcal{I}_{1}&\lesssim_k \sum_{N\gg 1}\sum_{\textbf{N}^1}\omega_N^2\langle N\rangle^{2s}N_3\big\Vert Q_{\gtrsim N^*}\mathds{1}_{T,R}^{\mathrm{low}}P_{N_1}u\big\Vert_{L^2_tL^2_x}\Vert \mathds{1}_{T,R}^{\mathrm{low}}P_{N_2}u\Vert_{L^2_tL^2_x}\prod_{i=3}^{k+1}\Vert P_{N_i}u\Vert_{L^\infty_tL^\infty_x}
\\ & \lesssim_k \sum_{N\gg 1}\sum_{\textbf{N}^1}N_2^{(-1)^+}\big\Vert Q_{\gtrsim N^*}\mathds{1}_{T,R}^{\mathrm{low}}P_{N_1}u\big\Vert_{X^{s-1,1}}\Vert \mathds{1}_{T,R}^{\mathrm{low}}\Vert_{L^2}\times 
\\ & \qquad \quad \times \Vert P_{N_2}u\Vert_{L^\infty_tH^s_x}\prod_{i=3}^{k+1}\min\{N_i^{1/2},N_i^{-(0^+)}\}\Vert P_{N_i}u\Vert_{L^\infty_tH^{1/2^+}_x}
\\ & \lesssim_k T^{1/2}\Vert u\Vert_{X^{s-1,1}}\Vert u\Vert_{L^\infty_tH^s_x}\Vert u\Vert_{L^\infty_tH^{1/2^+}_x}^{k-1}.
\end{align*}
Notice that, we have absorbed $\omega_N$ with $N_2^{-\varepsilon}$ thanks to the assumptions made in Section \ref{function_spaces_section}. Moreover, it is not difficult to see that, by following the same lines (up to trivial modifications), we can also bound $\mathcal{I}_2$, from where we obtain the same bound. On the other hand, to control $\mathcal{I}_3$ we use again both Lemma \ref{lem_high_low_1} and \ref{pseudo_holder_l2}, as well as H\"older and Bernstein inequalities, from where we obtain   \begin{align*}
\mathcal{I}_{3}&\lesssim_k \sum_{N\gg 1}\sum_{\textbf{N}^1}\omega_N^2\langle N\rangle^{2s}N_3\big\Vert Q_{\ll N^*}\mathds{1}_{T,R}^{\mathrm{low}}P_{N_1}u\big\Vert_{L^2_tL^2_x}\times
\\ & \qquad \quad  \times \Vert Q_{\ll N^*}\mathds{1}_{T,R}^{\mathrm{low}}P_{N_2}u\Vert_{L^\infty_tL^2_x}\Vert Q_{\gtrsim N^*}P_{N_3}u \Vert_{L^2_tL^\infty_x}\prod_{i=4}^{k+1}\Vert P_{N_i}u\Vert_{L^\infty_tL^\infty_x}
\\ & \lesssim_k \sum_{N\gg 1}\sum_{\textbf{N}^1}N_2^{(-1)^+}\Vert \mathds{1}_{T,R}^{\mathrm{low}}\Vert_{L^2}\big\Vert P_{N_1}u\big\Vert_{L^\infty_tH^s_x}\Vert P_{N_2}u\Vert_{L^\infty_tH^s_x}\min\{N_3^{1^-}N_1^{-1},N_3^{-(0^+)}\}\times 
\\ & \qquad \quad \times \Vert Q_{\gtrsim N^*}P_{N_3}u\Vert_{X^{(-1/2)^+,1}}\prod_{i=4}^{k+1}\min\{N_i^{1/2},N_i^{-(0^+)}\}\Vert P_{N_i}u\Vert_{L^\infty_tH^{1/2^+}_x}
\\ & \lesssim_k T^{1/2}\Vert u\Vert_{L^\infty_tH^s_x}^2\Vert u\Vert_{X^{(-1/2)^+,1}}\Vert u\Vert_{L^\infty_tH^{1/2^+}_x}^{k-2}.
\end{align*}
Notice that all the remaining cases $\mathcal{I}_i$, $i=4,...,k+1$, follow very similar lines to the latter case (up to trivial modifications), and they provide exactly the same bound. We omit the proof of these cases. 

\medskip

\textbf{Step 2.2:} Now we aim to deal with the region $N_1<8kN$. In fact, in this case it is enough to notice that, combining both, the hypotheses and the additional constraint imposed by $\Gamma^{k+1}$, we can write\footnote{Notice that this shall introduce a factor $k$ into the following estimates.} \[
N_1\in[\tfrac{1}{2}N,4kN] \quad \hbox{ and } \quad N_2\in[\tfrac{1}{2k}N_1,N_1].
\]
Therefore, up to a factor $k$, we deduce that $N_1\sim N$ and $N_2\sim N$. Hence, by using H\"older and Bernstein inequalities, as well as Lemma \ref{pseudo_holder_l2}, we get that \begin{align*}
\vert\mathcal{G}_2\vert& \lesssim_k \int_0^T\sum_{N\gg 1}\sum_{\mathbf{N}^2}\omega_N^2\langle N\rangle^{2s}\left\vert\int_{\Gamma^{k+1}}\widetilde{a}_k(\xi_1,...,\xi_{k+1})\prod_{i=1}^{k+1}\phi_{N_i}(\xi_i)\hat{u}(s,\xi_i)d\Gamma^{k+1}\right\vert ds
\\ & \lesssim_k \int_0^T\sum_{N\gg 1}\sum_{\mathbf{N}^2}\min\{N_3,N_3^{-(0^+)}\}\Vert P_{N_1}u(s,\cdot)\Vert_{H^s_\omega}\Vert P_{N_2}u(s,\cdot)\Vert_{H^s_\omega}\Vert J_x^{1/2^+}P_{N_3}u(s,\cdot)\Vert_{L^\infty_x}\times
\\ & \qquad \qquad \times \Vert J_x^{1/2^+}P_{N_4}u(s,\cdot)\Vert_{L^\infty_x}\prod_{i=5}^{k+1}\min\big\{N_i^{1/2},N_i^{-(0^+)}\big\}\Vert P_{N_i}u(s,\cdot)\Vert_{H^{1/2^+}_x}ds
\\ & \lesssim_k \Vert u\Vert_{L^\infty_TH^s_\omega}^{2}\big\Vert J_x^{1/2^+}u\big\Vert_{L^2_TL^\infty_x}^2 \Vert u\Vert_{L^\infty_TH^{1/2^+}_x}^{k-3}. 
\end{align*}
We emphasize that, in this case, to sum over $N\gg1$ we have used the fact that, for any function $f\in H^s_\omega(\R)$, the sequence $\{\Vert P_{2^n}f\Vert_{H^s_\omega}\}_{n\in\Z}$ belongs to $\ell^2(\Z)$.

\medskip

\textbf{Step 2.3:} Finally, it only remains to consider the case where $N_1\geq 8kN$. In fact, in this case, note that inequality $N_1\geq8kN$ implies in particular that $N_2\geq 4N$, and hence, we must also have $N_3\geq \tfrac{1}{2}N$, otherwise $\widetilde{a}_k\equiv0$. Then, we can proceed similarly as in the latter step but using the factor $N_3^{-}$ to sum over $N\gg 1$. Note that, in this case, we also have to use the fact that $\Vert P_{N_1}u(s,\cdot)\Vert_{H^s_\omega}$ and $\Vert P_{N_2}u(s,\cdot)\Vert_{H^s_\omega}$ are both square summable. The proof of Step $3$ is finished.

\medskip

\textbf{Step 3:} Finally, we consider the general case $\mathrm{I}_{u^k\Psi^m}$, where $k,m\geq 1$. As we have mentioned before, for small frequencies $N\lesssim 1$ we can directly bound the sum by simply using H\"older inequality as follows
\begin{align}\label{step3_small_freq}
\sum_{N\lesssim1}\omega_N^2\langle N\rangle^{2s}\sup_{t\in(0,T)}\left\vert\int_0^t\int_\R\partial_xP_N(u^k\Psi^m)P_Nu\right\vert\lesssim_k T\Vert u\Vert_{L^\infty_tH^s_\omega}^2\Vert u\Vert_{L^\infty_tH^{1/2^+}_x}^{k-1}\Vert\Psi\Vert_{L^\infty}^m.
\end{align}
Therefore, in the sequel we only consider the case where\footnote{Notice that this introduces another factor $k$ into inequality \eqref{step3_small_freq} comming from the use of $N\lesssim 1$ when controlling the operator $\partial_x$.} $N\gg k$. On the other hand, notice that, by using Plancherel Theorem we can rewrite the remaining quantity as \begin{align}\label{def_iukpsim}
\sum_{N\gg 1}\omega_N^2\langle N\rangle^{2s}\sup_{t\in(0,T)}\left\vert \int_0^t\int_{\Gamma^{k+2}} \mathbf{a}_k(\xi_1,...,\xi_{k+2})\hat{u}(\xi_1)...\hat{u}(\xi_{k+1})\widehat{\Psi^m}(\xi_{k+2})d\Gamma^{k+2}\right\vert,
\end{align}
where the symbol $\mathbf{a}_k(\xi_1,...,\xi_{k+2})$ is explicitly given by \[
\mathbf{a}_k(\xi_1,...,\xi_{k+2}):=i\phi_N^2(\xi_{k+1})\xi_{k+1}.
\]
In a similar spirit as for Steps $2$, in order to deal with this case we perform a symmetrization argument. Indeed, by symmetrizing the symbol we are lead to consider \[
\tilde{\mathbf{a}}_k(\xi_1,...,\xi_{k+2})=\big[\mathbf{a}_k(\xi_1,...,\xi_{k+2})\big]_{\mathrm{sym}}:=\dfrac{i}{k+1}\sum_{i=1}^{k+1}\phi_N^2(\xi_i)\xi_i.
\]
Then, by using frequency decomposition, the problem of bounding \eqref{def_iukpsim} is reduced to control the following quantity \begin{align}\label{def_int_step_7}
\sum_{N\gg1}\omega_N^2\langle N\rangle^{2s}\sup_{t\in(0,T)}\left\vert\int_0^t\sum_{N_1,...,N_{k+2}}\int_{\Gamma^{k+2}}\tilde{\mathbf{a}}_k(\xi_1,...,\xi_{k+2})\phi_{N_{k+2}}(\xi_{k+2})\widehat{\Psi^m}(\xi_{k+2})\prod_{i=1}^{k+1}\phi_{N_i}(\xi_i)\hat{u}(\xi_i)\right\vert.
\end{align}
Hence, by symmetry, without loss of generality from now on we assume\footnote{For the cases $k=1,2$ we only assume that $N_1\geq N_2$ and $N_1\geq N_2\geq N_3$, respectively. Once again, notice that these assumptions introduces a factor $k^4$ into the following estimates.} that $N_1\geq N_2\geq N_3\geq N_4=\max\{N_4,...,N_{k+1}\}$. We point out that, in this case, we consider\footnote{Since $N_{k+2}\in\mathbb{D}_{\mathrm{nh}}$, when $N_{k+2}=1$ we consider $\eta(\xi_{k+2})$ instead of $\phi_{N_{k+2}}(\xi_{k+2})$ in \eqref{def_int_step_7}, where $\eta(\cdot)$ is defined in \eqref{def_eta}.} $N_{k+2}\in \mathbb{D}_{\mathrm{nh}}$. Before going further notice that there is an important case that can be treated without any further decomposition. In fact, let us consider the region $8^{9}kN_{k+2}\geq N$. We begin by restricting ourselves to the case $N_2\geq 1$. Let us denote the set of indexes associated with all the above constraints by $\textbf{N}_{k+2}$. Then, by using Plancherel Theorem to go back to physical variables, taking advantage of the fact that $\Psi\in W^{(s+1)^+,\infty}_x$, we can control $\mathrm{I}_{u^k\Psi^m}$, in this region, by\footnote{Notice that here we obtain another factor $k^{s+1+}$ comming from the relation $8^9kN_{k+2}\geq N$.} \begin{align*}
&\sum_{N\gg 1}\omega_N^2\langle N\rangle^{2s}\sup_{t\in(0,T)}\left\vert\int_0^t\sum_{\textbf{N}_{k+2}}\int_{\Gamma^{k+2}}\tilde{\mathbf{a}}_k(\xi_1,...,\xi_{k+2})\phi_{N_{k+2}}(\xi_{k+2})\widehat{\Psi^m}(\xi_{k+2})\prod_{i=1}^{k+1}\phi_{N_i}(\xi_i)\hat{u}(\xi_i)\right\vert
\\ & \qquad \lesssim_k \int_0^T \sum_{N\gg 1}\sum_{\mathbf{N}_{k+2}}\omega_N^2\langle N\rangle^{2s}NN_{k+1}^{-(s+1)^+}\Vert P_{N_1}u(t',\cdot)\Vert_{L^2_x}\times 
\\ & \qquad \qquad \times \Vert P_{N_2}u(t',\cdot)\Vert_{L^2_x}\left\Vert P_{N_{k+2}}\big(\Psi^m(t',\cdot)\big)\right\Vert_{W^{(s+1)^+,\infty}_x}\prod_{i=3}^{k+1}\Vert P_{N_i}u(t',\cdot)\Vert_{L^\infty_x}dt'
\\ & \qquad \lesssim_k T\Vert u\Vert_{L^\infty_TH^s_\omega}^2\Vert u\Vert_{L^\infty_TH^{1/2^+}_x}^{k-1}\Vert \Psi\Vert_{L^\infty_TW^{(s+1)^+,\infty}_x}^m. 
\end{align*}
Here, we have absorbed one of the factors $\omega_N$ with $N_{k+1}^{-\varepsilon}$, thanks to the hypotheses made in Section \ref{function_spaces_section}.
Notice that, to deal with the case $N_2\leq 1$, it is enough to sum over $N_2$ inside the absolute value (before using H\"older inequality), so that we obtain a factor $\Vert P_{\lesssim1}u\Vert_{L^2}$ in the right-hand side (without any series in $N_2$). Therefore, in the sequel we can assume that $8^9kN_{k+2}< N$. Moreover, in this region we have $\phi_N(\xi_{k+2})\equiv0$, and hence we can write \[
\tilde{\mathbf{a}}_k(\xi_1,...,\xi_{k+2})=\dfrac{i}{k+1}\sum_{i=1}^{k+2}\phi_N^2(\xi_i)\xi_i.
\]
This is somehow important since it shall allow us to use Lemma \ref{pseudo_holder_l2} with no problems. 

\medskip

Now, in the same spirit as in Step $2$, we split the analysis into several cases, namely 
\begin{align*}
& 1) \ \, N_3< 8^8kN_{k+2},
\\ & 2) \ \, N_3\geq 8^8kN_{k+2} \, \hbox{ and } \, N_3\geq 2^9kN_4,
\\ & 3) \ \,  N_3\geq 8^8kN_{k+2} \, \hbox{ and } \, N_3<  2^9kN_4.
\end{align*}
Before getting into the details, let us introduce the notation for each of these regions. From now on we denote by $\mathcal{N}_1$, $\mathcal{N}_2$ and $\mathcal{N}_3$ the set of indexes associated with each of these regions\footnote{Recall we are also assuming that $N_1\geq N_2\geq N_3\geq N_4=\max\{N_4,...,N_{k+1}\}$ and $9^9kN_{k+2}<N$.}, and by $\mathrm{G}_1$, $\mathrm{G}_2$ and $\mathrm{G}_3$, the corresponding contribution of \eqref{def_int_step_7} associated with each of them, respectively.

\medskip

In the same spirit as in Step $2$, notice that all of the above regions require that $k\geq3$ to be well-defined. However, the case $k=1$ shall follows directly from the analysis we shall carry out to deal with the first region above, that is, the region $\mathcal{N}_1$. On the other hand, the case $k=2$ shall follow from the analysis associated with\footnote{In other words, roughly speaking, when $k=1$ we could think of $N_3$ as being equal to $0$, and hence the inequality of the first case is always satisfied, while when $k=2$ we could think of $N_4$ being zero, and hence we still have two cases, namely, $(1)$ and $(2)$.} cases $(1)$ and $(2)$ above.

\medskip

\textbf{Step 3.1:} We begin by studying the contribution of $I_{u^k\Psi^m}$ in the region $\mathcal{N}_{1}$. In fact, notice that, in this case, due to the hypotheses $\Psi\in W^{(s+1)^+,\infty}$ as well as the fact that $N_{k+2}\gtrsim N_3\geq \max\{N_4,...,N_{k+1}\}$, we can control the whole sum $\mathrm{G}_{1}$ directly from Lemma \ref{pseudo_holder_l2} and then using Bernstein inequalities, from where we get the bound \begin{align*}
\vert \mathrm{G}_1\vert &\lesssim_k \int_0^T \sum_{N\gg1}\sum_{\mathcal{N}_{1}}\omega_N^2\langle N\rangle^{2s}N_{k+2}(N_3....N_{k+1})^{1/2}\Vert P_{N_1}u(t',\cdot)\Vert_{L^2_x}\times 
\\ &\qquad \qquad \times  \Vert P_{N_2}u(t',\cdot)\Vert_{L^2_x} \Vert P_{N_{k+2}}(\Psi^m(t',\cdot))\Vert_{L^\infty_{x}} \prod_{i=3}^{k+1}\Vert  P_{N_i}u(t',\cdot)\Vert_{L^2_x}dt'
\\ &\lesssim_k \int_0^T \sum_{N\gg1}\sum_{\mathcal{N}_{1}}N_{k+2}^{-(0^+)}\Vert P_{N_1}u(t',\cdot)\Vert_{H^s_\omega} \Vert P_{N_2}u(t',\cdot)\Vert_{H^s_\omega}\times 
\\ &\qquad \qquad \times \Vert P_{N_{k+2}}(\Psi^m(t',\cdot))\Vert_{W^{1^+,\infty}_x} \prod_{i=3}^{k+1}\min\{N_i^{1/2},N_i^{-(0^+)}\}\Vert  P_{N_i}u(t',\cdot)\Vert_{H^{1/2^+}_x}dt'
\\ &  \lesssim_k T\Vert u\Vert_{L^\infty_TH^s_\omega}^{2}\Vert u\Vert_{L^\infty_TH^{1/2^+}_x}^{k-1}\Vert \Psi^m\Vert_{L^\infty_tW^{1^+,\infty}_x}.
\end{align*}
We point out that, in the estimates above, to sum over the indexes $N$, $N_1$ and $N_2$, we have used the fact that\footnote{Once again, this introduces a factor $k$ into the previous estimates. We stress that to avoid over repeated arguments, in the sequel we shall no longer point out these dependencies.} $N_1\in[ \tfrac{1}{2}N,4kN]$ and $N_2\in[\tfrac{1}{2k}N_1,N_1]$.

\medskip

\textbf{Step 3.2:} Now we seek to control the contribution of \eqref{def_int_step_7} in the region $\mathcal{N}_2.$ We aim to take advantage of classical Bourgain estimates. Similarly as in the previous steps, we begin using the decomposition given in \eqref{decomposition}, from where we infer that it is enough to control the following quantities \begin{align*}
\mathrm{G}_{2,R}^{\mathrm{high}}& :=\sum_{N\gg 1}\sum_{\mathcal{N}^2}\omega_N^2\langle N\rangle^{2s}\sup_{t\in(0,T)}\left\vert \int_{\R^2}\Pi_{\tilde{\mathbf{a}}_k}\big(\mathds{1}_{t,R}^{\mathrm{high}}P_{N_1}u,\mathds{1}_tP_{N_2}u,P_{N_3}u,...,P_{N_{k+1}}u\big)P_{N_{k+2}}\Psi\right\vert,
\\ \mathrm{G}_{2,R}^{\mathrm{low,high}}&:=\sum_{N\gg 1}\sum_{\mathcal{N}^2}\omega_N^2\langle N\rangle^{2s}\sup_{t\in(0,T)}\left\vert\int_{\R^2}\Pi_{\tilde{\mathbf{a}}_k}\big(\mathds{1}_{t,R}^{\mathrm{low}}P_{N_1}u,\mathds{1}_{t,R}^{\mathrm{high}}P_{N_2}u,P_{N_3}u,...,P_{N_{k+1}}u\big)P_{N_{k+2}}\Psi\right\vert,
\\ \mathrm{G}_{2,R}^{\mathrm{low,low}}&:=\sum_{N\gg 1}\sum_{\mathcal{N}^2}\omega_N^2\langle N\rangle^{2s}\sup_{t\in(0,T)}\left\vert \int_{\R^2}\Pi_{\tilde{\mathbf{a}}_k}\big(\mathds{1}_{t,R}^{\mathrm{low}}P_{N_1}u,\mathds{1}_{t,R}^{\mathrm{low}}P_{N_2}u,P_{N_3}u,...,P_{N_{k+1}}u\big)P_{N_{k+2}}\Psi\right\vert,
\end{align*}
where $R$ stands for a large real number to be fixed. We split the analysis into two steps. 

\medskip
 
\textbf{Step 3.2.1:} We start by  bounding $\mathrm{G}_{2,R}^{\mathrm{high}}$. We shall proceed in a similar fashion as in Step $2.1.1$. In fact, we define again $R(N,N_1,...,N_{k+2}):=N_1N_3$. Then, by using the first inequality in Lemma \ref{lem_high_low_1}, Lemma \ref{pseudo_holder_l2}, as well as Sobolev embedding, we obtain
\begin{align*}
\mathrm{G}_{2,R}^{\mathrm{high}}&\lesssim_k \sum_{N \gg 1}\sum_{\mathcal{N}^2}T^{1/4}\omega_N^2\langle N\rangle^{2s} \Vert \mathds{1}_{T,R}^{\mathrm{high}}\Vert_{L^{4/3}}\left\Vert\int_\R \Pi_{\tilde{\mathbf{a}}_k}\big(P_{N_1}u,...,P_{N_{k+1}}u\big)P_{N_{k+2}}\Psi\right\Vert_{L^\infty}
\\ & \lesssim_k \sum_{N \gg 1}\sum_{\mathcal{N}^2}T^{1/4}\omega_N^2\langle N\rangle^{2s}N_2^{-1/2} \big\Vert P_{N_1}u\big\Vert_{L^\infty_tL^2_x}\Vert P_{N_2}u\Vert_{L^\infty_tL^2_x}\Vert P_{N_{k+2}}\Psi\Vert_{L^\infty}\prod_{i=3}^{k+1} \Vert P_{N_i}u\Vert_{L^\infty_{t,x}}
\\ & \lesssim_k T^{1/4}\Vert u\Vert_{L^\infty_tH^s_\omega}^{2}\Vert u\Vert_{L^\infty_tH^{1/2^+}_x}^{k-1}\Vert \Psi\Vert_{L^\infty_TL^\infty_x}.
\end{align*}
To finish this first case, we point out that, thanks to the operator $\mathds{1}_{t,R}^\mathrm{high}$ acting on the factor $P_{N_2}u$, the same estimates also hold for $\mathrm{G}_{2,R}^{\mathrm{low,high}}$.

\medskip

\textbf{Step 3.2.2:} To conclude the proof of Step $3.2$ it only remains to consider $\mathcal{G}_{1,R}^{\mathrm{low,low}}$. As before, we begin by introducing some useful notation. We denote by $\mathcal{I}$ the functional given by
\begin{align*}
&\mathcal{I}\big(u_1,...,u_{k+2}):=\sum_{N\gg 1}\sum_{\mathcal{N}^2}\omega_N^2\langle N\rangle^{2s}\sup_{t\in(0,T)}\left\vert \int_{\R^2} \Pi_{\tilde{\mathbf{a}}_k}\big(u_1,...,u_{k+1}\big)u_{k+2}\right\vert.
\end{align*}
Now notice that, proceeding in the exact same fashion as in  \eqref{resonant_o_n1n3}, together with the fact that, in this case, $N_3\geq \max\{8^8kN_{k+2},2^9kN_4\}$, provides the relation \[
\big\vert\Omega_{k+1}(\xi_1,...,\xi_{k+2})\big\vert\sim N_1N_2N_3.
\]
Thus, in order to take advantage of the above relation, we decompose $\mathrm{G}_{2,R}^{\mathrm{low,low}}$ with respect to modulation variables in the following fashion
\begin{align*}
\big\vert\mathrm{G}_{2,R}^{\mathrm{low,low}}\big\vert &\leq \mathcal{I}\big(Q_{\gtrsim N^*}\mathds{1}_{t,R}^{\mathrm{low}}P_{N_1}u,\mathds{1}_{t,R}^{\mathrm{low}}P_{N_2}u,P_{N_3}u,...,P_{N_{k+1}}u,P_{N_{k+2}}\Psi\big)
\\ & \quad +\mathcal{I}\big(Q_{\ll N^*}\mathds{1}_{t,R}^{\mathrm{low}}P_{N_1}u,Q_{\gtrsim N^*}\mathds{1}_{t,R}^{\mathrm{low}}P_{N_2}u,P_{N_3}u,...,P_{N_{k+1}}u,P_{N_{k+2}}\Psi\big)
\\ & \quad +\mathcal{I}\big(Q_{\ll N^*}\mathds{1}_{t,R}^{\mathrm{low}}P_{N_1}u,Q_{\ll N^*}\mathds{1}_{t,R}^{\mathrm{low}}P_{N_2}u,Q_{\gtrsim N^*}P_{N_3}u,...,P_{N_{k+1}}u,P_{N_{k+2}}\Psi\big)
\\ & \quad +...
\\ & \quad +\mathcal{I}\big(Q_{\ll N^*}\mathds{1}_{t,R}^{\mathrm{low}}P_{N_1}u,Q_{\ll N^*}\mathds{1}_{t,R}^{\mathrm{low}}P_{N_2}u,Q_{\ll N^*}P_{N_3}u,...,Q_{\gtrsim N^*}P_{N_{k+1}}u,P_{N_{k+2}}\Psi\big) 
\\ & \quad +\mathcal{I}\big(Q_{\ll N^*}\mathds{1}_{t,R}^{\mathrm{low}}P_{N_1}u,Q_{\ll N^*}\mathds{1}_{t,R}^{\mathrm{low}}P_{N_2}u,Q_{\ll N^*}P_{N_3}u,...,Q_{\ll N^*}P_{N_{k+1}}u,P_{N_{k+2}}\Psi\big)
\\ & =: \mathcal{I}_1+...+\mathcal{I}_{k+2},
\end{align*}
where once again we are denoting by $N^*:=N_1N_2N_3$. At this point it is important to notice that, since in this case we have $N_2\geq \tfrac{1}{8}N\gg1$, then we must also have $N^*\gg N_1N_3=R$, what allows us to use the last inequality in Lemma \ref{lem_high_low_1}. Thus, bounding in a similar fashion as before, by using H\"older and Bernstein inequalities, as well as Lemma \ref{lem_high_low_1}, Lemma \ref{pseudo_holder_l2} and classical Bourgain estimates, we obtain
\begin{align*}
\mathcal{I}_{1}&\lesssim_k \sum_{N\gg 1}\sum_{\mathcal{N}^2}\omega_N^2\langle N\rangle^{2s}N_3\big\Vert Q_{\gtrsim N^*}\mathds{1}_{T,R}^{\mathrm{low}}P_{N_1}u\big\Vert_{L^2_tL^2_x}\Vert \mathds{1}_{T,R}^{\mathrm{low}}P_{N_2}u\Vert_{L^2_tL^2_x}\times
\\ & \qquad \qquad \times \Vert P_{N_{k+2}}(\Psi^m)\Vert_{L^\infty_{t,x}}\prod_{i=3}^{k+1}\Vert P_{N_i}u\Vert_{L^\infty_{t,x}}
\\ & \lesssim_k \sum_{N\gg 1}\sum_{\mathcal{N}^2}N_2^{(-1)^+}\big\Vert Q_{\gtrsim N^*}\mathds{1}_{T,R}^{\mathrm{low}}P_{N_1}u\big\Vert_{X^{s-1,1}}\Vert \mathds{1}_{T,R}^{\mathrm{low}}\Vert_{L^2}\Vert P_{N_2}u\Vert_{L^\infty_tH^s_x}\times 
\\ & \qquad \qquad \times \Vert P_{N_{k+2}}(\Psi^m)\Vert_{L^\infty_{t,x}} \prod_{i=3}^{k+1}\min\{N_i^{1/2},N_i^{-(0^+)}\}\Vert P_{N_i}u\Vert_{L^\infty_tH^{1/2^+}_x}
\\ & \lesssim_k T^{1/2}\Vert u\Vert_{X^{s-1,1}}\Vert u\Vert_{L^\infty_tH^s_x}\Vert u\Vert_{L^\infty_tH^{1/2^+}_x}^{k-1}\Vert\Psi\Vert_{L^\infty_{t,x}}^m.
\end{align*}
It is not difficult to see that, by following the same lines (up to trivial modifications), we can also bound $\mathcal{I}_2$, from where we obtain the same bound. On the other hand, to control $\mathcal{I}_3$ we use again both Lemma \ref{lem_high_low_1} and \ref{pseudo_holder_l2}, as well as H\"older and Bernstein inequalities, from where we obtain   \begin{align*}
\mathcal{I}_{3}&\lesssim_k \sum_{N\gg 1}\sum_{\mathcal{N}^2}\omega_N^2\langle N\rangle^{2s}N_3\big\Vert Q_{\ll N^*}\mathds{1}_{T,R}^{\mathrm{low}}P_{N_1}u\big\Vert_{L^2_tL^2_x}\Vert Q_{\ll N^*}\mathds{1}_{T,R}^{\mathrm{low}}P_{N_2}u\Vert_{L^\infty_tL^2_x}\times
\\ & \qquad \quad  \times \Vert Q_{\gtrsim N^*}P_{N_3}u \Vert_{L^2_tL^\infty_x}\Vert P_{N_{k+2}}(\Psi^m)\Vert_{L^\infty_{t,x}}\prod_{i=4}^{k+1}\Vert P_{N_i}u\Vert_{L^\infty_tL^\infty_x}
\\ & \lesssim_k \sum_{N\gg 1}\sum_{\mathcal{N}^2}N_2^{(-1)^+}\Vert \mathds{1}_{T,R}^{\mathrm{low}}\Vert_{L^2}\big\Vert P_{N_1}u\big\Vert_{L^\infty_tH^s_x}\Vert P_{N_2}u\Vert_{L^\infty_tH^s_x}\min\{N_3^{1^-}N_1^{-1},N_3^{-(0^+)}\}\times 
\\ & \qquad \quad \times \Vert Q_{\gtrsim N^*}P_{N_3}u\Vert_{X^{(-1/2)^+,1}}\Vert P_{N_{k+2}}(\Psi^m)\Vert_{L^\infty_{t,x}}\prod_{i=4}^{k+1}\min\{N_i^{1/2},N_i^{-(0^+)}\}\Vert P_{N_i}u\Vert_{L^\infty_tH^{1/2^+}_x}
\\ & \lesssim_k T^{1/2}\Vert u\Vert_{L^\infty_tH^s_x}^2\Vert u\Vert_{X^{(-1/2)^+,1}}\Vert u\Vert_{L^\infty_tH^{1/2^+}_x}^{k-2}\Vert \Psi\Vert_{L^\infty_{t,x}}^m.
\end{align*}
Notice that all the remaining cases $\mathcal{I}_i$, $i=4,...,k+1$, follow very similar lines to the latter case (up to trivial modifications), and hence we omit them. Finally, to control $\mathcal{I}_{k+2}$ notice that, since all factors $P_{N_i}u$ have an operator $Q_{\ll N^*}$ in front of them, then the factor $P_{N_{k+2}}(\Psi^m)$ is forced to be resonant, and hence in this case we can write $Q_{\gtrsim N^*}P_{N_{k+2}}\Psi=P_{N_{k+2}}(\Psi^m)$, otherwise $\mathcal{I}_{k+2}=0$ thanks to Lemma \ref{resonant_bourgain}. Moreover, notice also that, in the region $\mathcal{N}^2$ we have in particular that $N_1N_2N_3\gg N_{k+2}^3$, and hence we infer that $\vert \tau_{k+2}-\xi_{k+2}^3\vert\sim \vert\tau_{k+2}\vert$, and hence, we actually have $P_{N_{k+2}}(\Psi^m)=R_{\gtrsim N^*}P_{N_{k+2}}(\Psi^m)$. Therefore, by using Lemmas \ref{lem_high_low_1},  \ref{pseudo_holder_l2} as well as Bernstein inequality and the above properties, we obtain 
\begin{align*}
\mathcal{I}_{k+2}& \lesssim_k \sum_{N\gg 1}\sum_{\mathcal{N}^2}(N_1N_2)^{(-1)^+}\Vert \mathds{1}_{T,R}^{\mathrm{low}}\Vert_{L^2}^2\Vert P_{N_1}u\Vert_{L^\infty_tH^s_x}\Vert P_{N_2}u\Vert_{L^\infty_tH^s_x}\times
\\ & \qquad \times \Vert \partial_t R_{\gtrsim N^*}P_{N_{k+2}}(\Psi^m)\Vert_{L^\infty_{t,x}}\prod_{i=3}^{k+1}\min\{N_i^{1/2},N_i^{-(0^+)}\}\Vert P_{N_i}u\Vert_{L^\infty_{t}H^{1/2^+}_x}
\\ & \lesssim_k T\Vert u\Vert_{L^\infty_tH^s_x}^2\Vert u\Vert_{L^\infty_tH^{1/2^+}_x}^{k-1}\Vert\partial_t\Psi\Vert_{L^\infty_{t,x}}\Vert \Psi\Vert_{L^\infty_{t,x}}^{m-1}.
\end{align*}

\medskip

\textbf{Step 3.3:} To finish this step it only remains to consider $\mathrm{G_3}$. In this case it is enough to proceed in the same fashion as in Steps $2.2$ and $2.3$. In fact, noticing that, since \[
N_3\geq 8^8(k+1)N_{k+2}, \quad \hbox{then we have} \quad N_2\in[\tfrac{1}{2k}N_1,N_1].
\]
Therefore, by using H\"older and Bernstein inequalities, as well as Lemma \ref{pseudo_holder_l2}, we get that \begin{align*}
\vert\mathrm{G}_3\vert& \lesssim_k \int_0^T\sum_{N\gg 1}\sum_{\mathcal{N}^3}\langle N\rangle^{2s}(N_1N_2)^{-s}\min\{N_3,N_3^{-(0^+)}\}\Vert P_{N_1}u(t',\cdot)\Vert_{H^s_\omega}\times
\\ & \qquad \quad \times \Vert P_{N_2}u(t',\cdot)\Vert_{H^s_\omega}\Vert J_x^{1/2^+}P_{N_3}u(t',\cdot)\Vert_{L^\infty_x}\Vert J_x^{1/2^+}P_{N_4}u(t',\cdot)\Vert_{L^\infty_x}\times 
\\ & \qquad \quad \times \Vert P_{N_{k+2}}(\Psi^m(t',\cdot))\Vert_{L^\infty_{x}}\prod_{i=5}^{k+1}\min\big\{N_i^{1/2},N_i^{-(0^+)}\big\}\Vert P_{N_i}u(t',\cdot)\Vert_{H^{1/2^+}_x}dt'
\\ & \lesssim_k \Vert u\Vert_{L^\infty_TH^s_\omega}^{2}\big\Vert J_x^{1/2^+}u\big\Vert_{L^2_TL^\infty_x}^2 \Vert u\Vert_{L^\infty_TH^{1/2^+}_x}^{k-3}\Vert \Psi\Vert_{L^\infty_{t,x}}^m.
\end{align*}
We emphasize that, to sum over the indexes $N$, $N_1$, $N_2$ and $N_3$ in the case $N_3\ll N_2$, we have used the fact that $\Vert P_{N_1}u(s,\cdot)\Vert_{H^s}$ and $\Vert P_{N_2}u(s,\cdot)\Vert_{H^s}$ are both square summable, as well as the factor $N_3^{-}$, as in the proof of Steps $2.2$ and $2.3$. The proof of Step $3$ is complete.

\medskip

Now we explain how we control the contribution of the $L^2_TL^\infty_x$ terms. With this aim we use the Strichartz estimate \eqref{smoothing_eff} with $\delta=1$, from where we obtain \begin{align*}
\big\Vert J_x^{1/2^+}u\Vert_{L^2_TL^\infty_x}&\lesssim T^{1/4}\Vert u\Vert_{L^\infty_TH^{1/2^+}_x}+T^{3/4}\Vert \partial_t\Psi+\partial_x^3\Psi+\partial_xf(\Psi)\Vert_{L^\infty_TH^{-1/2^+}_x}
\\ & \quad + T^{3/4}\Vert u\Vert_{L^\infty_TH^{1/2^+}_x}\sum_{k=1}^\infty kc^k\vert a_k\vert \big(\Vert u\Vert_{L^\infty_TH^{1/2^+}_x}+\Vert\Psi\Vert_{L^\infty_TW^{1/2^+,\infty}_x}\big)^{k-1}.
\end{align*}
Gathering all the above estimate, and then using Lemma \ref{mst_basic_lemma}, we conclude the proof of the proposition.
\end{proof}

\smallskip 

\subsection{A priori estimates for the difference of two solutions}

In this subsection we seek to establish the key a priori estimate at the regularity level $s-1$ for the difference of two solutions. In the sequel we explicitly consider $\omega_N= 1$ for all $N\in\mathbb{D}$, and hence $H^s_\omega(\R)=H^s(\R)$.

\begin{prop}\label{prop_diff_sol}
Let $s>1/2$ and $T\in(0,2)$ both fixed. Consider $u,v\in L^\infty((0,T),H^s(\R))$ being two solutions to equation \eqref{gkdv_v} associated with an initial data $u_0,v_0\in H^s(\R)$. Then, the following inequality holds:
\begin{align*}
\Vert u-v\Vert_{L^\infty_TH^{s-1}_x}^2&\lesssim \Vert u_0-v_0\Vert_{H^{s-1}}^2 
\\ & \quad +T^{1/4}\Vert u-v\Vert_{L^\infty_TH^{s-1}_x}^2\mathcal{Q}^*\big(\Vert u\Vert_{L^\infty_TH^s_x},\Vert v\Vert_{L^\infty_TH^s_x},\Vert \Psi\Vert_{L^\infty_tW^{s+1,\infty}_x},\Vert \partial_t\Psi\Vert_{L^\infty_{t,x}}\big),
\end{align*}
where $\mathcal{Q}^*:\R^4\to\R_+$ is a smooth function.
\end{prop}

\begin{proof}
As before, in order to take advantage of Bourgain spaces, we have to extend the functions $u$ and $v$ from $(0,T)$ to the whole line $\R$. Hence, by using the extension operator we take extensions $\tilde{u}:=\rho_T[u]$ and $\tilde{v}:=\rho_T[v]$, supported in $(-2,2)$. For the sake of notation, we drop the tilde in the sequel. On the other hand, we point out that in the sequel we assume that $s\in(1/2,1]$. The case $s>1$ is simpler and follows very similar arguments.

\medskip

Now, let us denote by $w:=u-v$. Then $w(t,x)$ satisfies the equation \begin{align}\label{eq_w}
\partial_tw+\partial_x\big(\partial_x^2w+f(u+\Psi)-f(v+\Psi)\big)=0.
\end{align}
Now, we proceed as in the previous Proposition, taking the frequency projector $P_N$ to \eqref{eq_w} with $N>0$ dyadic, then taking the $L^2_x$-scalar product of the resulting equation against $P_Nw$ and multiplying the result by $\langle N\rangle^{2s-2}$. Finally, integrating in time on $(0,t)$ for $0<t<T$, and then applying Berstein inequality we are lead to \[
\Vert P_Nw(t)\Vert_{H^{s-1}_x}^2\lesssim \Vert P_Nw_0\Vert_{H^{s-1}}^2+\langle N\rangle^{2s-2}\sup_{t\in (0,T)}\left\vert \int_0^t\int_\R P_N\big(f(u+\Psi)-f(v+\Psi)\big)\partial_xP_Nw\right\vert.
\]
As before, we split the analysis in several steps, each of which, devoted to different ranges of $(k,i)$. Notice that the philosophy behind the estimates below, is the same one from the proof of the last proposition. However, since in this case we have more (different) functions, we must have several more cases as well, since we cannot order all the frequencies appearing in $P_N(u^{k-i}v^{i-1}\Psi w)$, as we did in the previous proposition when there was only $u^k$.

\medskip 

As we shall see, the estimates above do not depend of how many $u^{k-i}$ or $v^{i-1}$ we have. Thus, to simplify the notation we shall write
\begin{align*}
\sum_{N>0}\langle N\rangle^{2s-2}\sup_{t\in(0,T)}\left\vert\int_0^t\int_\R P_N(z_3...z_{k}\Psi^mw)\partial_xP_Nw\right\vert,
\end{align*}
for some $k\geq 3$ and $m\geq 0$, where each $z_i$ denotes either $u$ or $v$ (not necessarily all being the same).

\medskip

\textbf{Step 1:} Let us start by considering the case where we only have products of $z_i$, that is, where no power of $\Psi$ is involved. In other words, we seek to bound  the following quantity
\begin{align}\label{dif_k2_i1}
\sum_{N>0}\langle N\rangle^{2s-2}\sup_{t\in(0,T)}\left\vert\int_0^t\int_\R P_N(z_3...z_{k}w)\partial_xP_Nw\right\vert,
\end{align}
where $k\geq 3$. Before going further, we emphasize once again that in the sequel we assume $s\in(1/2,1]$. Then, in the same fashion as in the previous proposition, we begin by symmetrizing the underlying symbol in \eqref{dif_k2_i1}, which allows us to  reduce the problem to study the symbol
\begin{align}\label{break_d12}
d_k(\xi_1,...,\xi_{k})&:=\dfrac{i}{2}\phi_N^2(\xi_{1})\xi_{1}+\dfrac{i}{2}\phi_N^2(\xi_{2})\xi_{2},
\end{align}
where $\xi_{1}$ and $\xi_{2}$ denote the frequencies of each of the occurrences of $w$ in \eqref{dif_k2_i1} respectively. Hence, by frequency decomposition, it is enough to control the following quantity \begin{align}\label{def_int_frequenc_diff}
&\sum_{N>0}\langle N\rangle^{2s-2}\sup_{t\in (0,T)}\left\vert\int_0^t\sum_{N_1,...,N_{k}}\int_{\Gamma^{k}} d_k(\xi_1,...,\xi_{k})\prod_{i=1}^2\phi_{N_{i}}(\xi_i)\hat{w}(\xi_i)\prod_{j=3}^{k}\phi_{N_{j}}(\xi_j)\hat{z}_{j}(\xi_j)\right\vert.
\end{align}
Notice that, by symmetry, we can always assume $N_1\geq N_2$ and $N_3\geq N_4=\max\{N_4,...,N_{k+1}\}$. Now, for the sake of simplicity, let us denote by $\mathrm{I}$ the following functional \[
\mathrm{I}(N_1,...,N_k,u_1,...,u_k):=\int_{\Gamma^{k}} d_k(\xi_1,...,\xi_{k})\phi_{N_{1}}(\xi_1)u_{1}(\xi_1)...\phi_{N_{k}}(\xi_k)u_{k}(\xi_k)d\Gamma^{k}.
\]
Then, with this notation at hand, we define the set of admissible indexes 
\begin{align}\label{def_neqzero_indexes}
\mathbf{N}_k:=\mathbb{D}^k\setminus\big\{(N_1,...,N_k)\in\mathbb{D}^k: \forall (u_1,...,u_k)\in H^1(\R)^k,\ \mathrm{I}(N_1,...,N_k,u_1,...,u_k)=0\big\}.
\end{align}
Before going further let us rule out right away the case $N_2\lesssim1$. In fact, denoting by $\mathbf{N}_k^2:=\mathbf{N}_k\cap \{N_2<8^8k\}$, then, from Lemma \ref{pseudo_holder_l2}, Plancherel Theorem and H\"older and Bernstein inequalities we obtain
\begin{align*}
&\sum_{N>0}\langle N\rangle^{2s-2}\sup_{t\in (0,T)}\left\vert\int_0^t\sum_{\mathbf{N}_k^2}\int_{\Gamma^{k}} d_k(\xi_1,...,\xi_{k})\prod_{i=1}^2\phi_{N_{i}}(\xi_i)\hat{w}(t',\xi_i)\prod_{j=3}^{k}\phi_{N_{j}}(\xi_j)\hat{z}_{j}(t',\xi_j)\right\vert
\\ & \qquad \lesssim_k \int_0^T\sum_{N>0}\sum_{\mathbf{N}_k^2}\langle N\rangle^{2s-2}\left\vert\int_{\Gamma^k}d_k(\xi_1,...,\xi_{k})\prod_{i=1}^2\phi_{N_{i}}(\xi_i)\hat{w}(t',\xi_i)\prod_{j=3}^{k}\phi_{N_{j}}(\xi_j)\hat{z}_{j}(t',\xi_j)\right\vert
\\ & \qquad \lesssim_k \int_0^T\sum_{N>0}\sum_{\mathbf{N}_k^2}\langle N\rangle^{2s-2}\min\{N,N_3\}\Vert P_{N_1}w_1(t',\cdot)\Vert_{L^2_x}\times 
\\ & \qquad \qquad \times \Vert P_{N_2}w_2(t',\cdot)\Vert_{L^\infty_x}\Vert P_{N_3}z_3(t',\cdot)\Vert_{L^2_x}\prod_{j=4}^k\Vert P_{N_j}z_j(t',\cdot)\Vert_{L^\infty_x}dt'
\\ & \qquad \lesssim_k \int_0^T\sum_{N>0}\sum_{\mathbf{N}_k^2}\langle N\rangle^{2s-2}\langle N_1\rangle^{1-s}\langle N_3\rangle^{-s}N_2^{1/2}\min\{N,N_3\}\Vert P_{N_1}w_1(t',\cdot)\Vert_{H^{s-1}_x}\times 
\\ & \qquad \qquad \times \Vert P_{N_2}w_2(t',\cdot)\Vert_{H^{s-1}_x}\Vert P_{N_3}z_3(t',\cdot)\Vert_{H^{s}_x}\prod_{j=4}^k\min\{N_j^{1/2},N_j^{-}\}\Vert P_{N_j}z_j(t',\cdot)\Vert_{H^{1/2^+}_x}dt'  
\\ & \qquad \lesssim_k T\Vert w\Vert_{L^\infty_TH^{s-1}_x}^2\Vert z_3\Vert_{L^\infty_TH^s_x}\prod_{i=4}^k\Vert z_i\Vert_{L^\infty_TH^{1/2^+}_x}.
\end{align*}
Here, we have used the fact that, either $N\lesssim 1$, and then there is nothing to proof, or $N\gg 1$ and then, roughly speaking, we have $\{N_1\sim N \hbox{ and } N_3\gtrsim N_1\}$. In fact, if $N\geq 8^9k$, then, thanks to both facts, the explicit form of the symbol $d_k$ and the definition of $\Gamma^k$, we must have that $N_1\in[\tfrac{1}{2}N,2N]$ and $N_3\geq \tfrac{1}{2k}N_1$. Besides, in the case $N_4\ll N_3$ we have used the fact that $\Vert P_{N_1}w(t,\cdot)\Vert_{H^{s-1}_x}$ and $\Vert P_{N_3}z_3(t,\cdot)\Vert_{H^{s}_x}$ are both square summable. Notice lastly that, in particular, the previous computations allows us to rules out the case $N\leq 8^7k$.

\medskip

Now, in order to deal with the remaining region, we split the analysis into three cases, namely, 
\begin{align}\label{step_1_sets}
\mathbf{N}^1&:=\big\{(N_1,...,N_{k})\in \mathbb{D}^{k}: \, N_2\geq 8^8k,\, N_3< 2^9kN_4\big\}\cap \mathbf{N}_k,\nonumber
\\ \mathbf{N}^2&:=\big\{(N_1,...,N_{k})\in \mathbb{D}^{k}: \,N_2\geq 8^8k,\, N_3\geq 2^9kN_4,\,N_2< 2^9kN_4\big\}\cap \mathbf{N}_k
%
%
%
%
%
%
\\ \mathbf{N}^3&:=\big\{(N_1,...,N_{k})\in \mathbb{D}^{k}: \, N_2\geq 8^8k,\, N_3\geq 2^9kN_4,\,N_2\geq 2^9kN_4\big\}\cap \mathbf{N}_k.\nonumber
%
%
%
\end{align}
We denote the contribution of \eqref{dif_k2_i1} associated with each of these regions by $\mathrm{D}_1, \mathrm{D}_2,\mathrm{D}_3$, respectively. Notice that the case $k=3$ shall follow directly from the bound exposed for $\mathbf{N}^3$, while $\mathbf{N}^1$ and $\mathbf{N}^2$ only concern the cases $k\geq 4$.

\medskip

\textbf{Step 1.1:} Let us begin by considering the contribution of \eqref{def_int_frequenc_diff} associated with $\mathbf{N}^1$. We recall once again that $N>8^7k$. In fact, in this case we can proceed directly from Lemma \ref{pseudo_holder_l2}, Plancherel Theorem as well as H\"older and Bernstein inequalites, from where we obtain \begin{align*}
\mathrm{D}_1&\lesssim_k \int_0^T\sum_{N\gg 1}\sum_{\mathbf{N}^1}\langle N\rangle^{2s-2}\left\vert\int_{\Gamma^k}d_k(\xi_1,...,\xi_k)\prod_{i=1}^2\phi_{N_i}(\xi_i)\hat{w}(t',\xi_i)\prod_{j=3}^k\phi_{N_j}(\xi_j)\hat{z}_j(t',\xi_j)\right\vert dt'
\\ & \lesssim_k \int_0^T\sum_{N\gg 1}\sum_{\mathbf{N}^1}\langle N\rangle^{2s-2}\min\{N,N_3\}\Vert P_{N_1}w(t',\cdot)\Vert_{L^2_x}\Vert P_{N_2}w(t',\cdot)\Vert_{L^2_x}\times
\\ & \qquad \qquad \times \Vert P_{N_3}z_3(t',\cdot)\Vert_{L^\infty_x}\Vert P_{N_4}z_4(t',\cdot)\Vert_{L^\infty_x}\prod_{j=5}^k\Vert P_{N_j}z_j(t',\cdot)\Vert_{L^\infty_x}dt'
\\ & \lesssim_k \int_0^T\sum_{N\gg 1}\sum_{\mathbf{N}^1}\langle N\rangle^{s-1}\langle N_1\rangle^{1-s}\min\big\{N,N_3\big\}\Vert P_{N_1}w(t',\cdot)\Vert_{H^{s-1}_x}\Vert P_{N_2}w(t',\cdot)\Vert_{H^{s-1}_x}\times
\\ & \qquad \qquad \times \langle N_3\rangle^{-(1^+)} \Vert J_x^{1/2^+}P_{N_3}z_3(t',\cdot)\Vert_{L^\infty_x}\Vert J_x^{1/2^+}P_{N_4}z_4(t',\cdot)\Vert_{L^\infty_x}\prod_{j=5}^k\Vert P_{N_j}z_j(t',\cdot)\Vert_{L^\infty_x}dt'
\\ & \lesssim_k \Vert w\Vert_{L^\infty_TH^{s-1}_x}^2\big\Vert J_x^{1/2^+}z_3\Vert_{L^2_TL^\infty_x}\big\Vert J_x^{1/2^+}z_4\Vert_{L^2_TL^\infty_x}\prod_{j=5}^k\big\Vert z_j\Vert_{L^\infty_TH^{1/2^+}_x}.
\end{align*}
Here, we have used the fact that $s\in(1/2,1]$ so that $\langle N\rangle^{s-1}\langle N_2\rangle^{1-s}\lesssim 1$, and that, on $\mathbf{N}^1$, the following inequalities hold: \[
\langle N\rangle^s\langle N_1\rangle^{-s}\lesssim 1\quad \hbox{and}\quad \langle N\rangle^{-1}\langle N_1\rangle\min\{N,N_3\}\lesssim_k  N_3.
\]

\medskip

\textbf{Step 1.2:} Now we consider  the case of $\mathbf{N}^2$. Indeed, in a similar fashion as above, recalling that $s\in(1/2,1]$ and that $N>8^7k$, then, by using Lemma \ref{pseudo_holder_l2}, Plancherel Theorem as well as H\"older and Bernstein inequalites we infer that \begin{align*}
\mathrm{D}_2&\lesssim_k \int_0^T\sum_{N\gg 1}\sum_{\mathbf{N}^2}\langle N\rangle^{2s-2}\left\vert\int_{\Gamma^k}d_k(\xi_1,...,\xi_k)\prod_{i=1}^2\phi_{N_i}(\xi_i)\hat{w}(t',\xi_i)\prod_{j=3}^k\phi_{N_j}(\xi_j)\hat{z}_j(t',\xi_j)\right\vert dt'
\\ & \lesssim_k \int_0^T\sum_{N\gg 1}\sum_{\mathbf{N}^2}\langle N\rangle^{2s-2}\min\{N,N_3\}\Vert P_{N_1}w(t',\cdot)\Vert_{L^2_x}\Vert P_{N_2}w(t',\cdot)\Vert_{L^\infty_x}\times
\\ & \qquad \qquad \times \Vert P_{N_3}z_3(t',\cdot)\Vert_{L^2_x}\Vert P_{N_4}z_4(t',\cdot)\Vert_{L^\infty_x}\prod_{j=5}^k\Vert P_{N_j}z_j(t',\cdot)\Vert_{L^\infty_x}dt'
\\ & \lesssim_k \int_0^T\sum_{N\gg 1}\sum_{\mathbf{N}^2}\langle N\rangle^{2s-2}\langle N_1\rangle^{1-s}\langle N_3\rangle^{-s}N_2^{-}N_4^{-}\min\big\{N,N_3\big\}\Vert P_{N_1}w\Vert_{H^{s-1}_x}\times
\\ & \qquad \times \big\Vert J^{(-1/2)^+}_xP_{N_2}w\big\Vert_{L^\infty_x}\Vert P_{N_3}z_3\Vert_{H^s_x}\Vert J_x^{1/2^+}P_{N_4}z_4\Vert_{L^\infty_x}\prod_{j=5}^k\Vert P_{N_j}z_j\Vert_{L^\infty_x}dt'
\\ & \lesssim_k \Vert w\Vert_{L^\infty_TH^{s-1}_x}\big\Vert J_x^{(-1/2)^+}w\Vert_{L^2_TL^\infty_x}\big\Vert z_3\Vert_{L^\infty_TH^s_x}\big\Vert J_x^{1/2^+}z_4\big\Vert_{L^2_TL^\infty_x}\prod_{j=5}^k\big\Vert z_j\Vert_{L^\infty_TH^{1/2^+}_x}
\end{align*}
Here we have used that, due to our current hypotheses, we always have that $N_3\gtrsim N$. In fact, if $N_2\leq\tfrac{1}{16} N$, and since $N_4\ll N_3$, due to the explicit form of $d_{k}$ and the definition of $\Gamma^k$, we infer that   \[
N_1\in[\tfrac{1}{2}N,2N] \quad \hbox{and}\quad N_3\in[\tfrac{1}{4}N_1,4N_1].
\]
On the other hand, if $N_2\geq \tfrac{1}{8}N$, then due to the fact that $N_3\geq N_4\gtrsim N_2$, we obtain the desired relation.

\medskip

\textbf{Step 1.3:} Finally, we are ready to treat the remaining case in \eqref{step_1_sets}, that is, we now deal with the region $\mathbf{N}^3$. In order to do so, we begin using the decomposition given in \eqref{decomposition}, from where we infer that it is enough to control the following quantities \begin{align*}
\mathcal{D}_{3,R}^{\mathrm{high}}& :=\sum_{N\gg 1}\sum_{\textbf{N}^3}\langle N\rangle^{2s-2}\sup_{t\in(0,T)}\left\vert \int_{\R^2}\Pi_{d_k}\big(\mathds{1}_{t,R}^{\mathrm{high}}P_{N_1}w,\mathds{1}_tP_{N_2}w\big)P_{N_3}z_3...P_{N_{k}}z_k\right\vert,
\\ \mathcal{D}_{3,R}^{\mathrm{low,high}}&:=\sum_{N\gg 1}\sum_{\textbf{N}^3}\langle N\rangle^{2s-2}\sup_{t\in(0,T)}\left\vert\int_{\R^2}\Pi_{d_k}\big(\mathds{1}_{t,R}^{\mathrm{low}}P_{N_1}w,\mathds{1}_{t,R}^{\mathrm{high}}P_{N_2}w\big)P_{N_3}z_3...P_{N_{k}}z_k\right\vert,
\\ \mathcal{D}_{3,R}^{\mathrm{low,low}}&:=\sum_{N\gg 1}\sum_{\textbf{N}^3}\langle N\rangle^{2s-2}\sup_{t\in(0,T)}\left\vert \int_{\R^2}\Pi_{d_k}\big(\mathds{1}_{t,R}^{\mathrm{low}}P_{N_1}w,\mathds{1}_{t,R}^{\mathrm{low}}P_{N_2}w\big)P_{N_3}z_3...P_{N_k}z_k\right\vert,
\end{align*}
where $R$ stands for a large real number that shall be fixed later. For the sake of clarity we split the analysis into two steps. 

\medskip
 
\textbf{Step 1.3.1:} We begin by considering the case of $\mathcal{D}_{3,R}^{\mathrm{high}}$. Once again, the idea is to take advantage of the operator $\mathds{1}_{t,R}^{\mathrm{high}}$ by using Lemma \ref{lem_high_low_1}. In fact, thanks to our current hypothesis, we see that we can choose once again $R$ being equal to $R:=N_1N_3$. Moreover, due to the definition of $\Gamma^k$ again, since $\min\{N_2,N_3\}\geq 2^9kN_4$, we infer that $N_3\leq 8N_1$, and hence, either we have \begin{align}\label{possibilities_step13}
\big\{N_1\geq 16N \, \hbox{ and } \,  N_3\in[\tfrac{1}{2}N_1,2N_1]\big\} \quad \hbox{or}\quad \big\{N_1\in[\tfrac{1}{2}N,8N] \,\hbox{ and } \, N_3\leq 8N_1\big\}.
\end{align}
Therefore, we can then bound $\mathcal{G}_{3,R}^{\mathrm{high}}$ by using the first inequality in Lemma \ref{lem_high_low_1}, Lemmas  \ref{pseudo_holder_obvio_l2} and  \ref{pseudo_holder_l2}, as well as Sobolev's embedding in the following fashion
\begin{align*}
\mathcal{D}_{3,R}^{\mathrm{high}}&\lesssim_k \sum_{N \gg 1}\sum_{\textbf{N}^3}T^{1/4}\langle N\rangle^{2s-2} \Vert \mathds{1}_{T,R}^{\mathrm{high}}\Vert_{L^{4/3}}\left\Vert\int_\R \Pi_{d_k}\big(\mathds{1}_{t,R}^{\mathrm{high}}P_{N_1}w,\mathds{1}_tP_{N_2}w\big)P_{N_3}z_3...P_{N_{k}}z_k\right\Vert_{L^\infty_t}
\\ & \lesssim_k \sum_{N \gg 1}\sum_{\textbf{N}^3}T^{1/4}\langle N\rangle^{2s-2}\min\{N,N_3\}R^{-3/4}\times 
\\ & \qquad \qquad \times \Vert P_{N_1}w\Vert_{L^\infty_tL^2_x}\Vert P_{N_2}w\Vert_{L^\infty_tL^2_x}\prod_{j=3}^{k} N_j^{1/2}\Vert P_{N_j}z_j\Vert_{L^\infty_{t}L^2_x}
\\ & \lesssim_k \sum_{N \gg 1}\sum_{\textbf{N}^3}T^{1/4}\langle N\rangle^{s-1}\langle N_1\rangle^{1-s}(N_1N_3)^{-3/4}\min\{N,N_3\}\times
\\ & \qquad \qquad \times \Vert P_{N_1}u\Vert_{L^\infty_tH^{s-1}_x}\Vert P_{N_2}w\Vert_{L^\infty_tH^{s-1}_x}\prod_{j=3}^{k} \min\big\{N_j^{1/2},N_j^{-(0^+)}\big\}\Vert P_{N_j}z_j\Vert_{L^\infty_{t}H^{1/2^+}_x}
\\ & \lesssim_k T^{1/4}\Vert w\Vert_{L^\infty_tH^{s-1}_x}^{2}\prod_{j=3}^k\Vert z_j\Vert_{L^\infty_tH^{1/2^+}_x}.
\end{align*}
As before, notice that we have used the fact that $s\in(1/2,1]$ so that we have the following inequality $\langle N\rangle^{s-1}\langle N_2\rangle^{1-s}\lesssim 1$. To finish this first case, we point out that, thanks to the operator $\mathds{1}_{t,R}^\mathrm{high}$ acting on the factor $P_{N_2}w$, the same estimates also hold for $\mathcal{D}_{3,R}^{\mathrm{low,high}}$.

\medskip

\textbf{Step 1.3.2:} Now we consider the last term in the decomposition, that is, $\mathcal{D}_{3,R}^{\mathrm{low,low}}$. In fact, first of all, let us recall the notation introduced in the proof of the previous proposition (adapted to the current symbol)
\begin{align*}
&\mathcal{I}_k\big(u_1,...,u_{k}):=\sum_{N\gg 1}\sum_{\textbf{N}^3}\langle N\rangle^{2s-2}\sup_{t\in(0,T)}\left\vert \int_{\R^2} \Pi_{d_k}(u_1,u_{2})u_3...u_{k}\right\vert.
\end{align*}
Then, by using the the hypothesis $\min\{N_2,N_3\}\geq2^9kN_4$, following the same computations as in \eqref{resonant_o_n1n3}, we infer that the resonant relation satisfies \[
\big\vert\Omega_k(\xi_1,...,\xi_k)\big\vert\sim N_1N_2N_3.
\]
Thus, we are in a proper setting to take advantage of Bourgain spaces. In order to do so, we decompose $\mathcal{D}_{3,R}^{\mathrm{low,low}}$ with respect to modulation variables in the following fashion
\begin{align*}
\mathcal{D}_{3,R}^{\mathrm{low,low}} &\leq \mathcal{I}\big(Q_{\gtrsim N^*}\mathds{1}_{t,R}^{\mathrm{low}}P_{N_1}w,\mathds{1}_{t,R}^{\mathrm{low}}P_{N_2}w,P_{N_3}z_3,...,P_{N_{k}}z_k\big)
\\ & \quad +\mathcal{I}\big(Q_{\ll N^*}\mathds{1}_{t,R}^{\mathrm{low}}P_{N_1}w,Q_{\gtrsim N^*}\mathds{1}_{t,R}^{\mathrm{low}}P_{N_2}w,P_{N_3}z_3,...,P_{N_{k}}z_k\big)
\\ & \quad +\mathcal{I}\big(Q_{\ll N^*}\mathds{1}_{t,R}^{\mathrm{low}}P_{N_1}w,Q_{\ll N^*}\mathds{1}_{t,R}^{\mathrm{low}}P_{N_2}w,Q_{\gtrsim N^*}P_{N_3}z_3,...,P_{N_{k}}z_k\big) 
\\ & \quad +...+\mathcal{I}\big(Q_{\ll N^*}\mathds{1}_{t,R}^{\mathrm{low}}P_{N_1}w,Q_{\ll N^*}\mathds{1}_{t,R}^{\mathrm{low}}P_{N_2}w,Q_{\ll N^*}P_{N_3}z_3,...,Q_{\gtrsim N^*}P_{N_{k}}z_k\big)
\\ & =: \mathcal{I}_1+...+\mathcal{I}_{k},
\end{align*}
where $N^*$ stands for $N^*:=N_1N_2N_3$. At this point it is important to recall that, since $N_2\geq 8^9k$, then we also have that $N^*\gg N_1N_3=R$, what allows us to use the last inequality in Lemma \ref{lem_high_low_1}. Thus, bounding in a similar fashion as before, by using H\"older and Bernstein inequalities, as well as Lemmas \ref{lem_high_low_1}, \ref{pseudo_holder_obvio_l2} and \ref{pseudo_holder_l2}, and classical Bourgain estimates, we obtain
\begin{align*}
\mathcal{I}_{k}&\lesssim_k \sum_{N\gg 1}\sum_{\textbf{N}^3}\langle N\rangle^{2s-2}\min\{N,N_3\}\big\Vert Q_{\gtrsim N^*}\mathds{1}_{T,R}^{\mathrm{low}}P_{N_1}w\big\Vert_{L^2_tL^2_x}\times 
\\ & \qquad \qquad \times \Vert \mathds{1}_{T,R}^{\mathrm{low}}P_{N_2}w\Vert_{L^2_tL^2_x}\prod_{j=3}^{k}N_j^{1/2}\Vert P_{N_j}z_j\Vert_{L^\infty_{t}L^2_x}
\\ & \lesssim_k \sum_{N\gg 1}\sum_{\mathbf{N}^3}\langle N\rangle^{s-1}\langle N_1\rangle^{1-s}N_2^{-1}N_3^{-1}\min\{N,N_3\}\big\Vert Q_{\gtrsim N^*}\mathds{1}_{T,R}^{\mathrm{low}}P_{N_1}w\big\Vert_{X^{s-2,1}}\times 
\\ & \qquad \quad \times \Vert \mathds{1}_{T,R}^{\mathrm{low}}\Vert_{L^2}\Vert P_{N_2}w\Vert_{L^\infty_tH^{s-1}_x}\prod_{j=3}^{k+1}\min\{N_j^{1/2},N_j^{-(0^+)}\}\Vert P_{N_j}z_j\Vert_{L^\infty_tH^{1/2^+}_x}
\\ & \lesssim_k T^{1/2}\Vert w\Vert_{X^{s-2,1}}\Vert w\Vert_{L^\infty_tH^{s-1}_x}\prod_{j=3}^k\Vert z_j\Vert_{L^\infty_tH^{1/2^+}_x},
\end{align*}
where we have used again \eqref{possibilities_step13}, the fact that $N>8^7k$ and that $s\in(1/2,1]$, so that we have\footnote{We point out that, in order to avoid over-repeated sentences, in what follows we shall no longer emphasize that $s\in (1/2,1]$ and that $N>8^7k$.} $\langle N\rangle^{s-1}\langle N_2\rangle^{1-s}\lesssim 1$. Moreover, it is not difficult to see that, by following the same lines (up to trivial modifications), we can also bound $\mathcal{I}_2$, from where we obtain exactly the same bound. On the other hand, to control $\mathcal{I}_3$ we use again Lemmas \ref{lem_high_low_1}, \ref{pseudo_holder_obvio_l2} and \ref{pseudo_holder_l2}, as well as H\"older and Bernstein inequalities, from where we obtain   \begin{align*}
\mathcal{I}_{3}&\lesssim_k \sum_{N\gg 1}\sum_{\textbf{N}^3}\langle N\rangle^{2s-2}N_3^{1/2}\min\{N,N_3\}\big\Vert Q_{\ll N^*}\mathds{1}_{T,R}^{\mathrm{low}}P_{N_1}w\big\Vert_{L^2_tL^2_x}\times
\\ & \qquad \quad  \times \Vert Q_{\ll N^*}\mathds{1}_{T,R}^{\mathrm{low}}P_{N_2}w\Vert_{L^\infty_tL^2_x}\Vert Q_{\gtrsim N^*}P_{N_3}z_3 \Vert_{L^2_tL^2_x}\prod_{j=4}^{k}N_j^{1/2}\Vert P_{N_j}z_j\Vert_{L^\infty_tL^2_x}
\\ & \lesssim_k \sum_{N\gg 1}\sum_{\mathbf{N}^3}\langle N\rangle^{s-1}\langle N_1\rangle^{-s}\langle N_3\rangle^{1/2^-}N_2^{-1}N_3^{-1/2}\min\{N,N_3\}\Vert \mathds{1}_{T,R}^{\mathrm{low}}\Vert_{L^2}\big\Vert P_{N_1}w\big\Vert_{L^\infty_tH^{s-1}_x}\times 
\\ & \qquad \quad \times \Vert P_{N_2}w\Vert_{L^\infty_tH^{s-1}_x}\Vert Q_{\gtrsim N^*}P_{N_3}z_3\Vert_{X^{(-1/2)^+,1}}\prod_{j=4}^{k+1}\min\{N_j^{1/2},N_j^{-(0^+)}\}\Vert P_{N_j}z_j\Vert_{L^\infty_tH^{1/2^+}_x}
\\ & \lesssim_k T^{1/2}\Vert w\Vert_{L^\infty_tH^{s-1}_x}^2\Vert z_3\Vert_{X^{(-1/2)^+,1}}\prod_{j=4}^k\Vert z_j\Vert_{L^\infty_tH^{1/2^+}_x}.
\end{align*}
Notice that all the remaining cases $\mathcal{I}_i$, $i=4,...,k$, follow very similar lines to the latter case above (up to trivial modifications), and they provide exactly the same bound. Hence, in order to avoid over-repeated computations, we omit the proof of these cases. 

\medskip

\textbf{Step 2:} Now we seek to bound the case where we only have powers of $\Psi$. In particular, no $z_i$ is involved. More concretely, in this step we seek to study the following quantity \begin{align}\label{psideuxieme}
\sum_{N>0}\langle N\rangle^{2s-2}\sup_{t\in (0,T)}\left\vert\int_0^t\sum_{N_1,N_2,N_{3}}\int_\R \Pi_{d_3}\big(P_{N_{1}}w,P_{N_{2}}w\big)P_{N_3}(\Psi^m)\right\vert,
\end{align}
where $d_3(\xi_1,\xi_2,\xi_3)$ stands for the symbol given in \eqref{break_d12}, where $\xi_{1}$ and $\xi_{2}$ denote the frequencies of each of the occurrences of $w$ in \eqref{psideuxieme} respectively. 

\medskip

Now notice that, by symmetry, we can always assume $N_1\geq N_2$. Moreover, it is not difficult to see that, due to the additional constraint\footnote{We recall that, with this phrase we are referring to the condition $\xi_1+\xi_2+\xi_3=0$.} given by $\Gamma^3$, in this case we have that $\max\{N_2,N_3\}\geq \tfrac{1}{4}N_1$ and $N_3\leq 8N_1$, and hence, either we have $N_2\sim N_1$ or $N_3\sim N_1$, or both. By similar reasons, if $N_3\geq 16N$, then we must have $N_1\in[\tfrac{1}{4}N_3,4N_3]$, otherwise the inner integral in \eqref{psideuxieme} vanishes.

\medskip

On the other hand, following the same lines of the previous step, we can directly bound the case $\mathbf{N}_3^2:=\mathbf{N}_3\cap\{N_2\lesssim1\}$, where $\mathbf{N}_3$ is the set defined in \eqref{def_neqzero_indexes}. In fact, from Lemma \ref{pseudo_holder_l2}, Plancherel Theorem and H\"older inequality we obtain \begin{align*}
&\sum_{N>0}\langle N\rangle^{2s-2}\sup_{t\in(0,T)}\bigg\vert\int_0^t\sum_{\mathbf{N}_3^2}\int_{\Gamma^3}d(\xi_1,\xi_2,\xi_3)\phi_{N_1}(\xi_1)\hat{w}(\xi_1)\phi_{N_2}(\xi_2)\hat{w}(\xi_2)\phi_{N_3}(\xi_3)\widehat{\Psi^m}(\xi_3)\bigg\vert
\\ & \qquad \lesssim_k T\Vert w\Vert_{L^\infty_TH^{s-1}_x}^2\Vert \Psi^m\Vert_{L^\infty_TW^{1^+,\infty}_x},
\end{align*}
and hence in the sequel we can assume that $N_2\geq 8^8k$. Now, let us consider the region $\mathbf{N}_\gg:=\mathbf{N}_3\cap\{N_3\geq 16N\}$. Then, by using Plancherel Theorem and then H\"older and Bernstein inequalities we get  \begin{align*}
&\sum_{N\gg 1}\langle N\rangle^{2s-2}\sup_{t\in(0,T)}\left\vert\int_0^t\sum_{\mathbf{N}_\gg}\int_{\Gamma^3}d(\xi_1,\xi_2,\xi_3)\phi_{N_1}(\xi_1)\hat{w}(\xi_1)\phi_{N_2}(\xi_2)\hat{w}(\xi_2)\phi_{N_3}(\xi_3)\widehat{\Psi^m}(\xi_3)\right\vert
\\ & \qquad \lesssim_k \int_0^T\sum_{N\gg 1}\sum_{\mathbf{N}_\gg}\langle N_3\rangle^{-(0^+)}\Vert P_{N_1}w(t',\cdot)\Vert_{H^{s-1}_x}\Vert P_{N_2}w(t',\cdot)\Vert_{H^{s-1}_x}\Vert P_{N_3}(\Psi^m(t',\cdot))\Vert_{W^{1^+,\infty}_x}dt'
\\ & \qquad \lesssim_k T\Vert w\Vert_{L^\infty_TH^{s-1}_x}^2\Vert \Psi^m\Vert_{L^\infty_TW^{1^+,\infty}_x}.
\end{align*}
Hence, from now on we assume that $N_3\leq 8N$, which in turn forces $N_1\in[\tfrac{1}{4}N,32N]$ thanks to the additional constraint given by $\Gamma^3$.  Then, denoting this remaining region by $\mathbf{N}_\leq$, recalling that $\max\{N_2,N_3\}\geq \tfrac{1}{4}N_1$, we can bound the remaining portion of \eqref{psideuxieme} from Lemma \ref{pseudo_holder_l2} and Bernstein inequality as follows
\begin{align*}
&\sum_{N\gg 1}\langle N\rangle^{2s-2}\sup_{t\in(0,T)}\left\vert\int_0^t\sum_{\mathbf{N}_\leq}\int_{\Gamma^3}d(\xi_1,\xi_2,\xi_3)\phi_{N_1}(\xi_1)\hat{w}(\xi_1)\phi_{N_2}(\xi_2)\hat{w}(\xi_2)\phi_{N_3}(\xi_3)\widehat{\Psi^m}(\xi_3)\right\vert
\\ & \qquad \lesssim_k \int_0^T\sum_{N\gg 1}\sum_{\mathbf{N}_\leq}\min\{N_3,N_3^{-(0^+)}\}\Vert P_{N_1}w\Vert_{H^{s-1}_x}\Vert P_{N_2}w\Vert_{H^{s-1}_x}\Vert P_{N_3}(\Psi^m)\Vert_{W^{1^+,\infty}_x}
\\ & \qquad \lesssim_k T\Vert w\Vert_{L^\infty_TH^{s-1}_x}^2\Vert \Psi^m\Vert_{L^\infty_TW^{1^+,\infty}_x}
\end{align*}
where in this case, to sum over the region $N_3\ll N$, we have used the fact that $\Vert P_{N_1}w(t,\cdot)\Vert_{H^{s-1}_x}$ and $\Vert P_{N_2}w(t,\cdot)\Vert_{H^{s-1}_x}$ are both square summable.

\medskip

\textbf{Step 3:} Finally, it only remains to bound the ``crossed terms''. More specifically, in this step we aim to estimate the contribution of the following quantity
\begin{align}\label{def_int_troisieme_psi}
\sum_{N>0}\langle N\rangle^{2s-2}\sup_{t\in (0,T)}\left\vert\int_0^t\sum_{N_1,...,N_{k+1}}\int_\R \Pi_{d_{k+1}}\big(P_{N_{1}}w,P_{N_{2}}w\big)P_{N_3}z_3...P_{N_k}z_kP_{N_{k+1}}(\Psi^m)\right\vert,
\end{align}
where we assume $k\geq 3$ and $d_{k+1}$ is the symbol given in \eqref{break_d12}. We emphasize that, as in the previous proposition, we consider\footnote{Hence, when $N_{k+1}=1$ we consider $\eta(\xi_{k+1})$ instead of $\phi_{N_{k+1}}(\xi_{k+1})$ in \eqref{def_int_troisieme_psi}.} $N_{k+1}\in\mathbb{D}_{\mathrm{nh}}$. Now notice that, by symmetry, we can always assume that $N_1\geq N_2$ and $N_3\geq N_4=\max\{N_4,...,N_k\}$. Moreover, in contrast with Step $1$, in this case, by using either Lemma \ref{pseudo_holder_l2} or Plancherel Theorem together with H\"older inequality, the factor coming from the symbol $d_{k+1}$ shall be of order \[
\min\big\{N,N_{\max}\big\} \quad \hbox{instead of}\quad \min\{N,N_3\},
\]
as in the previous case, where we have adopted the notation $N_{\max}:=\max\{N_3,N_{k+1}\}$. On the other hand, due to the definition of $\Gamma^{k+1}$ we infer that,\[
\max\big\{N_2,N_3,N_{k+1}\big\}\geq \tfrac{1}{2k}N_1.
\]
In fact, more generally we have, \[
\max\big\{N_1,N_2,N_3,N_4,N_{k+1}\big\}\setminus\{\max\{N_1,N_3,N_{k+1}\}\}\geq \tfrac{1}{2k}\max\{N_1,N_3,N_{k+1}\}.
\]
Therefore, roughly speaking, the two largest frequencies are always equivalent (up to a factor depending on $k$). Now, let us start by ruling out the case $N_2< 9^9k$. Indeed, denoting by $\mathbf{N}_\lesssim^{k}:=\mathbf{N}_{k+1}\cap\{N_2\lesssim1\}$, from Lemma \ref{pseudo_holder_l2}, Plancherel Theorem and H\"older inequality we get
\begin{align*}
&\sum_{N>0}\langle N\rangle^{2s-2}\sup_{t\in (0,T)}\bigg\vert\int_0^t\sum_{\mathbf{N}_{\lesssim}^k}\int_\R \Pi_{d_{k+1}}\big(P_{N_{1}}w,P_{N_{2}}w\big)P_{N_3}z_3...P_{N_k}z_kP_{N_{k+1}}(\Psi^m)\bigg\vert
\\ & \quad \lesssim_k \int_0^T\sum_{N>0}\sum_{\mathbf{N}_{\lesssim}^k}\langle N\rangle^{2s-2}\langle N_1\rangle^{1-s}\langle N_3\rangle^{-s}N_2^{1/2}N_{k+1}^{-(1^+)}\min\{N,N_{\max}\}\Vert P_{N_1}w\Vert_{H^{s-1}_x}\times 
\\ & \quad \qquad \times \Vert P_{N_2}w\Vert_{H^{s-1}_x}\Vert P_{N_3}z_3\Vert_{H^s_x}\Vert P_{N_{k+1}}(\Psi^m)\Vert_{W^{1^+,\infty}_x}\prod_{j=4}^k\min\{N_j^{1/2},N_j^-\}\Vert P_{N_j}z_j\Vert_{H^{1/2^+}_x}dt'
\\ & \quad \lesssim_k T\Vert w\Vert_{L^\infty_TH^{s-1}_x}^2\Vert z_3\Vert_{L^\infty_TH^s_x}\Vert\Psi^m\Vert_{L^\infty_TW^{1^+,\infty}_x}\prod_{j=4}^k\Vert z_j\Vert_{L^\infty_TH^{1/2^+}_x},
\end{align*}
where we have used the fact that, if $N\gg1$, then $N_1\in[\tfrac{1}{2}N,2N]$, as well as the fact that $\Vert P_{N_1}w(t,\cdot)\Vert_{H^{s-1}_x}$ and $\Vert P_{N_3}z_3(t,\cdot)\Vert_{H^{s}_x}$ are both square summable. Once again, notice that the previous bound allows us to assume in the sequel that $N\geq9^8k$. Now, it is not difficult to see that, in the remaining region, we can further assume that $N_3\geq 8^8k$. In fact, let us assume that if $N_3< 8^8k$. Then, in this case, either we have $N_{k+1}\sim N_1$ or $\{N_{k+1}\ll N_1\hbox{ and }N_1\sim N_2\sim N\}$. Hence, denoting this region by $\mathbf{N}_{\leq}^{k,3}$, then, by using H\"older and Bernstein inequalities we have 
\begin{align*}
&\sum_{N>0}\langle N\rangle^{2s-2}\sup_{t\in (0,T)}\bigg\vert\int_0^t\sum_{\mathbf{N}_{\leq}^{k,3}}\int_\R \Pi_{d_{k+1}}\big(P_{N_{1}}w,P_{N_{2}}w\big)P_{N_3}z_3...P_{N_k}z_kP_{N_{k+1}}(\Psi^m)\bigg\vert
\\ & \qquad \quad  \lesssim_k \int_0^T\sum_{N>0}\sum_{\mathbf{N}_{\leq}^{k,3}}\langle N\rangle^{2s-2}\langle N_1\rangle^{1-s}\langle N_2\rangle^{1-s}N_{k+1}^{-(1^+)}\min\{N,N_{\max}\}\Vert P_{N_1}w\Vert_{H^{s-1}_x}\times 
\\ & \qquad \qquad \qquad \times \Vert P_{N_2}w\Vert_{H^{s-1}_x}\Vert P_{N_{k+1}}(\Psi^m)\Vert_{W^{1^+,\infty}_x}\prod_{j=3}^k\min\{N_j^{1/2},N_j^-\}\Vert P_{N_j}z_j\Vert_{H^{1/2^+}_x}dt'
\\ & \qquad \quad \lesssim_k T\Vert w\Vert_{L^\infty_TH^{s-1}_x}^2\Vert\Psi^m\Vert_{L^\infty_TW^{1^+,\infty}_x}\prod_{j=3}^k\Vert z_j\Vert_{L^\infty_TH^{1/2^+}_x},
\end{align*}
where, to sum over the region $\{N_{k+1}\ll N\hbox{ and }N_1\sim N_2\sim N\}$, we have used the fact that $\Vert P_{N_1}w(s,\cdot)\Vert_{H^{s-1}_x}$ and $\Vert P_{N_2}w(s,\cdot)\Vert_{H^{s-1}_x}$ are both square summable. Moreover, there is another important case that can be directly treated. Let us define the set \[
\mathbf{N}_\gg^{k,1}:=\mathbf{N}_{k+1}\cap\{N_2\geq 9^9k\}\cap\{N_3\geq 8^8k\}\cap\big\{2^9kN_{k+1}> \min\{N_1,N_3\}\big\}.
\]
The latter constraint implies, up to a factor $k$, that $\min\{N_1,N_3\}\lesssim_k N_{k+1}$. Then, proceeding in a similar fashion as above, noticing that $\min\{N,N_{\max}\}\lesssim_k\min\{N,N_{k+1}\} $, from H\"older and Bernstein inequalities we obtain
\begin{align*}
&\sum_{N\gg 1}\langle N\rangle^{2s-2}\sup_{t\in (0,T)}\bigg\vert\int_0^t\sum_{\mathbf{N}_{\gg}^{k,1}}\int_\R \Pi_{d_{k+1}}\big(P_{N_{1}}w,P_{N_{2}}w\big)P_{N_3}z_3...P_{N_k}z_kP_{N_{k+1}}(\Psi^m)\bigg\vert
\\ & \quad \lesssim_k \int_0^T\sum_{N\gg1}\sum_{\mathbf{N}_{\gg}^{k,1}} \langle N\rangle^{2s-2}\langle N_1\rangle^{1-s}\langle N_2\rangle^{1-s}N_3^{-(0^+)}N_{k+1}^{-(1^+)}\min\{N,N_{\max}\}\Vert P_{N_1}w\Vert_{H^{s-1}_x}\times 
\\ & \ \qquad \times \Vert P_{N_2}w\Vert_{H^{s-1}_x}\Vert P_{N_3}z_3\Vert_{H^{1/2^+}_x}\Vert P_{N_{k+1}}(\Psi^m)\Vert_{W^{1^+,\infty}_x}\prod_{j=4}^k\min\{N_j^{1/2},N_j^-\}\Vert P_{N_j}z_j\Vert_{H^{1/2^+}_x}dt'
\\ & \quad \lesssim_k T\Vert w\Vert_{L^\infty_TH^{s-1}_x}^2\Vert\Psi^m\Vert_{L^\infty_TW^{1^+,\infty}_x}\prod_{j=3}^k\Vert z_j\Vert_{L^\infty_TH^{1/2^+}_x}.
\end{align*}
Having dealt with the above cases, we can now easily deal with the region $2^9kN_{k+1}>N_2$. Indeed, denoting this region by $\mathbf{N}_\gg^{k,2}$, from H\"older and Bernstein inequalities we see that \begin{align*}
&\sum_{N\gg 1}\langle N\rangle^{2s-2}\sup_{t\in (0,T)}\bigg\vert\int_0^t\sum_{\mathbf{N}_{\gg}^{k,2}}\int_\R \Pi_{d_{k+1}}\big(P_{N_{1}}w,P_{N_{2}}w\big)P_{N_3}z_3...P_{N_k}z_kP_{N_{k+1}}(\Psi^m)\bigg\vert
\\ & \quad \lesssim_k \int_0^T\sum_{N\gg1}\sum_{\mathbf{N}_{\gg}^{k,2}} \langle N\rangle^{2s-2}\langle N_1\rangle^{1-s} \langle N_2\rangle^{1-s}N_2^{1/2}N_3^{-s}N_{k+1}^{-(1^+)}\min\{N,N_{3}\}\Vert P_{N_1}w\Vert_{H^{s-1}_x}\times 
\\ & \ \qquad \times \Vert P_{N_2}w\Vert_{H^{s-1}_x}\Vert P_{N_3}z_3\Vert_{H^{s}_x}\Vert P_{N_{k+1}}(\Psi^m)\Vert_{W^{1^+,\infty}_x}\prod_{j=4}^k\min\{N_j^{1/2},N_j^-\}\Vert P_{N_j}z_j\Vert_{H^{1/2^+}_x}dt'
\\ & \quad \lesssim_k T\Vert w\Vert_{L^\infty_TH^{s-1}_x}^2\Vert z_3\Vert_{L^\infty_TH^{s}_x}\Vert\Psi^m\Vert_{L^\infty_TW^{1^+,\infty}_x}\prod_{j=4}^k\Vert z_j\Vert_{L^\infty_TH^{1/2^+}_x}.
\end{align*}
Here, we have used the fact that $\Vert P_{N_1}w(s,\cdot)\Vert_{H^{s-1}_x}$ and $\Vert P_{N_3}z_3(s,\cdot)\Vert_{H^{s}_x}$ are both square summable, so that we are able to re-sum in the region $N_1\sim N_3\gg \max\{N_2,N_4,N_{k+1}\}$. Therefore, in the sequel we can assume that $\min\{N_1,N_2,N_3\}\geq 2^9kN_{k+1}$. This concludes all the straightforward cases. Now, in order to deal with the remaining region, we split the analysis into three cases, namely, 
\begin{align*}
\mathbf{N}^1_k&:=\big\{(N_1,...,N_{k+1})\in \mathbb{D}^{k}\times\mathbb{D}_{\mathrm{nh}}: \,N_2\geq 9^9k,\, N_3\geq 8^8k,\nonumber
\\ & \qquad \qquad \qquad \  \min\{N_1,N_2,N_3\}\geq 2^9kN_{k+1},\,N_3< 2^9kN_4\big\}\cap \mathbf{N}_{k+1},\nonumber
\\ \mathbf{N}^2_k&:=\big\{(N_1,...,N_{k+1})\in \mathbb{D}^{k}\times\mathbb{D}_{\mathrm{nh}}: \,N_2\geq 9^9k,\,N_3\geq 8^8k,
\\ & \qquad \qquad \qquad \  \min\{N_1,N_2,N_3\}\geq 2^9kN_{k+1},\, N_3\geq 2^9kN_4,\, N_2<2^9kN_4 \big\}\cap \mathbf{N}_{k+1},\nonumber 
%
%
%
%
%
%
%
%
\\ \mathbf{N}^3_k&:=\big\{(N_1,...,N_{k+1})\in \mathbb{D}^{k}\times\mathbb{D}_{\mathrm{nh}}: \,N_2\geq 9^9k,\, N_3\geq 8^8k,\nonumber
\\ & \qquad \qquad \qquad \  \min\{N_1,N_2,N_3\}\geq 2^9k N_{k+1},\, N_3\geq 2^9kN_4,\,N_2\geq 2^9kN_4\big\}\cap \mathbf{N}_{k+1}.\nonumber
%
%
%
%
%
%
%
\end{align*}
We denote by $\mathrm{D}_1, \mathrm{D}_2,\mathrm{D}_3$ the contribution of \eqref{def_int_troisieme_psi} associated with each of these regions, respectively. Notice that the case $k=3$ shall follow directly from the bound exposed for $\mathbf{N}^3_k$, while $\mathbf{N}^1_k$ and $\mathbf{N}^2_k$ only concern the cases $k\geq 4$.

\medskip

\textbf{Step 3.1:}  Let us begin by considering the case of $\mathbf{N}^1_k$. In fact, in this case, by using Lemma \ref{pseudo_holder_l2}, Plancherel Theorem as well as H\"older and Bernstein inequalities we get
\begin{align*}
\mathrm{D}_1&\lesssim_k \int_0^T\sum_{N\gg 1}\sum_{\mathbf{N}^1_k}\langle N\rangle^{2s-2}\left\vert\int_{\R}\Pi_{d_{k+1}}\big(P_{N_{1}}w,P_{N_{2}}w\big)P_{N_3}z_3...P_{N_k}z_kP_{N_{k+1}}(\Psi^m)\right\vert dt'
\\ & \lesssim_k \int_0^T\sum_{N\gg 1}\sum_{\mathbf{N}^1_k}\langle N\rangle^{2s-2}\langle N_1\rangle^{1-s}\langle N_2\rangle^{1-s}N_3^{-(1^+)}N_{k+1}^{-(1^+)}\min\big\{N,N_{3}\big\}\Vert P_{N_1}w\Vert_{H^{s-1}_x}\times
\\ & \quad \times \Vert P_{N_2}w\Vert_{H^{s-1}_x}\Vert J_x^{1/2^+}P_{N_3}z_3\Vert_{L^\infty_x}\Vert J_x^{1/2^+}P_{N_4}z_4\Vert_{L^\infty_x}\Vert P_{N_{k+1}}(\Psi^m)\Vert_{W^{1^+,\infty}_x}\prod_{j=5}^k\Vert P_{N_j}z_j\Vert_{L^\infty_x}
\\ & \lesssim_k \Vert w\Vert_{L^\infty_TH^{s-1}_x}^2\big\Vert J_x^{1/2^+}z_3\Vert_{L^2_TL^\infty_x}\big\Vert J_x^{1/2^+}z_4\Vert_{L^2_TL^\infty_x}\Vert\Psi^m\Vert_{W^{1^+,\infty}_x}\prod_{j=5}^k\big\Vert z_j\Vert_{L^\infty_TH^{1/2^+}_x},
\end{align*}
where, to sum over the region $\{N_{3}\ll N\hbox{ and }N_1\sim N_2\sim N\}$, we have used the fact that $\Vert P_{N_1}w(s,\cdot)\Vert_{H^{s-1}_x}$ and $\Vert P_{N_2}w(s,\cdot)\Vert_{H^{s-1}_x}$ are both square summable, while the case $N_3\gtrsim N$ follows directly thanks to the factor $N_3^-$.

\medskip

\textbf{Step 3.2:} Now we consider the contribution of \eqref{def_int_troisieme_psi} associated with $\mathbf{N}^2_k$. Indeed, proceeding similarly as above, by using Lemma \ref{pseudo_holder_l2}, Plancherel Theorem as well as H\"older and Bernstein inequalites we infer that \begin{align*}
\mathrm{D}_2&\lesssim_k \int_0^T\sum_{N\gg 1}\sum_{\mathbf{N}^2_k}\langle N\rangle^{2s-2}\left\vert\int_{\R}\Pi_{d_{k+1}}\big(P_{N_{1}}w,P_{N_{2}}w\big)P_{N_3}z_3...P_{N_k}z_kP_{N_{k+1}}(\Psi^m)\right\vert dt'
\\ & \lesssim_k \int_0^T\sum_{N\gg 1}\sum_{\mathbf{N}^2_k}\langle N\rangle^{2s-2}\min\{N,N_3\}\Vert P_{N_1}w(t',\cdot)\Vert_{L^2_x}\Vert P_{N_2}w(t',\cdot)\Vert_{L^\infty_x}\times
\\ & \quad \times \Vert P_{N_3}z_3(t',\cdot)\Vert_{L^2_x}\Vert P_{N_4}z_4(t',\cdot)\Vert_{L^\infty_x}\Vert P_{N_{k+1}}(\Psi^m)\Vert_{L^\infty_{t,x}}\prod_{j=5}^k\Vert P_{N_j}z_j(t',\cdot)\Vert_{L^\infty_x}dt'
\\ & \lesssim_k \int_0^T\sum_{N\gg 1}\sum_{\mathbf{N}^2_k}\langle N\rangle^{2s-2}\langle N_1\rangle^{1-s}\langle N_3\rangle^{-s}N_2^{-}N_4^{-}N_{k+1}^{-1}\min\big\{N,N_3\big\}\Vert P_{N_1}w\Vert_{H^{s-1}_x}\times
\\ &  \quad \times \big\Vert J^{(-1/2)^+}_xP_{N_2}w\big\Vert_{L^\infty_x}\Vert P_{N_3}z_3\Vert_{H^s_x}\Vert J_x^{1/2^+}P_{N_4}z_4\Vert_{L^\infty_x}\Vert P_{N_{k+1}}(\Psi^m)\Vert_{W^{1,\infty}_x}\prod_{j=5}^k\Vert P_{N_j}z_j\Vert_{L^\infty_x}dt'
\\ & \lesssim_k \Vert w\Vert_{L^\infty_TH^{s-1}_x}\big\Vert J_x^{(-1/2)^+}w\Vert_{L^2_TL^\infty_x}\big\Vert z_3\Vert_{L^\infty_TH^s_x}\big\Vert J_x^{1/2^+}z_4\big\Vert_{L^2_TL^\infty_x}\Vert \Psi\Vert_{L^\infty_tW^{1,\infty}_x}^m\prod_{j=5}^k\big\Vert z_j\Vert_{L^\infty_TH^{1/2^+}_x}.
\end{align*}
Here we have used the fact that, due to our current hypotheses, we always have that $N_3\gtrsim N$. In fact, if $N_2\leq \tfrac{1}{16}N$, since $N_4\ll N_3$ and $N_{k+1}\ll \min\{N_1,N_2,N_3\}$, by using the explicit form of $d_{k+1}$ we infer that \[
N_1\in[\tfrac{1}{2}N,2N] \quad \hbox{and}\quad N_3\in[\tfrac{1}{4}N_1,4N_1].
\]
On the other hand, if $N_2\geq \tfrac{1}{8} N$, then, since $N_3\geq N_4\gtrsim N_2$, we obtain the desired relation.

\medskip

\textbf{Step 3.3:} Finally, we are ready to treat the remaining case, that is, we now deal with the region $\mathbf{N}^3_k$. In order to do so, we begin using the decomposition given in \eqref{decomposition}, from where we infer that it is enough to control the following quantities \begin{align*}
\mathcal{D}_{3,R}^{\mathrm{high}}& :=\sum_{N\gg 1}\sum_{\textbf{N}^3_k}\langle N\rangle^{2s-2}\sup_{t\in(0,T)}\left\vert \int_{\R^2}\Pi_{d_{k+1}}\big(\mathds{1}_{t,R}^{\mathrm{high}}P_{N_1}w,\mathds{1}_tP_{N_2}w\big)P_{N_{k+1}}(\Psi^m)\prod_{j=3}^kP_{N_j}z_j\right\vert,
\\ \mathcal{D}_{3,R}^{\mathrm{low,high}}&:=\sum_{N\gg 1}\sum_{\textbf{N}^3_k}\langle N\rangle^{2s-2}\sup_{t\in(0,T)}\left\vert\int_{\R^2}\Pi_{d_{k+1}}\big(\mathds{1}_{t,R}^{\mathrm{low}}P_{N_1}w,\mathds{1}_{t,R}^{\mathrm{high}}P_{N_2}w\big)P_{N_{k+1}}(\Psi^m)\prod_{j=3}^kP_{N_j}z_j\right\vert,
\\ \mathcal{D}_{3,R}^{\mathrm{low,low}}&:=\sum_{N\gg 1}\sum_{\textbf{N}^3_k}\langle N\rangle^{2s-2}\sup_{t\in(0,T)}\left\vert \int_{\R^2}\Pi_{d_k}\big(\mathds{1}_{t,R}^{\mathrm{low}}P_{N_1}w,\mathds{1}_{t,R}^{\mathrm{low}}P_{N_2}w\big)P_{N_{k+1}}(\Psi^m)\prod_{j=3}^kP_{N_j}z_j\right\vert,
\end{align*}
where $R$ stands for $R:=N_1N_3$, as in the previous cases. For the sake of clarity we split the analysis into two steps. 

\medskip
 
\textbf{Step 3.3.1:} We begin by considering the case of $\mathcal{D}_{3,R}^{\mathrm{high}}$. Once again, the idea is to take advantage of the operator $\mathds{1}_{t,R}^{\mathrm{high}}$ by using Lemma \ref{lem_high_low_1}. Notice that, since in this case we have  $N_3\leq 8N_1$, then, roughly speaking, either we have \begin{align}\label{possibilities_step31}
\{N_1\gg N, \ N_2\sim N\ \hbox{and} \ N_1\sim N_3\} \quad \hbox{or}\quad \{N_1\sim N \,\hbox{ and } \, \max\{N_2,N_3\}\sim  N_1\}.
\end{align}
We can then bound $\mathcal{G}_{3,R}^{\mathrm{high}}$ by using the first inequality in Lemma \ref{lem_high_low_1}, Lemma   \ref{pseudo_holder_l2}, as well as Sobolev's embedding, in the following fashion
\begin{align*}
&\mathcal{D}_{3,R}^{\mathrm{high}}
\\ &  \lesssim_k \sum_{N \gg 1}\sum_{\textbf{N}^3_k}T^{1/4}\langle N\rangle^{2s-2} \Vert \mathds{1}_{T,R}^{\mathrm{high}}\Vert_{L^{4/3}}\bigg\Vert\int_\R \Pi_{d_{k+1}}\big(\mathds{1}_{t,R}^{\mathrm{high}}P_{N_1}w,\mathds{1}_tP_{N_2}w\big)P_{N_{k+1}}(\Psi^m)\prod_{j=3}^kP_{N_{j}}z_j\bigg\Vert_{L^\infty_t}
\\ & \lesssim_k \sum_{N \gg 1}\sum_{\textbf{N}^3_k}T^{1/4}\langle N\rangle^{2s-2}\langle N_1\rangle^{1-s}\langle N_2\rangle^{1-s}(N_1N_3)^{-3/4}\min\{N,N_3\}\Vert P_{N_1}w\Vert_{L^\infty_tH^{s-1}_x}\times
\\ & \qquad \qquad \times \Vert P_{N_2}w\Vert_{L^\infty_tH^{s-1}_x}\Vert P_{N_{k+1}}(\Psi^m)\Vert_{L^\infty_x}\prod_{j=3}^{k} \min\big\{N_j^{1/2},N_j^{-(0^+)}\big\}\Vert P_{N_j}z_j\Vert_{L^\infty_{t}H^{1/2^+}_x}
\\ & \lesssim_k T^{1/4}\Vert w\Vert_{L^\infty_tH^{s-1}_x}^{2}\Vert\Psi\Vert_{L^\infty_{t,x}}^m\prod_{j=3}^k\Vert z_j\Vert_{L^\infty_tH^{1/2^+}_x}.
\end{align*}
Once again, notice that thanks to the operator $\mathds{1}_{t,R}^\mathrm{high}$ acting on the factor $P_{N_2}w$, the same estimates also hold for $\mathcal{D}_{3,R}^{\mathrm{low,high}}$.

\medskip

\textbf{Step 3.3.2:} Now we consider the last term in the decomposition, that is, $\mathcal{D}_{3,R}^{\mathrm{low,low}}$. In fact, first of all, let us recall the notation introduced in the proof of the previous proposition (adapted to the current symbol)
\begin{align*}
&\mathcal{I}_k\big(u_1,...,u_{k+1}):=\sum_{N\gg 1}\sum_{\textbf{N}^3_k}\langle N\rangle^{2s-2}\sup_{t\in(0,T)}\left\vert \int_{\R^2} \Pi_{d_{k+1}}(u_1,u_{2})u_3...u_{k+1}\right\vert.
\end{align*}
Then, by using the the hypothesis $\min\{N_2,N_3\}\geq2^9kN_4$, following the same computations as in \eqref{resonant_o_n1n3}, we infer that the resonant relation satisfies \[
\big\vert\Omega_k(\xi_1,...,\xi_k)\big\vert\sim N_1N_2N_3.
\]
Therefore, taking advantage of the above relation, we can now decompose $\mathcal{D}_{3,R}^{\mathrm{low,low}}$ with respect to modulation variables in the following fashion
\begin{align*}
\mathcal{D}_{3,R}^{\mathrm{low,low}} &\leq \mathcal{I}\big(Q_{\gtrsim N^*}\mathds{1}_{t,R}^{\mathrm{low}}P_{N_1}w,\mathds{1}_{t,R}^{\mathrm{low}}P_{N_2}w,P_{N_3}z_3,...,P_{N_{k}}z_k,P_{N_{k+1}}(\Psi^m)\big)
\\ & \quad +\mathcal{I}\big(Q_{\ll N^*}\mathds{1}_{t,R}^{\mathrm{low}}P_{N_1}w,Q_{\gtrsim N^*}\mathds{1}_{t,R}^{\mathrm{low}}P_{N_2}w,P_{N_3}z_3,...,P_{N_{k}}z_k,P_{N_{k+1}}(\Psi^m)\big)
\\ & \quad +\mathcal{I}\big(Q_{\ll N^*}\mathds{1}_{t,R}^{\mathrm{low}}P_{N_1}w,Q_{\ll N^*}\mathds{1}_{t,R}^{\mathrm{low}}P_{N_2}w,Q_{\gtrsim N^*}P_{N_3}z_3,...,P_{N_{k}}z_k,P_{N_{k+1}}(\Psi^m)\big) 
\\ & \quad +...
\\ & \quad +\mathcal{I}\big(Q_{\ll N^*}\mathds{1}_{t,R}^{\mathrm{low}}P_{N_1}w,Q_{\ll N^*}\mathds{1}_{t,R}^{\mathrm{low}}P_{N_2}w,Q_{\ll N^*}P_{N_3}z_3,...,Q_{\gtrsim N^*}P_{N_{k}}z_k,P_{N_{k+1}}(\Psi^m)\big)
\\ & \quad +\mathcal{I}\big(Q_{\ll N^*}\mathds{1}_{t,R}^{\mathrm{low}}P_{N_1}w,Q_{\ll N^*}\mathds{1}_{t,R}^{\mathrm{low}}P_{N_2}w,Q_{\ll N^*}P_{N_3}z_3,...,Q_{\ll N^*}P_{N_{k}}z_k,P_{N_{k+1}}(\Psi^m)\big)
\\ & =: \mathcal{I}_1+...+\mathcal{I}_{k+1},
\end{align*}
where $N^*$ stands for $N^*:=N_1N_2N_3$. At this point it is important to recall that, since $N_2\geq 9^9k$, then we also have that $N^*\gg N_1N_3=R$, what allows us to use the last inequality in Lemma \ref{lem_high_low_1}. Thus, bounding in a similar fashion as before, by using H\"older and Bernstein inequalities, as well as Lemmas \ref{lem_high_low_1} and \ref{pseudo_holder_l2}, and classical Bourgain estimates, we obtain
\begin{align*}
\mathcal{I}_{1}&\lesssim_k \sum_{N\gg 1}\sum_{\textbf{N}^3_k}\langle N\rangle^{2s-2}\min\{N,N_3\}\big\Vert Q_{\gtrsim N^*}\mathds{1}_{T,R}^{\mathrm{low}}P_{N_1}w\big\Vert_{L^2_tL^2_x}\times 
\\ & \qquad \quad \times \Vert \mathds{1}_{T,R}^{\mathrm{low}}P_{N_2}w\Vert_{L^2_tL^2_x}\Vert P_{N_{k+1}}(\Psi^m)\Vert_{L^\infty_{t,x}}\prod_{j=3}^{k}N_j^{1/2}\Vert P_{N_j}z_j\Vert_{L^\infty_{t}L^2_x}
\\ & \lesssim_k \sum_{N\gg 1}\sum_{\textbf{N}^3_k}\langle N\rangle^{2s-2}\langle N_1\rangle^{1-s}\langle N_2\rangle^{1-s}N_2^{-1}N_3^{-1}\min\{N,N_3\}\big\Vert Q_{\gtrsim N^*}\mathds{1}_{T,R}^{\mathrm{low}}P_{N_1}w\big\Vert_{X^{s-2,1}}\times 
\\ & \qquad \quad \times \Vert \mathds{1}_{T,R}^{\mathrm{low}}\Vert_{L^2}\Vert P_{N_2}w\Vert_{L^\infty_tH^{s-1}_x}\Vert P_{N_{k+1}}(\Psi^m)\Vert_{L^\infty_{t,x}}\prod_{j=3}^{k+1}\min\{N_j^{1/2},N_j^{-}\}\Vert P_{N_j}z_j\Vert_{L^\infty_tH^{1/2^+}_x}
\\ & \lesssim_k T^{1/2}\Vert w\Vert_{X^{s-2,1}}\Vert w\Vert_{L^\infty_tH^{s-1}_x}\Vert \Psi\Vert_{L^\infty_{t,x}}^m\prod_{j=3}^k\Vert z_j\Vert_{L^\infty_tH^{1/2^+}_x},
\end{align*}
where we have used again \eqref{possibilities_step31}. Moreover, it is not difficult to see that, by following the same lines (up to trivial modifications), we can also bound $\mathcal{I}_2$, from where we obtain the same bound. On the other hand, to control $\mathcal{I}_3$ we use again Lemmas \ref{lem_high_low_1} and \ref{pseudo_holder_l2}, as well as H\"older and Bernstein inequalities, from where we get that   \begin{align*}
\mathcal{I}_{3}&\lesssim_k \sum_{N\gg 1}\sum_{\textbf{N}^3_k}\langle N\rangle^{2s-2}N_3^{1/2}\min\{N,N_3\}\big\Vert Q_{\ll N^*}\mathds{1}_{T,R}^{\mathrm{low}}P_{N_1}w\big\Vert_{L^2_tL^2_x}\Vert Q_{\ll N^*}\mathds{1}_{T,R}^{\mathrm{low}}P_{N_2}w\Vert_{L^\infty_tL^2_x}\times
\\ & \qquad \quad  \times \Vert Q_{\gtrsim N^*}P_{N_3}z_3 \Vert_{L^2_tL^2_x}\Vert P_{N_{k+1}}(\Psi^m)\Vert_{L^\infty_{t,x}}\prod_{j=4}^{k}N_j^{1/2}\Vert P_{N_j}z_j\Vert_{L^\infty_tL^2_x}
\\ & \lesssim_k \sum_{N\gg 1}\sum_{\textbf{N}^3_k}T^{1/2}\langle N\rangle^{2s-2}\langle N_1\rangle^{-s}\langle N_2\rangle^{-s}\langle N_3\rangle^{+}\min\{N,N_3\} \Vert P_{N_1}w\Vert_{L^\infty_tH^{s-1}_x}\Vert P_{N_2}w\Vert_{L^\infty_tH^{s-1}_x}\times 
\\ & \qquad \quad \times \Vert Q_{\gtrsim N^*}P_{N_3}z_3\Vert_{X^{(-1/2)^+,1}}\Vert P_{N_{k+1}}(\Psi^m)\Vert_{L^\infty_{t,x}}\prod_{j=4}^{k+1}\min\{N_j^{1/2},N_j^{-(0^+)}\}\Vert P_{N_j}z_j\Vert_{L^\infty_tH^{1/2^+}_x}
\\ & \lesssim_k T^{1/2}\Vert w\Vert_{L^\infty_tH^{s-1}_x}^2\Vert z_3\Vert_{X^{(-1/2)^+,1}}\Vert \Psi\Vert_{L^\infty_{t,x}}^m\prod_{j=4}^k\Vert z_j\Vert_{L^\infty_tH^{1/2^+}_x}.
\end{align*}
Notice that all the remaining cases $\mathcal{I}_i$, $i=4,...,k$, follow very similar lines to the latter case (up to trivial modifications), and they provide exactly the same bound as above, and hence we omit their proof. Finally, to control $\mathcal{I}_{k+1}$ notice that, since all factors $P_{N_i}u$ have an operator $Q_{\ll N^*}$ in front of them, then the factor $P_{N_{k+1}}(\Psi^m)$ is forced to be resonant, and hence in this case we can write $Q_{\gtrsim N^*} P_{N_{k+1}} (\Psi^m) = P_{N_{k+1}}(\Psi^m)$, otherwise $\mathcal{I}_{k+1}=0$ thanks to Lemma \ref{resonant_bourgain}. Moreover, notice also that, in the region $\mathbf{N}^3_k$ we have in particular that $N_1N_2N_3\gg N_{k+1}^3$, and hence we infer that $\vert \tau_{k+1}-\xi_{k+1}^3\vert\sim \vert\tau_{k+1}\vert$, and hence, we can write $P_{N_{k+2}}(\Psi^m)=R_{\gtrsim N^*}P_{N_{k+2}}(\Psi^m)$. Therefore, by using Lemmas \ref{lem_high_low_1} and  \ref{pseudo_holder_l2}, as well as Bernstein inequality and the above properties, we infer that 
\begin{align*}
\mathcal{I}_{k+1}& \lesssim_k \sum_{N\gg 1}\sum_{\mathbf{N}^3_k}\langle N\rangle^{2s-2}\langle N_1\rangle^{-s}\langle N_2\rangle^{-s}N_3^{-1}\min\{N,N_3\}\Vert \mathds{1}_{T,R}^{\mathrm{low}}\Vert_{L^2}^2\Vert P_{N_1}w\Vert_{L^\infty_tH^{s-1}_x}\times
\\ & \quad \ \times \Vert P_{N_2}w\Vert_{L^\infty_tH^{s-1}_x}\Vert \partial_t R_{\gtrsim N^*}P_{N_{k+2}}(\Psi^m)\Vert_{L^\infty_{t,x}}\prod_{i=3}^{k+1}\min\{N_i^{1/2},N_i^{-(0^+)}\}\Vert P_{N_i}u\Vert_{L^\infty_{t}H^{1/2^+}_x}
\\ & \lesssim_k T\Vert w\Vert_{L^\infty_tH^{s-1}_x}^2\Vert\partial_t\Psi\Vert_{L^\infty_{t,x}}\Vert \Psi\Vert_{L^\infty_{t,x}}^{m-1}\prod_{j=3}^k\Vert z_j\Vert_{L^\infty_tH^{1/2^+}_x}.
\end{align*}

\smallskip

Now we explain how we control the contribution of the $L^2_TL^\infty_x$ terms. With this aim, recalling the equation solved by $w(t,x)$, we use the Strichartz estimate \eqref{smoothing_eff} with $\delta=1$, from where we obtain \begin{align*}
\big\Vert J_x^{(-1/2)^+}w\Vert_{L^2_TL^\infty_x}&\lesssim T^{1/4}\Vert w\Vert_{L^\infty_TH^{(-1/2)^+}_x}+T^{3/4}\Vert w\Vert_{L^\infty_TH^{(-1/2)^+}_x}\times
\\ & \quad \times\sum_{k=1}^\infty kc^k\vert a_k\vert\big(\vert u\Vert_{L^\infty_TH^{1/2^+}_x}+\Vert v\Vert_{L^\infty_TH^{1/2^+}_x}+\Vert\Psi\Vert_{L^\infty_TW^{1/2^+,\infty}_x}\big)^{k-1}.
\end{align*}
Thus, gathering all the above estimate, and then using Lemma \ref{mst_basic_lemma}, we conclude the proof of the proposition.
The proof is complete.
\end{proof}

\medskip

\section{Unconditional well-posedness in $H^s$ for $s>1/2$}\label{section_MT1}

\subsection{Existence and unconditional uniqueness}

In this section we shall assume Theorem \ref{MT2} hold, that is, we assume equation \eqref{gkdv_v} is locally well-posed in $H^{3/2^+}(\R)$. In the next section we shall sketch the main ideas of its proof (see \cite{AbBoFeSa} for further details, for example).

\medskip

Before going further, for the sake of simplicity and by abusing notation, recalling that $ \Psi$ is a given function, from now on let us denote by 
\begin{align*}
|||\Psi|||_{r}&:=\Vert \partial_t\Psi+\partial_x^3\Psi+\partial_xf(\Psi)\Vert_{L^\infty_tH^{r}_x}, \quad \hbox{and by}
\\ \mathcal{Q}_*\big(\Vert u\Vert_{L^\infty_TH^{1/2^+}_x}\big)&:= \mathcal{Q}_*\big(\Vert u\Vert_{L^\infty_TH^{1/2^+}_x}, \Vert \Psi\Vert_{L^\infty_tW^{s+1^+,\infty}_x},\Vert \partial_t\Psi\Vert_{L^\infty_{t,x}},|||\Psi|||_{s^+}\big).
\end{align*}
Now, consider a sequence $\{\Psi_n\}_{n\in\N}\subset L^\infty(\R^2)$ satisyfing the hypotheses in  \eqref{hyp_psi} with $s'=\tfrac{3}{2}^+$ for all $n\in\N$, such that \begin{align*}
&\Vert \partial_t\Psi_n-\partial_t\Psi\Vert_{L^\infty_{t,x}}+\Vert \Psi_n-\Psi\Vert_{L^\infty_tW^{s+1^+,\infty}_x}+|||\Psi_n-\Psi|||_{s^+}\xrightarrow{n\to+\infty}0.
\end{align*}
Let $u\in C([0,T_0],H^\infty(\R))$ be a smooth solution to equation \eqref{gkdv_v} associated with $\Psi_n$, with minimal existence time \[T^\star=T^\star(\Vert u_0\Vert_{H^{3/2^+}},|||\Psi_n|||_{3/2^+},\Vert \Psi_n\Vert_{L^\infty_tW^{5/2,\infty}_x})>0,
\]
emanating from an initial data $u_0\in H^\infty(\R)$. Then, according to Proposition \ref{prop_energy_est_impro}, there exist a constant $c>0$ such that, after an application of Cauchy-Schwarz inequality, we have \begin{align}\label{apriori_hsw}
\Vert u\Vert_{L^\infty_TH^s_\omega}^2\leq \Vert u_0\Vert_{H^s_\omega}^2+cT\Vert u\Vert_{L^\infty_TH^s_\omega}^2+cT||| \Psi_n|||_{s^+}^2+ cT^{1/4}\Vert u\Vert_{L^\infty_TH^s_\omega}^2 \mathcal{Q}_*\big(\Vert u\Vert_{L^\infty_TH^{1/2^+}_x}\big),
\end{align}
for all $0< T\leq\min\{1,T_0\}$. We stress that $\mathcal{Q}_*$ only involves norms of $\Psi_n$ associated with $s$ and not with $s'=3/2^+$. Notice also that they do not depend on $T$ either. Thus, we can consider the function
\begin{align*}
F(T):= cT+cT^{1/4}\mathcal{Q}_*\big(\Vert u\Vert_{L^\infty_TH^{1/2^+}_x}\big), \quad T\in[0,T_0].
\end{align*}
At this point it is important to recall that $\Vert u\Vert_{L^\infty_TH^s_\omega}\to\Vert u_0\Vert_{H^s_\omega}$ as $T\to 0$. Moreover, notice that $F(0)=0$, and hence, thanks to the continuity of $T\mapsto F(T)$, we infer the existence of $T_*=T_*(\Vert u_0\Vert_{H^s_\omega})>0$ small enough, such that \[
F(T')<1/2 \quad \hbox{for all} \quad T'<T_*.
\]
In particular, the above inequality along with \eqref{apriori_hsw} implies that
\begin{align}\label{uniform_estimate_hs}
\Vert u\Vert_{L^\infty_{T'}H^s_\omega}\lesssim \Vert u_0\Vert_{H^s_\omega}+|||\Psi|||_{s^+} \quad \hbox{for all}\quad T'<T_*.
\end{align}
Note that, from \eqref{uniform_estimate_hs} we infer that the minimal existence time can be chosen only depending on $\Vert u_0\Vert_{H^{s}_\omega}$ and $|||\Psi|||_{s^+}$. On the other hand, by using Proposition \ref{prop_diff_sol} we have that \begin{align*}
\Vert u-v\Vert_{L^\infty_TH^{s-1}_x}^2&\lesssim \Vert u_0-v_0\Vert_{H^{s-1}}^2 
+T^{1/4}\Vert u-v\Vert_{L^\infty_TH^{s-1}_x}^2\mathcal{Q}^*\big(\Vert u\Vert_{L^\infty_TH^s_x},\Vert v\Vert_{L^\infty_TH^s_x}\big),
\end{align*}
where, as above, by an abuse of notation we are denoting by \[
\mathcal{Q}^*\big(\Vert u\Vert_{L^\infty_TH^s_x},\Vert v\Vert_{L^\infty_TH^s_x}\big):= \mathcal{Q}^*\big(\Vert u\Vert_{L^\infty_TH^s_x},\Vert v\Vert_{L^\infty_TH^s_x},\Vert \Psi\Vert_{L^\infty_tW^{s+1,\infty}_x},\Vert \partial_t\Psi\Vert_{L^\infty_{t,x}}\big).
\]
Therefore, a similar continuity argument as before yield us to the existence of a positive time $\tilde{T}_*=\tilde{T}_*(\Vert u_0\Vert_{H^{s}},\Vert v_0\Vert_{H^{s}})>0$ such that 
 \begin{align}\label{uniform_estimate_diff2_concl}
\Vert u-v\Vert_{L^\infty_{\tilde T'}H^{s-1}}\lesssim \Vert u_0-v_0\Vert_{H^{s-1}} \quad \hbox{for all}\quad \tilde T'<\tilde T_*.
\end{align}
Now, let us consider an initial data $u_0\in H^s(\R)$ with $s>1/2$. Consider a smooth sequence of functions $\{u_{0,n}\}_{n\in\N}$ strongly converging to $u_0$ in $H^s(\R)$. Let $u_n(t)$ be the solution to equation \eqref{gkdv_v} associated with $\Psi_n$, with initial data $u_{0,n}$. Note that the above analysis assures us that we can define the whole family of solutions $\{u_n\}$  in a common existence time interval $[0,T^*]$, for some $T^*>0$ only depending on $\Vert u_0\Vert_{H^s}$ and $|||\Psi|||_{s^+}$. Then, thanks to estimate \eqref{uniform_estimate_hs} with $\omega_N\equiv 1$, $\{u_n\}_{n\in\N}$ defines a bounded sequence in $C([0,T^*],H^s(\R))$, and hence, we can extract a subsequence (which we still denote by $u_n$) converging in the weak-$\star$ topology of $L^\infty_{T^*}H^s_x$ to some limit $u$. Moreover, from this latter convergence we also infer that $\partial_xf(u_n+\Psi_n)$ converges in a distributional sense to $\partial_xf(u+\Psi)$. Therefore, the limit object $u$ solves equation \eqref{gkdv_v} with $\Psi$, in a distributional sense. Furthermore, from  \eqref{uniform_estimate_diff2_concl} we get that $\{u_n\}$ defines a Cauchy sequence in $C([0,T^*],H^{s-1}(\R))$, and hence $\{u_n\}_{n\in\N}$ strongly converges to $u$ in $L^\infty((0,T^*),H^{s-1}(\R))$. By the same reasons, from estimate \eqref{uniform_estimate_diff2_concl} we conclude that this solution is the only one in the class $L^\infty((0,T),H^{s}(\R))$. On the other hand, the above results ensure that the map \[
[0,T^*]\ni t\mapsto u(t)\in H^s(\R)
\]
is weakly continuous. In fact, let $\varphi\in H^s$ arbitrary, and consider $\tilde{\varphi}\in H^{s+1}$ to be any function satisfying $\Vert \varphi-\tilde{\varphi}\Vert_{H^s}\leq \varepsilon/\Vert u\Vert_{L^\infty_{T^*}H^s_x}$. Then, for all $t,t'\in(0,T^*)$ we have \begin{align*}
\vert\langle u(t)-u(t'),\varphi\rangle\vert_{H^s}&\leq \varepsilon+\vert\langle J_x^{s-1}(u(t)-u(t')),J^{s+1}\tilde\varphi\rangle_{H^s}\vert
\\ & \leq \varepsilon+2\Vert u_n-u\Vert_{L^\infty_TH^{s-1}_x}\Vert \tilde\varphi\Vert_{H^{s+1}}+\Vert u_n(t)-u_n(t')\Vert_{H^{s-1}_x}\Vert\tilde\varphi\Vert_{H^{s+1}}.
\end{align*}
Then, choosing $n$ sufficiently large, using the strong convergence result in $H^{s-1}$, we deduce that we can control the right-hand side of the above inequality by $3\varepsilon$, and hence $u(t)$ is weakly continuous from $[0,T^*]$ into $H^s(\R)$. Moreover recalling that, due to \eqref{uniform_estimate_diff2_concl},  $\{u_n\}$ defines a Cauchy sequence in $C([0,T^*],H^{s-1}(\R))$, we infer in particular that $u\in C([0,T^*],H^{s-1}(\R))$.

\smallskip

\subsection{Continuity of the flow map}

We are finally ready to prove both, the continuity of the flow map and the continuity of $u(t)$ with values in $H^s(\R)$. Before getting into the details, let us recall the following standard lemma.
\begin{lem}[See \cite{KoTz}]
Let $\{f_n\}_{n\in\N}\subset H^s(\R)$ satisfying $f_n\to f$ in $H^s(\R)$. Then, there exists an increasing sequence $\{\omega_N\}_{N\in\mathbb{D}}\subset\R$ of positive numbers satisfying $\omega_N\leq \omega_{2N}\leq 2^+\omega_N$, with \[
\{ \omega_N\nearrow+\infty \, \hbox{ as } \, N\to+\infty\} \quad \hbox{and}\quad \{ \omega_N\to1 \, \hbox{ as } \, N\to0\},
\]
such that 
\[
\sup_{n\in\N}\sum_{N>0}\omega_N^2\langle N\rangle^{2s}\Vert P_Nf_n\Vert_{L^2}^2<\infty.
\] 
\end{lem}

With this in mind, let $\{u_n\}_{n\in\N}$ be a sequence of solutions in $L^\infty([0,T],H^s(\R))$ associated with an initial datum $u_n(0)$ satisfying $u_n(0)\to u(0)$ in $H^s(\R)$.  Now, we use the previous lemma with $f_n=u_{n}(0)$ and $f=u(0)$. Consider $\{\omega_N\}_{N\in\mathbb{D}}$ given by the previous lemma. Then, it follows from Proposition \ref{prop_energy_est_impro} and estimate \eqref{uniform_estimate_hs} that
\begin{align}\label{uniform_bound_conclusions}
\sup_{n\in\N}\sup_{t\in(0,T)}(\Vert u_n(t)\Vert_{H^s_\omega}+\Vert u(t)\Vert_{H^s_\omega})<+\infty.
\end{align}
Note that the strong continuity in $C([0,T],H^{s-1}(\R))$, together with the boundedness in $H^s_\omega$ implies, in particular, the strong continuity of the map $[0,T]\ni t\mapsto u(t)$ in $H^s(\R)$.

\medskip 

Finally, to complete the proof of Theorem \ref{MT1} it only remains to show the continuity of the flow map.  Consider $\{u_n\}$ and $u$ as above. We intend to control $\Vert u_n-u\Vert_{L^\infty_TH^s_x}$. First of all, by using triangular inequality we have \[
\Vert u_n-u\Vert_{L^\infty_TH^s_x}\leq \Vert u_n-P_{\leq N}u_n\Vert_{L^\infty_TH^s_x}+\Vert P_{\leq N}u_n-P_{\leq N}u\Vert_{L^\infty_TH^s_x}+\Vert P_{\leq N}u-u\Vert_{L^\infty_TH^s_x},
\]
for any $N\in\mathbb{D}$. Then, take $\varepsilon\in(0,1)$ arbitrary but fixed. We claim that, as a particular consequence of  \eqref{uniform_bound_conclusions},  there exists $N_*\gg1$ dyadic, such that for all $t\in[0,T]$ we have \[
\sup_{n\in\N}\Vert u_n(t)-P_{\leq N_*}u_n(t)\Vert_{H^s_x}+\Vert u(t)-P_{\leq N_*}u(t)\Vert_{H^s_x}<\tfrac{1}{2}\varepsilon.
\]
In fact, it is enough to notice that \begin{align*}
&\sup_{n\in\N}\sup_{t\in(0,T)}\Vert u_n(t)-P_{\leq N_*}u_n(t)\Vert_{H^s_x}+\Vert u(t)-P_{\leq N_*}u(t)\Vert_{H^s_x}
\\ & \qquad \lesssim \sup_{n\in\N}\sup_{t\in(0,T)}\dfrac{1}{\omega_{N_*}}\left(\sum_{N>N_*}\omega_N^2\langle N\rangle^{2s}\big(\Vert u_n(t)\Vert_{L^2_x}^2+\Vert u(t)\Vert_{L^2_x}^2\big)\right)^{1/2}.
\end{align*}
Therefore, since $\omega_N\nearrow+\infty$ as $N\to+\infty$, we conclude the proof of the claim. On the other hand, from the strong convergence in $C([0,T],H^{s-1}(\R))$ deduced in the previous subsection, we infer the existence of $n_*$ such that, for all $n\geq n_*$ and all $t\in(0,T)$, we have \[
\Vert P_{\leq N_*}u_n(t)-P_{\leq N_*}u(t)\Vert_{H^s_x}\leq 2N_*\Vert P_{\leq N_*}u_n(t)-P_{\leq N_*}u(t)\Vert_{H^{s-1}_x}<\tfrac{1}{2}\varepsilon.
\]
Gathering the last two estimates, we conclude the proof of the continuity of the flow map, and hence, the proof of Theorem \ref{MT1}. \qed

\subsection{Proof of Theorem \ref{MT_Zhidkov}}

The proof of Theorem \ref{MT_Zhidkov} is a direct consequence of Theorem \ref{MT1} along with the following lemma, proved in \cite{Ga,IoLiSc}.
\begin{lem}\label{deco}
Let $\Phi\in \mathcal{Z}^s(\R)$ for $s> \tfrac{1}{2}$. Then, there exists $\Psi\in C_b^\infty$ and $v\in H^s$ such that \[
\Phi=u+\Psi, \quad \hbox{with} \quad \Psi'\in H^\infty(\R).
\]
Moreover, the maps $\Phi\mapsto \Psi$ and $\Phi\mapsto u$ can be defined as linear maps such that for every $\tilde{s}>\tfrac{1}{2}$ the following holds: The map $\Phi\mapsto \Psi$ is continuous from $\mathcal{Z}^s$ into $\mathcal{Z}^{\tilde{s}}$, whereas the map $\Phi\mapsto v$ is continuous from $\mathcal{Z}^s$ into $H^s$.
\end{lem}

\begin{proof} As we already mentioned, the proof follows almost the same lines as the corresponding versions in \cite{Ga,IoLiSc}. However, due to the hypothesis $s>1/2$ we need a slight modification of the argument. In fact, we shall actually explicitly define the function $\Psi$. Indeed, let us consider $$
\Psi(x):=(k*\Phi)(x) \quad \hbox{where}\quad k(x):=\tfrac{1}{(4\pi)^{1/2}}e^{-x^2/4}.
$$
Then, it immediately follows that $\Psi\in C^\infty_b(\R)$ and that $\Psi'\in H^\infty(\R)$. Therefore, $\Phi-\Psi\in L^\infty \subset\mathcal{S}'$. Now, by direct computations we obtain \[
\mathcal{F}\big(\Phi-\Psi)=(1-e^{-\xi^2})\widehat{\Phi}(\xi)=(1+\vert \xi\vert^2)^{1/4}\left(\tfrac{1-e^{-\xi^2}}{\xi}\right)\times\xi(1+\vert\xi\vert^2)^{-1/4}\widehat{\Phi}(\xi)=:\mathrm{I}\times\mathrm{II}.
\]
Then, it is enough to notice that $\mathrm{I}\in L^\infty$ and that, due to hypothesis $\Phi\in\mathcal{Z}^s(\R)$, we get $\mathrm{II}\in L^2$. It is not difficult to see that from the above computation we have $u:=\Phi-\Psi\in H^s$. Finally, to obtain the continuity part of the statement, it is enough to notice that \begin{align*}
\Vert \Psi'\Vert_{H^{\tilde{s}-1}}^2 \leq \Vert \Phi'\Vert_{H^{s-1}}^2\sup_{\xi\in\R}\left((1+\xi^2)^{\tilde{s}-s}e^{-2\xi^2}\right)\lesssim \Vert \Phi'\Vert_{H^{s-1}}^2.
\end{align*}
On the other hand, by straightforward computations from the definition of $\Psi$, we also obtain that $\Vert \Psi\Vert_{L^\infty}\leq \Vert\Phi\Vert_{L^\infty}$, what give us the continuity of the map $\Phi\mapsto \Psi$ from $\mathcal{Z}^s$ into $\mathcal{Z}^{\tilde{s}}$. Furthermore, proceeding similarly as above we also get that \begin{align*}
\Vert u\Vert_{H^s}^2\leq \Vert \Phi'\Vert_{H^{s-1}}^2 \sup_{\xi\in\R}\left(\dfrac{(1+\xi^2)(1-e^{-\xi^2})^2}{\xi^2}\right)\lesssim \Vert \Phi'\Vert_{H^{s-1}}^2,
\end{align*}
what give us the continuity of $\Phi\mapsto u$ from $\mathcal{Z}^s(\R)$ into $H^s(\R)$. The proof is complete.
\end{proof}

Therefore, by using the above lemma, we can decompose the initial data $v(0,\cdot)$ associated with the IVP \eqref{ggkdv_v} into two functions $u_0\in H^s(\R)$ and $\Psi\in\mathcal{Z}^\infty(\R)$. Hence, it is enough to write \eqref{ggkdv_v} in terms of the Cauchy problem \eqref{gkdv_v} with $\Psi=\Psi(x)$ being a time-independent function belonging to $\Psi\in\mathcal{Z}^\infty(\R)$. Notice that $\Psi$ satisfies all the hypotheses in \eqref{hyp_psi}. Thus, Theorem \ref{MT_Zhidkov} follows by using Theorem \ref{MT1} with the above decomposition.

\medskip

\section{Local well-posedness in $H^{3/2^+}(\R)$}\label{section_MT_smooth}

This section is devoted to show the following result, that gives us the LWP for smooth initial data.

\begin{thm}[LWP for smooth data]\label{MT2}
The Cauchy problem associated with \eqref{gkdv_v} is locally well-posed in $H^s(\R)$ for $s>3/2$, with minimal existence time \[
T=T\big(\Vert u_0\Vert_{H^s_x},\Vert \Psi\Vert_{L^\infty_tW^{s+1^+,\infty}_x},|||\Psi|||_s\big)>0.
\]
\end{thm}

To establish the existence and uniqueness of smooth solutions to the IVP \eqref{gkdv_v}, we use the parabolic regularization method, that is, we
consider solutions to the following equation
\begin{align}\label{parabolic_reg}
\partial_t u+\partial_x^3u-\mu\partial_x^2u=-(\partial_t\Psi+\partial_x^3\Psi+\partial_xf(\Psi))-\partial_x(f(u+\Psi)-f(\Psi)),
\end{align}
for $\mu>0$. Roughly, the idea is to start by showing LWP of the above equation, and then take the limit $\mu\to0$. Since these ideas are (nowadays) fairly standard and have been used multiple times in many different contexts, we shall be brief and only sketch their main estimates. We refer to \cite{AbBoFeSa} and \cite{IoLiSc} for further details.

\medskip

Before going further let us recall some preliminary lemmas needed to prove Theorem \ref{MT_gwp}. The following lemma give us the main estimate to prove the LWP of \eqref{parabolic_reg} (see \cite{Io}). 
\begin{lem}\label{lema_group}
Let $\mu>0$ fixed. Let $W_\mu(t)$ to be the free group associated with the linear part of \eqref{parabolic_reg}, that is \[
W_\mu(t) :=\exp((\mu\partial_x^2-\partial_x^3)t).
\]
Then, for all $s\in\R$, $r\geq0$ and all $f\in H^s(\R)$ the following holds \[
\Vert W_\mu(t)f\Vert_{H^{s+r}}\lesssim_r\Big(1+\dfrac{1}{(2\mu t)^r}\Big)^{1/2}\Vert f\Vert_{H^s}.
\]
\end{lem}

As a direct consequence of the previous property, we have the following result.
\begin{lem}
Let $\mu>0$ fixed. Consider $u_0\in H^s(\R)$ with $s>3/2$. Then, there exists $T=T(\Vert u_{0}\Vert_{H^s},\mu)>0$ and a unique solution $u_\mu(t)$ to equation \eqref{parabolic_reg} satisfying \[
u_\mu\in C([0,T],H^s(\R))\cap C((0,T],H^\infty(\R)).
\] 
\end{lem}
The previous lemma can be proven by writing $u_\mu(t)$  in its equivalent Duhamel form, and then proceeding by standard fixed point arguments, using Lemma \ref{lema_group}. We omit its proof.

\medskip

In the sequel we shall need the following lemma that combines commutator estimates with Sobolev inequalities.
\begin{lem}[\cite{Ka}]\label{Kato}
Let $s>3/2$ and $r>1$. Then, for all $f,g\in \mathcal{S}(\R)$ the following holds \[
\vert\langle fg_x,g\rangle_{H^s}\vert\lesssim \Vert f_x\Vert_{H^{r-1}}\Vert g\Vert_{H^s}^2+\Vert f_x\Vert_{H^{s-1}}\Vert g\Vert_{H^s}\Vert g\Vert_{H^r},
\] 
where the implicit constant only depends on $s$ and $r$.
\end{lem}

\medskip

The next step is to show that the previously found solution $u_\mu(t)$ can be extended to an interval of existence independent of $\mu>0$. 

\begin{lem}
Let $\mu>0$ fixed. Let $u_\mu\in C([0,T],H^s(\R))$ be the solution to equation \eqref{parabolic_reg} given by the previous lemma, with initial data $u_{0}\in H^s(\R)$ with $s>3/2$. Then, $u_\mu(t)$ can be extended to an interval $T'=T'(\Vert u_0\Vert_{H^s})>0$ independent of $\mu$. Moreover, there exists a continuous function $\rho:[0,T']\to\R$ such that \[
\Vert u_\mu(t)\Vert_{H^s_x}^2\leq \rho(t)\quad \hbox{ with }\quad \rho(0)=\Vert u_{0}\Vert_{H^s}^2.
\]
\end{lem}

\begin{proof}
In fact, directly taking the derivative of the $H^s$-norm, using \eqref{parabolic_reg}, after suitable integration by parts, we obtain
\begin{align}\label{l2_commutator}
\dfrac{d}{dt}\Vert u_\mu(t)\Vert_{H^s_x}^2&\leq -2\langle u_\mu, \partial_x(f(u+\Psi)-f(\Psi))\rangle_{H^s_x}-2\langle u_\mu,\partial_t\Psi+\partial_x^3\Psi+\partial_xf(\Psi)\rangle_{H^s_x}
\end{align}
For the latter term above, from Cauchy-Schwarz we can see that
\[
\vert \langle u_\mu,\partial_t\Psi+\partial_x^3\Psi+\partial_xf(\Psi)\rangle_{H^s_x}\vert\leq \Vert u_\mu(t)\Vert_{H^s_x}^2+\Vert \partial_t\Psi+\partial_x^3\Psi+\partial_xf(\Psi)\Vert_{L^\infty_tH^s_x}^2.
\]
On the other hand, to estimate the first term in the right-hand side of \eqref{l2_commutator} we write
\begin{align*}
\partial_x\big(f(u_\mu+\Psi)-f(\Psi)\big)&=u_{\mu,x}\sum_{k=1}^{\infty}\sum_{m=0}^{k-1}a_k(k-m)\binom{k}{m}u_\mu^{k-m-1}\Psi^m
\\ & \quad +\Psi_x\sum_{k=1}^{\infty}\sum_{m=1}^{k-1}a_km\binom{k}{m}u_\mu^{k-m}\Psi^{m-1}=:\mathrm{I}+\mathrm{II}.
\end{align*}
Then, by classical Sobolev estimates for products as well as Lemma \ref{Kato} we infer that there exists a constant $c_1>0$ such that
\begin{align*}
\vert\langle u_\mu,\mathrm{I}\rangle\vert_{H^s_x}\lesssim \Vert u_\mu\Vert_{H^s_x}^2\sum_{k=1}^{\infty}k\vert a_k\vert c_1^k\big(\Vert u_\mu\Vert_{H^s_x}+\Vert \Psi\Vert_{L^\infty_tW^{s^+,\infty}_x}\big)^{k-1}.
\end{align*}
In a similar fashion, there exists another constant $c_2>0$ such that  \begin{align*}
\vert\langle u_\mu,\mathrm{II}\rangle_{H^s_x}\vert \lesssim \Vert u_\mu\Vert_{H^s_x}^2\sum_{k=1}^\infty k^2\vert a_k\vert c_2^k\big(\Vert u_\mu\Vert_{H^s_x}+\Vert \Psi\Vert_{L^\infty_tW^{s+1^+,\infty}_x}\big)^{k-1}.
\end{align*}
Therefore, gathering the above estimates, recalling that $\Psi$ is given, we infer that, there exists a smooth function $\mathcal{F}_*:\R\to\R_+$ such that \begin{align*}
\dfrac{d}{dt}\Vert u_\mu(t)\Vert_{H^s_x}^2&\lesssim \Vert \partial_t\Psi+\partial_x^3\Psi+\partial_xf(\Psi)\Vert_{L^\infty_tH^s_x}^2+\Vert u_\mu(t)\Vert_{H^s_x}^2\mathcal{F}_*\big(\Vert u_\mu(t)\Vert_{H^s_x}\big).
\end{align*}
Then, denoting by $C_\Psi$ the first term in the right-hand side above, it is enough to consider $\rho(t)$ to be the solution of the equation \[
\dot\rho(t)=C_\Psi+\rho(t)\mathcal{F}_*(\rho^{1/2}(t)), \qquad \rho(0)=\Vert u_0\Vert_{H^s}^2.
\]
Notice that the solution exists thanks to Cauchy-Lipschitz Theorem. Taking $T_*>0$ to be the maximal existence time of $\rho(t)$,  we conclude $\Vert u_\mu(t)\Vert_{H^s_x}^2\leq \rho(t)$ for all $t\leq T_*$.
\end{proof}

\smallskip

\subsection{Proof of Theorem \ref{MT2}}

By using the latter lemma we can now take a sequence of initial data $u_{0,\mu}\in H^s(\R)$ strongly converging to some $u_0$ in $H^s(\R)$. Then, by the uniform (in $\mu$) bound we infer that, up to a subsequence, we can pass to the limit in the sequence of solutions $u_\mu(t)$, which converge in the weak-$\star$ topology of $L^\infty((0,T),H^s(\R))$ to some limit object $u(t)$. It is not difficult to see, reasoning similarly as in the previous section, that $u(t)$ solves the equation in the distributional sense and the map $[0,T]\ni t\mapsto u(t)\in H^s(\R)$ is weakly continuous. Let us now consider the uniqueness of the solution. With this aim, let us consider $w:=u-v$, with $u$ and $v$ solutions of the equation. Recall that then $w$ solves
\[
\partial_tw+\partial_x\big(\partial_x^2w+f(u+\Psi)-f(v+\Psi)\big)=0.
\]
Then, taking the $L^2$-scalar product of the above equation with against $w$, we obtain 
\begin{align*}
\dfrac{d}{dt}\Vert w\Vert_{L^2_x}^2&=-\langle w,\partial_x(f(u+\Psi)-f(v+\Psi))\rangle_{L^2_x}
\\ & \lesssim \Vert w\Vert_{L^2_x}^2\sum_{k=1}^\infty c^k\vert a_k\vert \big(\Vert u\Vert_{H^1_x}+\Vert v\Vert_{H^1_x}+\Vert\Psi\Vert_{L^\infty_tW^{1^+,\infty}_x}\big)^{k-1}.
\end{align*}
Therefore, a direct application of Gr\"onwall inequality, recalling that $\Vert u(t)\Vert_{H^s_x}^2+\Vert v(t)\Vert_{H^s_x}^2\leq 2\rho(t)$, implies the uniqueness.

\medskip

The strong continuity of the solution with values in $H^s(\R)$, as well as the continuity of the flow-map can be proven by classical Bona-Smith arguments. We omit this proof. \qed

\medskip

\section{Proof of Theorem \ref{MT_gwp}}\label{section_MT_gwp}

In this section we seek to prove the global well-posedness theorem \ref{MT_gwp}. We recall that in this case we assume that \begin{align}\label{fpp_hyp}
\vert f''(x)\vert\lesssim 1, \quad \forall x\in\R,
\end{align}
what shall allow us to use Gronwall inequality. We emphasize once again that, due to the presence of $\Psi(t,x)$, equation \eqref{gkdv_v} has no evident conservation laws. Our first lemma states that the $L^2$-norm of the solution grows at most exponentially fast in time.
\begin{lem}\label{lemma_growth_l2}
Let $u(t)\in C([0,T],H^1(\R))$ be a solution to equation \eqref{gkdv_v} emanating from an initial data $u_0\in H^1(\R)$. Then, for all $t\in[0,T]$ we have
\begin{align}\label{L2_gronwall_general}
\Vert u(t)\Vert_{L^2_x}^2\leq C_{u_0,\Psi}\exp(C_{\Psi}t),
\end{align}
where $C_{\Psi}>0$ is a positive constants that only depends on $\Psi$, while $C_{u_0,\Psi}>0$ depends on $\Psi$ and $u_0$.
\end{lem}

\begin{proof}
In fact, multiplying equation \eqref{gkdv_v} by $u(t)$ and then integrating in space we obtain
\begin{align*}
\dfrac{1}{2}\dfrac{d}{dt}\int_\R u^2(t,x)dx&=-\int u\partial_x\big(f(u+\Psi)-f(\Psi)\big) -\int u\big(\partial_t\Psi+\partial_x^3\Psi+\partial_xf(\Psi)\big)
\\ &=: \mathrm{I}+\mathrm{II}.
\end{align*}
Notice that, thanks to our hypotheses on $\Psi$, we can immediately bound $\mathrm{II}$ by using Young inequality for products, from where we get
\[
\vert \mathrm{II}\vert\leq\Vert u(t)\Vert_{L^2_x}^2+\Vert \partial_t\Psi+\partial_x^3\Psi+\partial_xf(\Psi)\Vert_{L^\infty_tL^2_x}^2.
\]
Now, the estimate for $\mathrm{I}$ is more delicate since we require to integrate by parts. Also, we must be careful while splitting the integral into several integrals since there might be terms that do not integrate (due to $\Psi$). Hence, in this case we can proceed as follows
\begin{align*}
\vert \mathrm{I}\vert&=\lim_{R_1,R_2\to+\infty}\Big\vert\int_{-R_2}^{R_1} u_x\big(f(u+\Psi)-f(\Psi)\big)\Big\vert
\\ &=\lim_{R_1,R_2\to+\infty}\Big\vert\int_{-R_2}^{R_1} (u_x+\Psi_x)f(u+\Psi)-\Psi_xf(u+\Psi)-\Psi_xf(\Psi)
\\ & \qquad  +\Psi_xf(\Psi)-u_xf(\Psi)-\Psi_xuf'(\Psi)+\Psi_xuf'(\Psi)\Big\vert
\\ & \leq \limsup_{R_1,R_2\to+\infty }\Big\vert\int_{-R_2}^{R_1}(u_x+\Psi_x)f(u+\Psi)-\Psi_xf(\Psi)\Big\vert
\\ & \qquad +\limsup_{R_1,R_2\to+\infty }\Big\vert\int_{-R_2}^{R_1}u_xf(\Psi)+\Psi_xuf'(\Psi)\Big\vert
\\ & \qquad +\limsup_{R_1,R_2\to+\infty }\Big\vert\int_{-R_2}^{R_1}\Psi_xf(u+\Psi)-\Psi_xf(\Psi)-\Psi_xuf'(\Psi)\Big\vert
\\ &=: \mathrm{I}_1+\mathrm{I}_2+\mathrm{I}_3. 
\end{align*}
Now, for $\mathrm{I}_1$ notice that we can write the integrand as a full derivative, and hence we have \begin{align*}
\mathrm{I}_1&=\limsup_{R_1,R_2\to+\infty }\Big\vert\int_{-R_2}^{R_1}\partial_x\Big(F(u+\Psi)-F(\Psi)\Big)\Big\vert
\\ & \leq \limsup_{R_1\to+\infty}\big\vert F(u+\Psi)-F(\Psi)\big\vert(R_1)+\limsup_{R_2\to+\infty}\big\vert F(u+\Psi)-F(\Psi)\big\vert(-R_2)=0, 
\end{align*}
where in the last equality we have used the fact that $F$ is smooth and that $u(t)\in H^1(\R)$, so in particular $u(t)\to 0$ as $x\to\pm \infty$ for all $t\in[0,T]$. On the other hand, for $\mathrm{I}_2$ we integrate by parts, from where we obtain \begin{align*}
\mathrm{I}_2\leq\limsup_{R_1\to+\infty}\vert uf(\Psi)\vert(R_1)+\limsup_{R_2\to+\infty}\vert uf(\Psi)\vert(-R_2)=0,
\end{align*}
since $u(t)\in H^1(\R)$, $\Psi\in L^\infty(\R^2)$ and $f$ is smooth. Then, gathering all the above estimates, and then using H\"older inequality along with hypothesis \eqref{fpp_hyp}, we deduce that \[
\vert \mathrm{I}\vert\lesssim \left\vert\int_\R \Psi_x\big(f(u+\Psi)-f(\Psi)-uf'(\Psi)\big)\right\vert\lesssim\Vert \Psi_x\Vert_{L^\infty_{t,x}}\Vert u(t)\Vert_{L^2_x}^2.
\]
Therefore, Gronwall inequality provides \eqref{L2_gronwall_general}. The proof is complete.
\end{proof}

Now, in order to control the $H^1$-norm, we consider the following modified energy functional
\[
\mathcal{E}\big(u(t)\big):=\dfrac{1}{2}\int_\R u_x^2(t,x)dx-\int_\R \Big(F\big(u(t,x)+\Psi(t,x)\big)-F\big(\Psi(t,x)\big)-u(t,x)f\big(\Psi(t,x)\big)\Big)dx.
\]
It is worth to notice that the previous functional is well defined for all times $t\in[0,T]$. The following lemma give us the desired control on the growth of the $H^1$-norm of the solution $u(t)$, and hence, it finishes the proof of Theorem \ref{MT_gwp}.

\begin{lem}
Let $u(t)\in C([0,T],H^1(\R))$ be a solution to equation \eqref{gkdv_v} emanating from an initial data $u_0\in H^1(\R)$. Then, for all $t\in[0,T]$ we have
\begin{align*}
\Vert u(t)\Vert_{H^1_x}\lesssim C_{u_0,\Psi}^*\exp(C_{\Psi}^*t).
\end{align*}
where $C_{\Psi}^*>0$ is a positive constants that only depends on $\Psi$, while $C_{u_0,\Psi}^*>0$ depends on $\Psi$ and $u_0$.
\end{lem}

\begin{proof}
First of all, by using the continuity of the flow with respect to the initial data, given by Theorem \ref{MT1}, we can assume $u(t)$ is sufficiently smooth so that all the following computations hold. Now, let us begin by explicitly computing the time derivative of the energy functional. In fact, by using equation \eqref{gkdv_v}, after suitable integration by parts we obtain 
\begin{align}\label{derivative_energy_gwp}
\dfrac{d}{dt}\mathcal{E}&=-\int u_{xx}u_t-\int u_t\big( f(u+\Psi)-f(\Psi)\big) -\int\Psi_t\big( f(u+\Psi)-f(\Psi)-uf'(\Psi)\big)\nonumber
\\ & =\int u_{xx}\partial_x\big(f(u+\Psi)-f(\Psi)\big)+\int u\partial_x^2\big(\Psi_t+\partial_x^3\Psi+\partial_xf(\Psi)\big)\nonumber
\\ & \quad +\int u_{xxx}\big( f(u+\Psi)-f(\Psi)\big)
+\int \big( f(u+\Psi)-f(\Psi)\big)\partial_x\big( f(u+\Psi)-f(\Psi)\big)\nonumber
\\ & \quad +\int \big( f(u+\Psi)-f(\Psi)\big)\big(\Psi_t+\partial_x^3\Psi+\partial_xf(\Psi)\big) -\int\Psi_t\big( f(u+\Psi)-f(\Psi)-uf'(\Psi)\big)\nonumber
\\ & = \int u\partial_x^2\big(\Psi_t+\partial_x^3\Psi+\partial_xf(\Psi)\big) +\int \big( f(u+\Psi)-f(\Psi)\big)\big(\Psi_t+\partial_x^3\Psi+\partial_xf(\Psi)\big)\nonumber
\\ & \quad -\int\Psi_t\big( f(u+\Psi)-f(\Psi)-uf'(\Psi)\big)\nonumber
\\ & \lesssim \big(1+\Vert \Psi_t\Vert_{L^\infty_{t,x}}+\Vert \Psi\Vert_{L^\infty_{t,x}}^2+\Vert \partial_t\Psi+\partial_x^3\Psi+ \partial_xf(\Psi)\Vert_{L^\infty_{t,x}}\big)\Vert u(t)\Vert_{L^2_x}^2
\\ & \quad +\Vert \Psi_t+\partial_x^3\Psi+\partial_xf(\Psi)\Vert_{L^\infty_tH^2_x}^2. \nonumber
\end{align}
On the other hand, by using Gagliardo-Nirenberg interpolation inequality, and then applying Young inequality for products, we have \begin{align*}
\left\vert\int_\R u^3(t,x)dx\right\vert&\leq C\Vert u_x\Vert_{L^2_x}^{1/2}\Vert u(t)\Vert_{L^2_x}^{5/2} 
\\ & \leq \dfrac{\varepsilon^4}{4}\Vert u_x(t)\Vert_{L^2_x}^2+\dfrac{3}{4\varepsilon^{4/3}}\Vert u(t)\Vert_{L^2_x}^{10/3}.
\end{align*}
Thus, by using the above inequality together with our current hypothesis on $f(x)$, we deduce \begin{align*}
\left\vert \int \big(F(u+\Psi)-F(\Psi)-uf(\Psi)\big)\right\vert &\lesssim \Vert \Psi\Vert_{L^\infty_{t,x}}\Vert u(t)\Vert_{L^2_x}^2+\Vert u(t)\Vert_{L^3_x}^3
\\ &\lesssim \Vert \Psi\Vert_{L^\infty_{t,x}}\Vert u(t)\Vert_{L^2_x}^2+\dfrac{\varepsilon^4}{4}\Vert u_x(t)\Vert_{L^2_x}^2+\dfrac{3}{4\varepsilon^{4/3}}+\Vert u(t)\Vert_{L^2_x}^{10/3}.
\end{align*}
Therefore, integrating \eqref{derivative_energy_gwp} on $[0,T]$, and then plugging the latter inequality in the resulting right-hand side, together with the conclusion of Lemma \ref{lemma_growth_l2}, denoting by $C_{\varepsilon}:=1-\tfrac{1}{4}\varepsilon^4$, we infer that \begin{align*}
C_\varepsilon\int_\R u_x^2(t,x)dx&\lesssim \int u_{0,x}^2-\int \big(F(u_0+\Psi_0)-F(\Psi_0)-u_0f(\Psi_0)\big)
\\ & \quad +C_{u_0,\Psi}\big(1+\Vert \Psi_t\Vert_{L^\infty_{t,x}}+\Vert \Psi\Vert_{L^\infty_{t,x}}^2+\Vert \partial_t\Psi+\partial_x^3\Psi+ \partial_xf(\Psi)\Vert_{L^\infty_{t,x}}\big)e^{10C_\Psi t/3},
\end{align*}
where $C_{u_0,\Psi}$ and $C_{\Psi}$ are the constants founded in the previous lemma. Then, choosing $\varepsilon>0$ small, we conclude the proof of the lemma.
\end{proof}

\bigskip

\textbf{Acknowledgements:} The author is very grateful to Professor Luc Molinet for encouraging him to solve this problem and for several remarkably helpful comments and conversations.

\end{document}